%% file: arxiv.tex
\documentclass[hidelinks, 12pt]{article}

\input{LIBRARY/packages}

\input{LIBRARY/config}

\input{LIBRARY/environments}

\input{LIBRARY/custom-commands}

\title{Local categories: a new framework for partiality}
\author{\uppercase{Marcello Lanfranchi, Jean-Simon Pacaud Lemay}}
\affil{\normalsize\textit{Macquarie University, School of Mathematical and Physical Sciences}}
\date{}

\begin{document}

\maketitle

\input{CODE/0-abstract}

\tableofcontents

\input{CODE/1-introduction}

\input{CODE/2-restriction}

\input{CODE/3-local}

\input{CODE/4-equivalence}

\input{CODE/5-concepts}

\input{CODE/6-cartesian}

\input{CODE/7-split}

\input{CODE/8-inverse}

\input{CODE/9-join}

\input{CODE/10-partial}

\input{CODE/11-inclusion}

\input{CODE/12-fundamental-theorem}

\input{BIBLIOGRAPHY/bibliography}

\end{document}

%% file: LIBRARY/packages.tex
\usepackage[UKenglish]{babel}
\usepackage[utf8]{inputenc}
\usepackage[T1]{fontenc}
\usepackage{erewhon}
\usepackage{amsmath}
\usepackage{mathrsfs}
\usepackage{amsthm}
\usepackage{amsfonts}
\usepackage{amssymb}
\usepackage{mathtools}
\usepackage{dsfont} 
\usepackage{bussproofs}
\usepackage{mathrsfs}

\usepackage{quiver}
\usepackage{adjustbox}
\usepackage{graphicx}

\usepackage[dvipsnames]{xcolor}
\usepackage{hyperref}
\usepackage{fullpage}
\usepackage{microtype}
\usepackage{booktabs}
\usepackage{enumitem} 
\usepackage{comment}
\usepackage{authblk} 


%% file: LIBRARY/config.tex
\tikzset{node distance=2cm, auto}
\tikzcdset{row sep/normal=2.7em,column sep/normal=3.5em}

\setlist[description]{font=\normalfont\bfseries\:\!}

\renewcommand\labelenumi{(\alph{enumi})}
\renewcommand\theenumi\labelenumi


%% file: LIBRARY/environments.tex
\newtheorem{theorem}{Theorem}[section]
\newtheorem{proposition}[theorem]{Proposition}
\newtheorem{lemma}[theorem]{Lemma}
\newtheorem{corollary}[theorem]{Corollary}

\theoremstyle{definition}
\newtheorem{definition}[theorem]{Definition}
\newtheorem{example}[theorem]{Example}

\newtheorem{remark}[theorem]{Remark}

\newenvironment{contentX}[1][blank]{\noindent{\large\bfseries\scshape\MakeLowercase{#1}}.\thinspace\noindent}{}
\newenvironment{content}[1][blank]{\vskip .5cm\begin{contentX}[#1]}{\end{contentX}}
\newenvironment{acknowledgements}{\begin{content}[Acknowledgements]}{\end{content}}

\newenvironment{conventions}{\begin{content}[Conventions]}{\end{content}}
\newenvironment{outline}{\begin{content}[Outline]}{\end{content}}

\renewenvironment{abstract}{\begin{contentX}[Abstract]}{\end{contentX}}


%% file: LIBRARY/custom-commands.tex
\newcommand{\CategoryFont}[1]{\mathsf{\uppercase{#1}}}
\newcommand{\FunctorFont}[1]{\mathrm{#1}}

\newcommand{\Tag}[1]{\tag{$#1$}}

\renewcommand{\bar}[1]{\mkern 1.5mu\overline{\mkern-1.5mu#1\mkern-1.5mu}\mkern 1.5mu}

\newcommand{\1}{\mathds{1}}

\renewcommand{\=}{\mathrel{\mathop:}=}
\renewcommand{\epsilon}{\varepsilon}
\renewcommand{\o}{\circ}
\renewcommand{\b}{\bullet}
\newcommand{\<}{\langle}
\renewcommand{\>}{\rangle}
\newcommand{\op}{\mathsf{op}}

\renewcommand{\*}{\ast}
\newcommand{\nto}{\mathrel{\mkern3mu\vcenter{\hbox{$\scriptscriptstyle|\:$}}\mkern-12mu{\to}}}

\newcommand{\lax}{\mathsf{lax}}

\newcommand{\id}{\mathsf{id}}

\newcommand{\CC}{{\mathbf{K}}}
\newcommand{\DD}{{\mathbf{J}}}
\newcommand{\C}{\mathbb{C}}

\newcommand{\X}{\mathbb{X}}

\newcommand{\cRing}{\CategoryFont{cRing}}

\newcommand{\Set}{\CategoryFont{Set}}
\newcommand{\Top}{\CategoryFont{Top}}
\newcommand{\Cat}{\CategoryFont{Cat}}

\newcommand{\Cart}{\CategoryFont{Cart}}

\newcommand{\ParFnc}{\CategoryFont{Par}}
\newcommand{\RCat}{\CategoryFont{RCat}}
\newcommand{\cRCat}{\CategoryFont{cRCat}}
\newcommand{\sRCat}{\CategoryFont{sRCat}}
\newcommand{\jRCat}{\CategoryFont{jRCat}}
\newcommand{\iRCat}{\CategoryFont{invRCat}}
\newcommand{\Tot}{\FunctorFont{Tot}}
\newcommand{\Split}{\FunctorFont{Split}}

\newcommand{\M}{\mathscr{M}}

\renewcommand{\L}{\mathbf{L}}
\newcommand{\R}{\mathbf{R}}
\newcommand{\Las}{\mathsf{L}}
\newcommand{\ParSet}{\CategoryFont{ParSet}}
\newcommand{\ParTop}{\CategoryFont{ParTop}}
\newcommand{\IdemRing}{\CategoryFont{IdemRing}}

\newcommand{\LCat}{\CategoryFont{LCat}}
\newcommand{\cLCat}{\CategoryFont{cLCat}}
\newcommand{\sLCat}{\CategoryFont{sLCat}}
\newcommand{\jLCat}{\CategoryFont{jLCat}}
\newcommand{\iLCat}{\CategoryFont{invLCat}}
\newcommand{\PCat}{\CategoryFont{PCat}}
\newcommand{\ICat}{\CategoryFont{NCat}}
\newcommand{\iBICat}{\CategoryFont{invBNCat}}
\newcommand{\BPCat}{\CategoryFont{BPCat}}
\newcommand{\BICat}{\CategoryFont{BNCat}}

\def\whitelollipop#1{
\begin{tikzpicture}[scale=#1]
\draw (0,7.5pt) -- (0,1.5pt) (0,0) circle (1.5pt);
\end{tikzpicture}
}
\def\blacklollipop#1{
\begin{tikzpicture}[scale=#1]
\draw[fill] (0,0) -- (0,6pt) (0,7.5pt) circle (1.5pt);
\end{tikzpicture}
}
\DeclareMathOperator{\down}{\mathchoice{\whitelollipop{1}} 
{\whitelollipop{1}} 
{\whitelollipop{.75}}  
{\whitelollipop{5}}  
}
\DeclareMathOperator{\up}{\mathchoice{\blacklollipop{1}} 
{\blacklollipop{1}} 
{\blacklollipop{.75}}  
{\blacklollipop{.5}}  
}
\newcommand{\Incl}{\mathscr{N}}
\newcommand{\Famly}{\mathscr{F}}
\newcommand{\I}{\FunctorFont{N}}
\renewcommand{\P}{\FunctorFont{P}}

\newcommand{\dom}{\mathsf{dom}}


%% file: CODE/0-abstract.tex
\begin{abstract}\noindent
Restriction categories provide a categorical framework for partiality. In this paper, we introduce three new categorical theories for partiality: local categories, partial categories, and inclusion categories. The objects of a local category are partially accessible resources, and morphisms are processes between these resources. In a partial category, partiality is addressed via two operators, restriction and contraction, which control the domain of definition of a morphism. Finally, an inclusion category is a category equipped with a family of monics which axiomatize the inclusions between sets. The main result of this paper shows that restriction categories are $2$-equivalent to local categories, that partial categories are $2$-equivalent to inclusion categories, and that both restriction/local categories are $2$-equivalent to bounded partial/inclusion categories. Our result offers four equivalent ways to describe partiality: on morphisms, via restriction categories; on objects, with local categories; operationally, with partial categories; and via inclusions, with inclusion categories. We also translate several key concepts from restriction category theory to the local category context, which allows us to show that various special kinds of restriction categories, such as inverse categories, are $2$-equivalent to their analogous kind of local categories. In particular, the equivalence between inverse (restriction) categories and inverse local categories is a generalization of the celebrated Ehresmann-Schein-Nambooripad theorem for inverse semigroups.
\end{abstract}


%% file: CODE/1-introduction.tex
\section{Introduction}
\label{section:introduction}
Partiality is a fundamental concept throught all of mathematics and computer science. Restriction categories provide a categorical framework for partiality. Restriction category theory is now a firmly established field with a rich literature, having found interesting connections and applications to various areas in mathematics and computer science. Restriction categories were originally introduced in~\cite{cockett:restrictionI}, inspired by the works of~\cite{di-paola:dominical-categories},~\cite{robinson:partial-morphisms},~\cite{jacobs:weaking-contraction}, and~\cite{mulry:partial-morphisms}. In a restriction category, every morphism $f\colon A\to B$ is associated with an endomorphism $\bar f\colon A\to A$, called the \textit{restriction idempotent} of $f$, containing the information about the domain of definition of $f$. For example, if $g\colon A\to C$ is another morphism, the composition $\bar fg$ of $\bar f$ followed by $g$ is interpreted as the \textit{restriction} of $g$ to the domain of definition of $f$. In particular, $\bar ff$ coincides with $f$ itself, and $\bar f\bar g=\bar g\bar f$ is interpreted as the intersection of the domain of definitions of $f$ and $g$. The archetypal example of a restriction category is the category $\ParFnc$ of sets and partial functions, whose restriction idempotents $\bar f\colon A\to A$ are the partial functions which send each $a$ to itself when $f(a)$ is defined and undefined otherwise. One of the advantages of restriction categories is that one can refer to $f\colon A\to B$ as a morphism from $A$ to $B$ despite $f$ being only partially defined on $A$. In particular, a theory developed for total functions of a restriction category is also defined for partially-defined functions, without the need to rethink the latter as total functions on a different domain. 

\par In this paper, we present a new categorical approach to encoding partiality: \textit{local categories}. This notion was inspired by Cruttwell's definition of a Generalized Cartesian Differential Catetory~\cite[Definition~3.1]{cruttwell:generalized-CDC}. The objects of a local category can be regarded as partially accessible resources, and morphisms as processes between these resources. In particular, a local category $\C$ is a category equipped with an assignment, for every object $M$, to a morphism $\eta_M\colon M\to\Las M$, where $\Las M$ is called the \textit{enlargement} of $M$, satisfying some basic natural axioms: (i) $\eta_{\Las M}$ is the identity on $\Las M$; (ii) $\eta_M$ is monic; (iii) the pullback $m\colon P\to N$ of $\eta_M$ along any morphism $f\colon N\to\Las M$ exists, $\Las P=\Las N$, and $m\colon P\to N$ composed with $\eta_N$ is equal to $\eta_P$. The archetypal example of a local category is the category $\ParSet$ whose objects are pairs $(U,A)$ of sets in which $U\subseteq A$ is a subset of $A$, and whose morphisms $f\colon(U,A)\to(V,B)$ are functions $f\colon U\to V$ between the subsets. The enlargement $\Las(U,A)$ of $(U,A)$ is $(A,A)$ and $\eta_{(U,A)}$ is the inclusion morphism of $U$ into $A$.

\par The main result of this paper shows that local categories are not an orthogonally different perspective on partiality but rather an equivalent alternative to restriction categories. Indeed, while restriction categories encode partiality via morphisms, local categories encode partiality via the objects. We make this precise in Theorem~\ref{theorem:2-equivalence-local-restriction} by constructing a $2$-equivalence between the $2$-category of restriction categories and the $2$-category of local categories. Exploiting this equivalence, we translate various concepts of restriction categories into concepts for local categories, with the table below summarizing our findings. 
\begin{center}
\adjustbox{width=\linewidth}{
\begin{tabular}{c|c|c|c}
\textbf{Restriction categories}     &\textbf{Local categories}         &\textbf{Definition}        &\textbf{Theorem}\\
\toprule
Restriction functors                        &Local functors                             &\ref{definition:local-functors}&\ref{theorem:2-equivalence-local-restriction}\\
Restriction natural transformations         &Local natural transformations              &\ref{definition:total-local-natural-transformation}&\ref{theorem:2-equivalence-local-restriction}\\
Total natural transformations               &Total natural transformations              &\ref{definition:total-local-natural-transformation}&\ref{theorem:2-equivalence-local-restriction}\\
Total morphisms                                  &Total objects                              &\ref{definition:total-objects}&\\
Compatibility of morphisms $f\smile g$ &Compatibility of objects $A\smile B$  &\ref{definition:total-compatible}&\\
Partial order on morpisms $f\leq g$         &Partial order on objects $A\leq B$         &\ref{definition:partial-order-in-local-categories}&\ref{proposition:partial-order-on-local-categories}\\
Restricted monics                              &Monics                                     &\ref{definition:restricted-monic}&\ref{lemma:partial-monic}~\ref{lemma:partial-monic-2}\\
Cartesian restriction category              &Cartesian local category                   &\ref{definition:cartesian-local-category}&\ref{theorem:cartesian-local-categories}\\
Split restriction category                  &Split local category                       &\ref{definition:split-local-category}&\ref{theorem:2-equivalence-local-restriction-split}\\
Inverse (restriction) category              &Inverse local category                     &\ref{definition:inverse-local-category}&\ref{theorem:inverse-local-categories}\\
Join restriction category                   &Join local category                        &\ref{definition:join-local-category}&\ref{theorem:join-local-categories}
\end{tabular}
}
\end{center}

Moreover, we also introduce two additional approaches to partiality: \textit{partial categories} and \textit{inclusion categories}. In these settings, neither the objects nor the morphisms carry an intrinsic notion of partiality. Instead, partiality is treated extrinsically. In particular, a partial category consists of a category equipped with a partial order $\leq$ on objects and two operators: \textit{restriction} and \textit{contraction}. One can regard an object $U\leq A$ of a partial category as an object \textit{included} in $A$. Thus, given a $U\leq A$, the restriction operator \textit{restricts} a morphism $f\colon A\to B$ to a morphism $f\down U\colon U\to B$, which is the restriction of $f$ to the elements of $U$. On the other hand, given a $V\leq B$, the contraction operator \textit{contracts} a morphism $f\colon A\to B$ to a morphism $f\up V\colon A\up_fV\to V$ defined on an object $A\up_fV\leq A$, which can be regarded as the pre-image of $V$ along $f$. After introducing a technical assumption, boundedness, we prove the main theorem (Theorem~\ref{theorem:fundamental}) of the paper: restriction categories, local categories, bounded partial categories, and bounded inclusion categories are all $2$-equivalent to each other. This result generalizes the fundamental theorem of restriction categories, which establishes that \textit{split} restriction categories are $2$-equivalent to $\M$-categories, categories equipped with a system of monics, stable under composition and pullbacks. Theorem~\ref{theorem:fundamental} shows that (non-necessarily split) restriction categories are $2$-equivalent to \textit{bounded inclusion categories}, which are categories equipped with a system of monics which satisfy slightly different axioms from the system of monics of an $\M$-category and that generalize inclusions of sets.

\begin{outline}
In Section~\ref{section:restriction-categories}, we briefly recall the definition of a restriction category and give some examples. In Section~\ref{section:local-categories}, we introduce and study local categories. In Section~\ref{section:equivalence-restriction-local}, we prove that restriction categories are $2$-equivalent to local categories and in Section~\ref{section:concepts-local-categories}, we harness this equivalence to introduce some concepts of local categories. In Sections~\ref{section:cartesian},~\ref{section:split},~\ref{section:inverse}, and~\ref{section:join} we introduce Cartesian, split, inverse, and join local categories, respectively, and prove that they are $2$-equivalent to their restriction analogues. In Section~\ref{section:partial-categories}, we introduce partial categories and give some examples. In Section~\ref{section:inclusion-systems}, we introduce inclusion categories, give some examples, and show that partial categories are $2$-equivalent to inclusion categories. Finally, in Section~\ref{section:fundamental-theorem}, we introduce boundedness and prove that restriction categories, local categories, bounded partial categories, and bounded inclusion categories are all $2$-equivalent to each other.
\end{outline}

\begin{conventions}
We assume the reader is familiar with the basics of category theory. We denote an arbitrary category by $\mathbb{X}$, objects using capital letters such $A$, $B$, $C$, etc., morphisms as arrows $A \to B$ denoted using lowercase letters such as $f$, $g$, $h$, etc., and we denote identity morphisms as $\id_A: A \to A$. In this paper, we use diagrammatic notation for composition, that is, we write $fg\colon A\to C$ to denote the composition of $f\colon A\to B$ and $g\colon B\to C$ of first doing $f$ then $g$.
\end{conventions}

\begin{acknowledgements}
We would like to thank Tim Stokes who suggested a missing axiom (axiom \textbf{[I.6]}) in the definition of an inclusion category, Steve Lack for suggesting a remark~\ref{remark:boundedness-vs-equivalences-of-categories}, and Richard Garner for suggesting to compare our work with the ESN theorem, as discussed in Section~\ref{section:fundamental-theorem}. We would also like to thank Robin Cockett and Geoff Cruttwell for useful discussions and to the anonymous referees for their many useful suggestions. This material is based upon work supported by the AFOSR under award number FA9550-24-1-0008.
\end{acknowledgements}


%% file: CODE/2-restriction.tex
\section{Restriction categories: partiality on morphisms}
\label{section:restriction-categories}

In this background section, we review the basics of restriction categories. For an in-depth introduction to restriction categories, we invite the reader to see~\cite{cockett:restrictionI}.

\begin{definition}~\cite[Section~2.1.1]{cockett:restrictionI}
\label{definition:restriction-category}
A \textbf{restriction structure} on a category $\X$ consists of an assignment of each morphism $f\colon A\to B$ of $\X$ to an endomorphism $\bar f\colon A\to A$, called the \textbf{restriction idempotent} of $f$, satisfying the following axioms: 
\begin{description}
\item[R.1] For each $f\colon A\to B$, $\bar ff=f$;

\item[R.2] For each $f\colon A\to B$ and $g\colon A\to C$, $\bar f\bar g=\bar g\bar f$;

\item[R.3] For each $f\colon A\to B$ and $g\colon A\to C$, $\bar g\bar f=\bar{\bar gf}$;

\item[R.4] For each $f\colon A\to B$, $g\colon B\to C$, $f\bar g=\bar{fg}f$.
\end{description}
A \textbf{restriction category} is a category equipped with a restriction structure.
\end{definition}

As we will see below, the restriction idempotent $\bar f$ is indeed an idempotent in the usual sense, that is, $\bar f\bar f = \bar f$. So to help better understand the axioms \textbf{[R.1]}-\textbf{[R.4]}, let us assume that the restriction idempotent $\bar f$ splits, that is, there exists an object $E_{\bar f}$ and a section-retraction pair $s_{\bar f}\colon E_{\bar f}\leftrightarrows A\colon r_{\bar f}$ such that $s_{\bar f}r_{\bar f}=\id_{E_{\bar f}}$ and $\bar f=r_{\bar f}s_{\bar f}$. Morphisms $f$ of a restriction category are interpreted as being partially defined and so we interpret $E_{\bar f}$ as the domain of definition of $f$. Then \textbf{[R.1]} says that $f$ is defined on its domain of definition. \textbf{[R.2]} says that the intersection of the domain of definition of $f$ with the domain of definition of $g$ is equal to the intersection of the domain of definition of $g$ with that of $f$. \textbf{[R.3]} establishes that the intersection of the domain of definitions of two morphisms $f\colon A\to B$ and $g\colon A\to C$ is equal to the domain of definition of the morphism $f$ restricted to the domain of definition of $g$. Finally, \textbf{[R.4]} establishes that restricting a morphism $f\colon A\to B$ to those elements of $A$ whose image along $f$ lands in the domain of definition of a morphism $g\colon B\to C$ is to restrict $f$ to the domain of definition of the composition $fg$ of $f$ with $g$. Here are some key examples of restriction categories.  

\begin{example}
\label{example:trivial-restriction}
Every category is trivially a restriction category by setting $\bar f = \id_A$ for all morphisms $f\colon A\to B$. In this case, the restriction idempotents are the identities.
\end{example}

\begin{example}
\label{example:restriction-monoid}
A restriction monoid consists of a monoid $M$ equipped with an operation which sends every element $x\in M$ to $\bar x$, such that $\bar xx=x$, $\bar x\bar y=\bar y\bar x=\bar{\bar xy}$, and $x\bar y=\bar{xy}x$. In particular, a restriction monoid is precisely a one-object restriction category.
\end{example}

\begin{example}
\label{example:partial-set-restriction}
Sets and partial functions form a restriction category denoted by $\ParFnc$. The restriction idempotent of a partial morphism $f\colon A\to B$ is the partial morphism $\bar f$ such that $\bar f(a)=a$ for every $a$ on which $f$ is defined, and undefined elsewhere.
\end{example}

\begin{example}
\label{example:open-restriction}
Topological spaces and continuous functions $f\colon A\to B$ defined on open subsets $U\subseteq A$ form a restriction category whose restriction idempotents are defined as in $\ParFnc$.
\end{example}

\begin{example}
\label{example:computable-restriction}
Let $\mathbb{N}$ be the set of natural numbers. Then powers of the natural numbers $\mathbb{N}^n$ and partial recursive functions $f\colon \mathbb{N}^m \to \mathbb{N^n}$ form a restriction subcategory of $\ParFnc$.
\end{example}

\begin{example}
\label{example:rings-restriction}
Commutative and unital rings and non-unital ring morphisms, so functions $f\colon A\to B$ which preserve zero, sum, and multiplication but not necessarily the unit, so $f(1)$ is not necessarily equal to $1$, form a category denoted by $\cRing_\b$. The opposite category $\cRing_\b^\op$ is a restriction category, or in other words, $\cRing_\b$ is a corestriction category. So for a morphism $f\colon A\to B$ in $\cRing_\b$, its corestriction idempotent $\bar f\colon B\to B$ sends each $b\in B$ to $f(1)b$.
\end{example}

We now recall two important classes of morphisms in restriction categories. 

\begin{definition}~\cite[Section~2.1.2]{cockett:restrictionI}
\label{definition:total-morphisms}
In a restriction category $\X$, a \textbf{total morphism} is a morphism $f\colon A\to B$ whose restriction idempotent is the identity, $\bar f=\id_A$. Moreover, a \textbf{restriction idempotent} is an endomorphism $e\colon A\to A$ that is equal to its own restriction idempotent, $e=\bar e$. 
\end{definition}

Intuitively, total morphisms are those morphisms which are everywhere defined, while the restriction idempotents encode domains of definitions of morphisms.

\begin{definition}
\label{definition:total-sub-category}
Total morphisms of a restriction category $\X$ are closed under composition and identities; thus, they form a sub-category of $\X$ denoted by $\Tot[\X]$, known as the \textbf{total subcategory} of $\X$.
\end{definition}

Here are some useful identities which we will make use of throughout this paper. 

\begin{lemma}~\cite[Lemma~2.1]{cockett:restrictionI}
\label{lemma:technical-restriction}
In a restriction category $\X$, the following statements hold:
\begin{enumerate}
\item Every restriction idempotent $e=\bar e$ is an idempotent, that is, $ee=e$;

\item For every $f\colon A\to B$, $\bar f$ is a restriction idempotent, that is, $\bar{\bar f}=\bar f$;

\item For every $f\colon A\to B$ and $g\colon B\to C$, $\bar{fg}=\bar{f\bar g}$;

\item Every monic $m\colon A\to B$ is total.
\end{enumerate}
\end{lemma}

As with ordinary idempotents, we can also ask all restriction idempotents to split, as was done above when we discussed interpretations of the axioms \textbf{[R.1]}-\textbf{[R.4]}. 

\begin{definition}~\cite[Section~2.3.3]{cockett:restrictionI}
\label{definition:split-restriction-category}
A \textbf{split restriction category} is a restriction category $\X$ whose restriction idempotents split, that is, for each restriction idempotent $e=\bar e$, there exists an object $E_e$ of $\X$ and a section-retraction pair $(s_e,r_e)\colon E_e\leftrightarrows A$ such that $r_es_e=e$.
\end{definition}

\begin{example}
\label{example:split-restriction-set}
The restriction categories of Examples~\ref{example:partial-set-restriction} and~\ref{example:open-restriction} split. Concretely, the restriction idempotent $\bar f\colon A\to B$ of a partial (continuous) function splits along the object $\dom f$, which is the domain of definition of $f$.
\end{example}

We conclude this section by recalling that every restriction category $\X$ canonically embeds in a split restriction category $\Split_R[\X]$~\cite[Proposition~2.26]{cockett:restrictionI}, which is the \textbf{split completion} of $\X$. Since $\Split_R[\X]$ plays an important role in our story, let us unpack this construction:
\begin{description}
\item[Objects] An object of $\Split_R[\X]$ consists of a pair $(A,a)$ formed by an object $A$ of $\X$ together with a restriction idempotent $a=\bar a\colon A\to A$ of $\X$;

\item[Morphisms] A morphism $f\colon(A,a)\to(B,b)$ of $\Split_R[\X]$ consists of a morphism $f\colon A\to B$ of $\X$, such that the following diagram commutes: 
\begin{equation*}
\begin{tikzcd}
A & A \\
B & B
\arrow["a", from=1-1, to=1-2]
\arrow["f"', from=1-1, to=2-1]
\arrow["f", from=1-1, to=2-2]
\arrow["f", from=1-2, to=2-2]
\arrow["b"', from=2-1, to=2-2]
\end{tikzcd}
\end{equation*}
Equivalently, $f$ satisfies that $afb=f$;

\item[Identities] The identity morphism $\id_{(A,a)}: (A,a) \to (A,a)$ is the morphism $a\colon A\to A$;

\item[Composition] Composition works as in $\X$;

\item[Restriction] The restriction idempotent of a morphism $f\colon(A,a)\to(B,b)$ is as in $\X$, so $\bar f: (A,a) \to (A,a)$;

\item[Splitting] The splitting of a restriction idempotent $e=\bar e\colon(A,a)\to(A,a)$ is given by the object $(A,e)$ with morphisms $r_e\colon(A,a) \to (A,e)$ and $s_e \colon(A,e)\to(A,a)$ both defined as $r_e=e$ and $s_e =e$. 
\end{description}


%% file: CODE/3-local.tex
\section{Local categories: partiality on objects}
\label{section:local-categories}
In this section, we introduce the first major contribution of this paper: local categories. In contrast to restriction categories, which encode partiality on morphisms, local categories provide a categorical context to capture partiality on objects.

\begin{definition}
\label{definition:local-category}
A \textbf{local structure} on a category $\C$ consists of an assignment for each object $M\in\C$ to an object $\Las M$, called the \textbf{enlargement} of $M$, with a morphism $\eta_M\colon M\to\Las M$, called the \textbf{maximal inclusion} of $M$, which satisfies the following axioms:
\begin{description}
\item[L.1] For every $M\in\C$, $\Las\Las M=\Las M$ and $\eta_{\Las M}=\id_{\Las M}$;

\item[L.2] For every $M\in\C$, $\eta_M\colon M\to\Las M$ is monic in $\C$;

\item[L.3] For every $M\in\C$ and every morphism $f\colon N\to\Las M$, there exists an object $P$ together with two morphisms $g\colon P\to M$ and $m\colon P\to N$ such that the following diagram
\begin{equation*}
\begin{tikzcd}
P & M \\
N & {\Las M}
\arrow["g", from=1-1, to=1-2]
\arrow["m"', from=1-1, to=2-1]
\arrow["\lrcorner"{anchor=center, pos=0.125}, draw=none, from=1-1, to=2-2]
\arrow["{\eta_M}", from=1-2, to=2-2]
\arrow["f"', from=2-1, to=2-2]
\end{tikzcd}
\end{equation*}
is a pullback diagram. Furthermore, $\Las P=\Las N$ and the following diagram commutes:
\begin{equation*}
\begin{tikzcd}
P & N \\
{\Las P} & {\Las N}
\arrow["m", from=1-1, to=1-2]
\arrow["{\eta_P}"', from=1-1, to=2-1]
\arrow["{\eta_N}", from=1-2, to=2-2]
\arrow[equals, from=2-1, to=2-2]
\end{tikzcd}
\end{equation*}
\end{description}
A \textbf{local category} is a category equipped with a local structure.
\end{definition}

The definition of a local category was inspired by Cruttwell's definition of a Generalized Cartesian Differential Catetory. In particular, the assignment $\Las$ was introduced by Cruttwell to axiomatize the \emph{space of tangent vectors} of an object. However, Cruttwell's definition does not involve the map $\eta$. We suggest the interested reader consult~\cite{cruttwell:generalized-CDC}.

\begin{remark}
\label{remark:non-functoriality-of-Las}
It is important to stress that $\Las$ is only an assignment on objects and does not involve any assignment on morphisms. In particular, $\Las$ does not extend to a functor. Similarly, $\eta$ is not a natural transformation.
\end{remark}

\begin{remark}
\label{remark:canonical-pullback-local-cats}
Axiom \textbf{[L.3]} establishes that (1) the pullback of a morphism $f\colon N\to\Las B$ along $\eta_B$ always exists and (2) among the isomorphic choices of these pullbacks (remember that a pullback is only specified up to a unique isomorphism) there is at least one for which $\Las P=\Las N$ and $m\eta_N=\eta_P$. This does not imply that such a pullback is uniquely defined, nor that every pullback of $f$ along $\eta_B$ satisfies this property.
\end{remark}

In a local category, $\Las M$ can be regarded as a total space in which $M$ embeds via the inclusion $\eta_M\colon M\to\Las M$. With this interpretation, \textbf{[L.1]} establishes that the enlargement $\Las\Las M$ of the enlargement $\Las M$ of $M$ is just itself and that the inclusion morphism $\eta_{\Las M}$ corresponds to the identity; \textbf{[L.2]} makes $\eta_M$ into a proper inclusion; finally, \textbf{[L.3]} establishes that the pre-image $P$ of $M$ along $f$ exists in $\C$ and is embedded into $N$. An important class of objects in a local category are those which coincide with their enlargement. 

\begin{definition}
\label{definition:total-objects}
A \textbf{total object} in a local category is an object $M$ which coincides with its enlargement, that is, $M=\Las M$, and $\eta_M=\id_{M}$.
\end{definition}

Here are now some examples of local categories.

\begin{example}
\label{example:trivial-local}
Every category is trivially a local category where for each object $M$ we set $\Las M\= M$ and $\eta_M=\id_{M}$. In this case, every object is total. 
\end{example}

\begin{example}
\label{example:local-monoid}
Given a restriction monoid $M$, consider the category $\L[M]$ of restriction idempotents of $M$, which is the category whose objects are elements $a \in M$ such that $a=\bar a$, while a morphism $x\colon a\to b$ from $a$ to $b$ is an element $x \in M$ such that $\bar x=a$ and $xb=x$. In particular, the identity morphism of an object $a$ is represented by $a$ itself. Thus, $\L[M]$ comes with a local structure, where for each object $a$, $\Las a\= 1$ is the unit of $M$ and $\eta_a\colon a\to 1$ is $a$. For a morphism $x\colon b\to 1$, the pullback of $x$ along $\eta_a$ of Axiom \textbf{[L.3]} is given by the following diagram:
\begin{equation*}
\begin{tikzcd}
{\bar{xa}} & a \\
b & 1
\arrow["xa", from=1-1, to=1-2]
\arrow["{\bar{xa}}"', from=1-1, to=2-1]
\arrow["a", from=1-2, to=2-2]
\arrow["x"', from=2-1, to=2-2]
\end{tikzcd}
\end{equation*}
It is clear that $\Las(\bar{xa})=1=\Las{b}$. Moreover, $\bar{xa}\eta_b=\bar{xa}b=b\bar{xa}=\bar b\bar{xa}=\bar{bxa}=\bar{xa}=\eta_{\bar{xa}}$.
\end{example}

\begin{example}
\label{example:partial-set-local}
Consider the category $\ParSet$ whose objects are pairs $(A,U)$ formed by two sets such that $U\subseteq A$ and whose morphisms $f\colon(A,U)\to(B,V)$ are functions $f\colon U\to V$ defined between the subsets. Thus, $\ParSet$ is a local category, where $\Las(A,U)\=(A,A)$ and $\eta_{(A,U)}$ is the inclusion function $U\hookrightarrow A$.
Given a morphism $f\colon(A,U)\to(B,B)$, the pullback of $f$ along $\eta_{(B,V)}$ of Axiom \textbf{[L.3]} corresponds to the following diagram
\begin{equation*}
\begin{tikzcd}
{(B,f^{-1}(U))} & {(A,U)} \\
{(B,V)} & {(A,A)}
\arrow["{\tilde f}", from=1-1, to=1-2]
\arrow["m"', from=1-1, to=2-1]
\arrow["{\iota_B}", from=1-2, to=2-2]
\arrow["f"', from=2-1, to=2-2]
\end{tikzcd}
\end{equation*}
where $\tilde f(v)\=f(v)$, for each $v\in f^{-1}(U)$ and $m$ is the inclusion of $f^{-1}(U)$ in $V$. We can easily see that $\Las(B,f^{-1}(U))=(B,B)=\Las(B,V)$ and that $m$ commutes with the maximal inclusions.
\end{example}

\begin{example}
\label{example:open-local}
Consider the category $\ParTop$ whose objects are pairs $(A,U)$ formed by two topological spaces such that $U\subseteq A$ and in which the topology of $U$ is the induced topology, and whose morphisms $f\colon(A,U)\to(B,V)$ are continuous functions $f\colon U\to V$ defined between the subspaces. Thus, $\ParTop$ is a local category, where $\Las(A,U)\=(A,A)$ and $\eta_{(A,U)}$ is the inclusion function $U\hookrightarrow A$.
\end{example}

\begin{example}
\label{example:computable-local}
Consider the full subcategory of $\ParSet$ whose objects are pairs $(\mathbb{N}^n, U)$ in $\ParSet$ and morphisms $f\colon(\mathbb{N}^n,U)\to(\mathbb{N}^m,V)$ are partial recursive functions $f\colon U\to V$. This is a local subcategory of $\ParSet$.
\end{example}

\begin{example}
\label{example:rings-local}
Consider the category $\IdemRing$ whose objects are pairs $(R,e_R)$ formed by a unital and commutative ring $R$ and an idempotent element $e_R\in R$, that is, $e_R^2=e_R$, and whose morphisms $f\colon(S,e_S)\to(R,e_R)$ are functions $f\colon S\to R$ that preserve the sum, the zero, and multiplication, but not necessarily the unit, and such that $f(e_S)=f(1_S)=e_R$, where $1_S$ denotes the unit of $S$. Then $\IdemRing^\op$ is a local category, or in other words, $\IdemRing$ is a colocal category where $\Las(R,e_R)\=(R,1_R)$ and $\eta_{(R,e_R)}\colon \Las(R,e_R) \to (R,1_R)$ sends $r\in R$ to $e_Rr$. 
\end{example}


%% file: CODE/4-equivalence.tex
\section{Restriction categories are equivalent to local categories}
\label{section:equivalence-restriction-local}
In this section, we provide one of the main results of this paper: restriction categories are $2$-equivalent to local categories. This justifies the concept that restriction categories and local categories are two sides of the same coin, with one encoding partiality using morphisms and the other using objects. 

We begin by constructing a local category from a restriction category. So given a restriction category $\X$, define the category $\L[\X]$ as the subcategory of total morphisms of the split completion of $\X$, that is, $\L[\X] = \Tot[\Split_R[\X]]$\footnote{This construction is inspired by Cruttwell's work on Generalized Cartesian Differential Categories (GCDCs). In~\cite[Corollary~3.6]{cruttwell:generalized-CDC}, he showed that the subcategory of total morphisms of the split-idempotent completion of a Cartesian Differential Restriction Category carries the structure of a GCDC. Here, we isolate the exact structure on $\Tot[\Split[\X]]$ that retains the information of the restriction structure. However, the map $\eta$ was considered in~\cite{cruttwell:generalized-CDC}.}. Let us unwrap the definition of $\L[\X]$:
\begin{description}
\item[Objects] An object of $\L[\X]$ consists of a pair $(A,a)$ formed by an object $A$ of $\X$ together with a restriction idempotent $a=\bar a\colon A\to A$ of $\X$;

\item[Morphisms] A morphism $f\colon(A,a)\to(B,b)$ of $\L[\X]$ consists of a morphism $f\colon A\to B$ satisfying the following conditions:
\begin{align*}
&\bar f=a           &&fb=f
\end{align*}

\item[Identities] The identity morphism $\id_{(A,a)}: (A,a) \to (A,a)$ is the morphism $a\colon A\to A$;

\item[Composition] Composition works as in $\X$.
\end{description}
On morphisms $f\colon(A,a)\to(B,b)$, the first condition $\bar f=a$ comes from $f$ being a total morphism in $\Split_R[\X]$, which requires the restriction idempotent $\bar f$ of $f$ to be the identity on $(A,a)$, that is, the restriction idempotent $a$. An object $(A,a)$ of $\L[\X]$ can be intepreted as a subobject to the object $A$ of $\X$. This interpretation is supported by the fact that every morphism $f\colon(A,a)\to(B,b)$ satisfies the equation $\bar f=a$, which imposes that the domain of $f$ is fully determined by the restriction idempotent $a$. From a resource-theory point of view, the objects of $\L[\X]$ can be regarded as partially accessible resources, and morphisms as processes which access only the allowed component of each resource. With this perspective, the equation $fb=f$, establishes that the accessible resources of $A$ processed by $f$ end up in the accessible component of $B$.

\par For each object $(A,a)$ of $\L[\X]$, we can associate it to the object $\Las(A,a)\=(A,\id_A)$ which can be regarded as the enlargement in which $(A,a)$ lives. In the resource interpretation, applying the assignment $\Las$ to a resource $(A,a)$ means removing every restriction on the resource $A$. Furthermore, for each object $(A,a)$ there is a distinct morphism
\begin{align*}
&\eta_{(A,a)}\colon(A,a)\xrightarrow{a}(A,\id_A)=\Las(A,a)
\end{align*}
of $\L[\X]$. It is immediate to see that $\Las\Las(A,a)=\Las(A,a)$ and that $\eta_{\Las(A,a)}=\id_{\Las(A,a)}$. This proves that $\Las$ and $\eta$ satisfy \textbf{[L.1]} of Definition~\ref{definition:local-category}. The next lemma proves that \textbf{[L.2]} is also satisfied.

\begin{lemma}
\label{lemma:eta-monic}
For each object $(A,a)$ of $\L[\X]$, the morphism $\eta_{(A,a)}$ is monic in $\L[\X]$.
\end{lemma}
\begin{proof}
Consider two parallel morphisms $f,g\colon(B,b)\to(A,a)$ such that $f\eta_{(A,a)}=g\eta_{(A,a)}$. Since $f$ and $g$ are morphisms of $\L[\X]$, they must satisfy the following equations:
\begin{align*}
&fa=f           &ga=g
\end{align*}
However, $\eta_{(A,a)}$ is, by definition, $a$, thus:
\begin{align*}
&f=fa=f\eta_{(A,a)}=g\eta_{(A,a)}=ga=g
\end{align*}
Thus, $f\eta_{(A,a)}=g\eta_{(A,a)}$ implies $f=g$, that is, $\eta_{(A,a)}$ is monic in $\L[\X]$.
\end{proof}

The next step is to show that $\Las$ and $\eta$ satisfy \textbf{[L.3]}. To show this, let us first prove that $\eta$ admits all pullbacks.

\begin{lemma}
\label{lemma:eta-display}
For every morphism $f\colon(B,b)\to\Las(A,a)$ of $\L[\X]$, the following
\begin{equation*}
\begin{tikzcd}
{(B,\bar{fa})} & {(A,a)} \\
{(B,b)} & {(A,\id_A)}
\arrow["fa", from=1-1, to=1-2]
\arrow["{\bar{fa}}"', from=1-1, to=2-1]
\arrow["a", from=1-2, to=2-2]
\arrow["f"', from=2-1, to=2-2]
\end{tikzcd}
\end{equation*}
is a pullback diagram in $\L[\X]$.
\end{lemma}
\begin{proof}
For starters, let us prove that this diagram commutes. We compute:
\begin{align*}
\bar{fa}f&=~f\bar a              \Tag{\textbf{[R.4]}}\\
&=~fa                            \Tag{\bar a=a}\\
&=~faa                           \Tag{\text{Lemma~\ref{lemma:technical-restriction}~(a)}}
\end{align*}
Now, consider two morphisms $g\colon(C,c)\to(B,b)$ and $h\colon(C,c)\to(A,a)$ such that $gf=ha$. In particular, $gf=ha=h$, since $ha=h$, being $h$ a morphism in $\L[\X]$. Define the following morphism:
\begin{align*}
&t\colon(C,c)\xrightarrow{g}(B,\bar{fa})
\end{align*}
Let us prove that $t$ is a well-defined morphism of $\L[\X]$. First of all, since $\bar g=c$, we get that $\bar t=c$. Moreover, we compute that: 
\begin{align*}
t\bar{fa}&=~g\bar{fa}            \Tag{t=g}\\
&=~\bar{gfa}g                    \Tag{\textbf{[R.4]}}\\
&=~\bar{ha}g                     \Tag{gf=h}\\
&=~\bar hg                       \Tag{ha=h}\\
&=~\bar gg                       \Tag{\bar h=c=\bar g}\\
&=~g                             \Tag{\textbf{[R.1]}}\\
&=~t                             \Tag{t=g}
\end{align*}
Now, we want to show that $t\bar{fa}=g$ and $tfa=h$. The first equation is already proven, since $g=t=t\bar{fa}$. Thus, we easily compute that:
\begin{align*}
tfa&=~gfa                        \Tag{t=g}\\
&=~ha                            \Tag{gf=h}\\
&=~h                             \Tag{ha=h}
\end{align*}
Finally, let $t'\colon(C,c)\to(B,\bar{fa})$ be a morphism of $\L[\X]$ satisfying $t'\bar{fa}=g$ and $t'fa=h$. Thus,
\begin{align*}
&t'=t'\bar{fa}=g=t
\end{align*}
where we used that $t'=t'\bar{fa}$, since $t'$ a morphism of $\L[\X]$. Thus $t$ is unique, and so we conclude that the desired diagram is a pullback. 
\end{proof}

\begin{proposition}
\label{proposition:L-construction}
For a restriction category $\X$, $\L[\X]$ is a local category.
\end{proposition}
\begin{proof}
\textbf{[L.1]} is immediate since $\Las\Las(A,a)=\Las(A,\id_A)=(A,\id_A)=\Las(A,a)$ and $\eta_{\Las(A,a)}=\eta_{(A,\id_A)}=\id_A=\id_{\Las(A,a)}$. \textbf{[L.2]} is proved by Lemma~\ref{lemma:eta-monic}. To prove \textbf{[L.3]}, notice that by Lemma~\ref{lemma:eta-display}, the pullback of $\eta_{(A,a)}$ along every morphism $f\colon(B,b)\to(A,\id_A)$ exists in $\L[\X]$ and the pullback is given by the object $(B,\bar{fa})$. In particular, $\Las(B,\bar{fa})=(B,\id_B)=\Las(B,b)$. Moreover, the morphism $\bar{fa}\colon(B,\bar{fa})\to(B,b)$, the pullback of $\eta$ along $f$, commutes with the maximal inclusions, since:
\begin{align*}
\bar{fa}\eta_{(B,b)}&=~\bar{fa}b\\
&=~\bar{fa}\bar{b}                      \Tag{b=\bar b}\\
&=~\bar{\bar{fa}b}\bar{fa}              \Tag{\textbf{[R.4]}}\\
&=~\bar{b\bar{fa}}\bar{fa}              \Tag{\textbf{[R.2]}}\\
&=~\bar{bfa}\bar{fa}                    \Tag{\text{Lemma}~\ref{lemma:technical-restriction}~(c)}\\
&=~\bar{fa}\bar{fa}                     \Tag{bf=f}\\
&=~\bar{fa}                             \Tag{\text{Lemma}~\ref{lemma:technical-restriction}~(a)}\\
&=~\eta_{(B,\bar{fa})}                  &&&&\qedhere
\end{align*}
\end{proof}

\begin{remark}
\label{remark:examples-restriction-to-local}
Examples~\ref{example:trivial-local}~-~\ref{example:rings-local} are all examples of local categories $\L[\X]$, where $\X$ runs through Examples~\ref{example:trivial-restriction}~-~\ref{example:rings-restriction}, respectively. Of course, the point of this section is to show that all local categories are essentially of this form. 
\end{remark}

We now go the other way, constructing a restriction category $\R[\C]$ from a local category $\C$. So define $\R[\C]$ as follows:
\begin{description}
\item[Objects] An object of $\R[\C]$ is a total object of $\C$;

\item[Morphisms] A morphism of $\R[\C]$ from a total object $M$ to a total object $N$ consists of a class of isomorphism 
of spans of the following form
\begin{equation*}
\begin{tikzcd}
A && B \\
& U
\arrow["{\eta_U}", from=2-2, to=1-1]
\arrow["f"', from=2-2, to=1-3]
\end{tikzcd}
\end{equation*}
where $U$ is an object $\C$ satisfying $\Las U=M$ and $f\colon U\to N$ is a morphism of $\C$. In the following, we shall denote morphisms of $\R[\C]$ as $(U,f)\colon M \nto N$;

\item[Identities] The identity morphism of an object $M=\Las M$ of $\R[\C]$ is the morphism $(M,\id_{\Las M})$;

\item[Composition] The composition is composition of spans. Concretely, given two morphisms $(U,f)\colon M\nto N$ and $(V,g)\colon N\nto P$ of $\R[\C]$, their composition is the morphism $(W,\pi_Vg)\colon M\to P$, where:
\begin{equation*}
\begin{tikzcd}
W & V \\
U & N
\arrow["{\pi_V}", from=1-1, to=1-2]
\arrow[from=1-1, to=2-1]
\arrow["\lrcorner"{anchor=center, pos=0.125}, draw=none, from=1-1, to=2-2]
\arrow["{\eta_V}", from=1-2, to=2-2]
\arrow["f"', from=2-1, to=2-2]
\end{tikzcd}
\end{equation*}

\item[Restriction] The restriction idempotent of a morphism $(U,f)\colon M\nto N$ of $\R[\C]$ is the morphism $\overline{(U,f)}\=(U,\eta_U) \colon M\nto M$. 
\end{description}

\begin{proposition}
\label{proposition:R-construction}
For a local category $\C$, $\R[\C]$ is a restriction category.
\end{proposition}
\begin{proof}
Let us start by proving \textbf{[R.1]}. Consider a morphism $(U,f)\colon M\nto N$ of $\R[\C]$. Since $\eta_U\colon U\to\Las U=M$ is monic, the pullback
\begin{equation*}
\begin{tikzcd}
U & U \\
U & M
\arrow[equals, from=1-1, to=1-2]
\arrow[equals, from=1-1, to=2-1]
\arrow["\lrcorner"{anchor=center, pos=0.125}, draw=none, from=1-1, to=2-2]
\arrow["{\eta_U}", from=1-2, to=2-2]
\arrow["{\eta_U}"', from=2-1, to=2-2]
\end{tikzcd}
\end{equation*}
of $\eta_U$ along itself~\cite[Dual of Example~4.4]{maclane:categories} is just $U$ and the second projection $\pi_U\colon U\to U$ of the pullback is the identity on $U$. Thus:
\begin{align*}
&\bar{(U,f)}(U,f)=(U,\eta_U)(U,f)=(U,\pi_Uf)=(U,f)
\end{align*}
This proves \textbf{[R.1]}. 

Now, consider two morphisms $(U,f)\colon M\nto N$ and $(V,g)\colon M\nto P$ and its restriction idempotents $\bar{(U,f)}=(U,\eta_U)$ and $\bar{(V,g)}=(V,\eta_V)$. The pullback of $\eta_U$ along $\eta_V$ is isomorphic to the pullback of $\eta_V$ along $\eta_U$, by the canonical symmetry $U\times_MV\cong V\times_MU$. Thus, the isomorphism classes represented by $\bar{(U,f)}\bar{(V,g)}$ and $\bar{(V,g)}\bar{(U,f)}$ coincide. This proves \textbf{[R.2]}.
\par To prove \textbf{[R.3]}, consider again the same pair of morphisms $(U,f)\colon M\nto N$ and $(V,g)\colon M\nto P$. First, let us compute the composition $\bar{(U,f)}(V,g)$. Consider the pullback diagram:
\begin{equation*}
\begin{tikzcd}
{U\times_MV} & V \\
U & M
\arrow["{\pi_V}", from=1-1, to=1-2]
\arrow["{\pi_U}"', from=1-1, to=2-1]
\arrow["\lrcorner"{anchor=center, pos=0.125}, draw=none, from=1-1, to=2-2]
\arrow["{\eta_V}", from=1-2, to=2-2]
\arrow["{\eta_U}"', from=2-1, to=2-2]
\end{tikzcd}
\end{equation*}
Thus:
\begin{align*}
&\bar{(U,f)}(V,g)=(U\times_MV,\pi_Vg)
\end{align*}
Therefore:
\begin{align*}
&\bar{\bar{(U,f)}(V,g)}=\bar{(U\times_MV,\pi_Vg)}=(U\times_MV,\eta_{U\times_MV})
\end{align*}
However, $\eta_{U\times_MV}=\pi_V\eta_V$. Now, recall that $\bar{(U,f)}\bar{(V,g)}$ is equal to $(U\times_MV,\pi_V\eta_V)$. Thus, $\bar{(U,f)}\bar{(V,g)}=\bar{\bar{(U,f)}(V,g)}$, and so \textbf{[R.3]} holds.

Finally, consider two composable morphisms $(U,f)\colon M\nto N$ and $(V,g)\colon N\nto P$. Then by the following pullback diagram
\begin{equation*}
\begin{tikzcd}
{U\times_NV} & V \\
U & N
\arrow["{\pi_V}", from=1-1, to=1-2]
\arrow["{\pi_U}"', from=1-1, to=2-1]
\arrow["\lrcorner"{anchor=center, pos=0.125}, draw=none, from=1-1, to=2-2]
\arrow["{\eta_V}", from=1-2, to=2-2]
\arrow["f"', from=2-1, to=2-2]
\end{tikzcd}
\end{equation*}
we have that $(U,f)\bar{(V,g)} = (U\times_NV,\pi_Vg)$. On the other hand, we first compute $\bar{(U,f)(V,g)}$ to be:
\begin{align*}
&\bar{(U,f)(V,g)}=\bar{(U\times_NV,\pi_Vg)}=(U\times_NV,\eta_{U\times_NV})
\end{align*}
However, $\eta_{U\times_NV}=\pi_V\eta_V=\pi_Uf$. Now, let us compose $\bar{(U,f)(V,g)}$ with $(U,f)$. First, consider the following pullback diagram:
\begin{equation*}
\begin{tikzcd}
{U\times_NV} & U & U \\
{U\times_NV} & U & M
\arrow["{\pi_U}", from=1-1, to=1-2]
\arrow[equals, from=1-1, to=2-1]
\arrow["\lrcorner"{anchor=center, pos=0.125}, draw=none, from=1-1, to=2-2]
\arrow[equals, from=1-2, to=1-3]
\arrow[equals, from=1-2, to=2-2]
\arrow["\lrcorner"{anchor=center, pos=0.125}, draw=none, from=1-2, to=2-3]
\arrow["{\eta_U}", from=1-3, to=2-3]
\arrow["{\pi_U}"', from=2-1, to=2-2]
\arrow["{\eta_U}"', from=2-2, to=2-3]
\end{tikzcd}
\end{equation*}
Therefore:
\begin{align*}
&\bar{(U,f)(V,g)}(U,f)=(U\times_NV,\pi_Uf)=(U\times_NV,\pi_V\eta_V)=(U,f)\bar{(V,g)}
\end{align*}
This proves \textbf{[R.4]}.
\end{proof}

\begin{remark}
\label{remark:examples-local-to-restriction}
Examples~\ref{example:trivial-restriction}~-~\ref{example:rings-restriction} are all examples of restriction categories $\R[\C]$, where $\C$ runs through Examples~\ref{example:trivial-local}~-~\ref{example:rings-local}, respectively. Of course, the main result of this section tells that every restriction category is essentially of this form. 
\end{remark}

We now want to extend the assignments which send a restriction category $\X$ to the local category $\L[\X]$ and a local category $\C$ to the restriction category $\R[\C]$ to two $2$-functors. We start by introducing the appropriate notions of functors between local categories.

\begin{definition}
\label{definition:local-functors}
A \textbf{local functor} from a local category $\C$ to another local category $\C'$ consists of a functor $F\colon\C\to\C'$ which preserves the local structure strictly, that is, $F\Las M=\Las'FM$ and $F\eta_M=\eta'_{FM}$, for each $M\in\C$, and that preserves each pullback of $\eta_N$ along every morphism $f\colon M\to\Las N$.
\end{definition}

Next, we introduce $2$-morphisms between local functors.

\begin{definition}
\label{definition:total-local-natural-transformation}
A \textbf{local natural transformation} from a local functor $F\colon\C\to\C'$ to another local functor $G\colon\C\to\C'$ is a natural transformation $\varphi_M\colon FM\to GM$. A \textbf{total natural transformation} from a local functor $F\colon\C\to\C'$ to another local functor $G\colon\C\to\C'$ is a local natural transformation $\varphi_M\colon FM\to GM$ such that for every $M\in\C$, the naturality square diagram
\begin{equation*}
\begin{tikzcd}
FM & GM \\
{F\Las M} & {G\Las M}
\arrow["{\varphi_M}", from=1-1, to=1-2]
\arrow["{F\eta_M}"', from=1-1, to=2-1]
\arrow["{G\eta_M}", from=1-2, to=2-2]
\arrow["{\varphi_{\Las M}}"', from=2-1, to=2-2]
\end{tikzcd}
\end{equation*}
of $\eta_M$ is a pullback diagram.
\end{definition}

Local categories, local functors, and local natural transformations form a $2$-category denoted by $\LCat_\lax$. Furthermore, the $2$-subcategory of $\LCat_\lax$ with total natural transformations is denoted by $\LCat$.

\par Unsurprisingly, the definitions of local functors and local and total natural transformations are designed to match the notions of restriction functors and restriction and total natural transformations via the $\L$ construction. For starters, recall the definitions of restriction functors and restriction/total natural transformations.

\begin{definition}~\cite[Section~2.2.1]{cockett:restrictionI}
\label{definition:restriction-functors}
A \textbf{restriction functor} from a restriction category $\X$ to a restriction category $\X'$ consists of a functor $F\colon\X\to\X'$ which preserves restriction idempotents $F(\bar f)=\bar{F(f)}$. 
\end{definition}

The literature offers two distinct flavours of $2$-morphisms between restriction functors~\cite{cockett:restrictionI}.

\begin{definition}~\cite[Section~2.2.2]{cockett:restrictionI}~\cite[Remark~4.5]{cockett:restrictionII}
\label{definition:total-restriction-natural-transformation}
A \textbf{total natural transformation} from a restriction functor $F\colon\X\to\X'$ to another restriction functor $G\colon\X\to\X'$ consists of a natural transformation $\varphi_A\colon FA\to GA$ such that for each $A$, $\varphi_A$ is total in $\X'$. A \textbf{restriction natural transformation} from a restriction functor $F\colon\X\to\X'$ to another restriction functor $G\colon\X\to\X'$ consists of a collection of total morphisms $\varphi_A\colon FA\to GA$, such that for any morphism $f\colon A\to B$ of $\X$, the following diagram commutes:
\begin{equation*}
\begin{tikzcd}
FA & FA & GA \\
FB && GB
\arrow["{F\bar f}", from=1-1, to=1-2]
\arrow["Ff"', from=1-1, to=2-1]
\arrow["{\varphi_A}", from=1-2, to=1-3]
\arrow["Gf", from=1-3, to=2-3]
\arrow["{\varphi_B}"', from=2-1, to=2-3]
\end{tikzcd}
\end{equation*}
\end{definition}

Restriction categories, restriction functors, and restriction natural transformations form a $2$-category denoted by $\RCat_\lax$. Furthermore, the $2$-subcategory of $\RCat_\lax$ with total natural transformations is denoted by $\RCat$.

\par Now that we have introduced the appropriate $2$-categorical structures for local categories, we can extend the operation $\L$, which sends every restriction category $\X$ to a local category $\L[\X]$, to a $2$-functor. Consider a restriction functor $F\colon\X\to\X'$ of restriction categories and let $\L[F]\colon\L[\X]\to\L[\X']$ be the functor which sends an object $(A,a)$ of $\L[\X]$ to $(FA,Fa)$ and a morphism $f\colon(A,a)\to(B,b)$ of $\L[\X]$ to $Ff$. It is immediate that $(FA,Fa)$ is an object of $\L[\X']$ and that $Ff\colon(FA,Fa)\to(FB,Fb)$ is a morphism of $\L[\X']$. In particular, $\bar{Ff}=F\bar f=Fa$, since $F$ preserves restriction idempotents. To show that $\L[F]$ is local, consider $\eta_{(A,a)}\colon(A,a)\to(A,\id_A)$. Then $\L[F](\Las(A,a))=(FA,F\id_A)=\Las'(\L[F](A,a))$. Moreover:
\begin{align*}
&\L[F](\eta_{(A,a)})=Fa=\eta'_{(FA,Fa)}=\eta'_{\L[F](A,a)}
\end{align*}
Finally, $\L[F]$ preserves the pullback of Lemma~\ref{lemma:eta-display} on the nose. Now, consider a restriction natural transformation $\varphi_A\colon FA\to GA$ between two restriction functors $F,G\colon\X\to\X'$ and define:
\begin{align*}
&\L[\varphi]_{(A,a)}\=(FA,Fa)\xrightarrow{(Fa)\varphi_A}(GA,Ga)
\end{align*}

\begin{lemma}
\label{lemma:lax-restriction-natural-transformations}
Given a restriction natural transformation $\varphi_A\colon FA\to GA$ between two restriction functors $F,G\colon\X\to\X'$, the transformation $\L[\varphi]_{(A,a)}\=(Fa)\varphi_A$ is a local natural transformation from $\L[F]$ to $\L[G]$.
\end{lemma}
\begin{proof}
First, let us prove that $(Fa)\varphi_A$ is a well-defined morphism of $\L[\X']$:
\begin{align*}
\bar{(Fa)\varphi_A}&=~\bar{(Fa)\bar{\varphi_A}}  \Tag{\text{Lemma}~\ref{lemma:technical-restriction}~(c)}\\
&=~\bar{Fa}                                      \Tag{\varphi\text{ total}}\\
&=~Fa                                            \Tag{F\text{ preserves restriction idempotents}}
\end{align*}
Moreover, by the restriction naturality applied to $a\colon A\to A$ we have that: 
\begin{align*}
&(Fa)\varphi_A(Ga)=(Fa)\varphi_A
\end{align*}
This shows that $(Fa)\varphi_A\colon(FA,Fa)\to(GA,Ga)$ is morphism of $\L[\X']$. Now, we want to prove that $\L[\varphi]$ is natural. Consider a morphism $f\colon(A,a)\to(B,b)$ of $\L[\X]$. Let us compute:
\begin{align*}
(\L[F](f))(\L[\varphi]_{(B,b)})&=~(Ff)(Fb)\varphi_B\\
&=~(Ff)\varphi_B                     \Tag{fb=f}\\
&=~(F\bar f)\varphi_A(Gf)            \Tag{(Ff)\varphi_B=(F\bar f)\varphi_A(Gf)}\\
&=~(Fa)\varphi_A(Gf)                 \Tag{\bar f=a}\\
&=~(\L[\varphi]_{(A,a)})(\L[G]f)
\end{align*}
This proves that $\L[\varphi]$ is a local natural transformation as desired. 
\end{proof}

\begin{lemma}
Given a total natural transformation $\varphi_A\colon FA\to GA$ between two restriction functors $F,G\colon\X\to\X'$, the transformation $\L[\varphi]_{(A,a)}\=(Fa)\varphi_A$ is a total natural transformation from $\L[F]$ to $\L[G]$.
\end{lemma}
\begin{proof}
We have already proved in the previous lemma that $(Fa)\varphi_A\colon(FA,Fa)\to(GA,Ga)$ is a well-defined morphism of $\L[\X']$. Moreover, since $\varphi_A$ is natural, for any morphism $f\colon(A,a)\to(B,b)$:
\begin{align*}
&(Ff)\varphi_B=\varphi_A(Gf)=\varphi_AG(\bar f)(Gf)=(F\bar f)\varphi_A(Gf)
\end{align*}
Thus, $\varphi$ is also a restriction natural transformation; thus, by the previous lemma, $\L[\varphi]$ is natural. To conclude, we now need to prove that the following is a pullback diagram:
\begin{equation*}
\begin{tikzcd}
{(FA,Fa)} & {(GA,Ga)} \\
{(FA,F\id_A)} & {(GA,G\id_A)}
\arrow["{\L[\varphi]_{(A,a)}}", from=1-1, to=1-2]
\arrow["{\L[F]\eta_{(A,a)}}"', from=1-1, to=2-1]
\arrow["{\L[G]\eta_{(A,a)}}", from=1-2, to=2-2]
\arrow["{\L[\varphi]_{(A,\id_A)}}"', from=2-1, to=2-2]
\end{tikzcd}
\end{equation*}
However, this corresponds to the diagram
\begin{equation*}
\begin{tikzcd}
{(FA,Fa)} & {(GA,Ga)} \\
{(FA,\id_{FA})} & {(GA,\id_{GA})}
\arrow["{(Fa)\varphi_A}", from=1-1, to=1-2]
\arrow["Fa"', from=1-1, to=2-1]
\arrow["Ga", from=1-2, to=2-2]
\arrow["{\varphi_A}"', from=2-1, to=2-2]
\end{tikzcd}
\end{equation*}
which, by the naturality of $\varphi_A$ which implies $(Fa)\varphi_A=\varphi_AGa$, coincides with the pullback diagram of Lemma~\ref{lemma:eta-display} where $f=\varphi_A$, $a$ is replaced by $Ga$, $b$ by $\id_{Fa}$, and $\bar{fa}=\bar{\varphi_AGa}$.
\end{proof}

Bringing all of this together gives us our $2$-functor from restriction categories to local categories. 

\begin{proposition}
\label{proposition:L-functoriality}
$\L_\lax\colon\RCat_\lax\to\LCat_\lax$ and $\L\colon\RCat\to\LCat$ are $2$-functors. 
\end{proposition}

We now turn our attention towards extending $\R$ to a $2$-functor. Consider a local functor $F\colon\C\to\C'$ between two local categories $\C$ and $\C'$. Define $\R[F]\colon\R[\C]\to\R[\C']$ as the functor which sends an object $M$ of $\R[\C]$ to $FM$ and a morphism $(U,f)\colon M\nto N$ to the morphism $(FU,Ff)\colon FM\nto FN$. Moreover, since $F(\eta_U)=\eta'_{FU}$ for each $U\in\C$, $\R[F]$ preserves the restriction idempotents. Furthermore, since $F$ preserves the pullbacks of the maximal inclusions along every morphism, $\R[F]$ preserves composition. Thus, $\R[F]$ is a restriction functor.
\par Consider now a local natural transformation $\varphi_M\colon FM\to GM$, natural in $M$, between two local functors $F,G\colon\C\to\C'$. Define $\R[\varphi]_M=(FM,\varphi_M)\colon FM\nto FN$. First, notice that $\R[\varphi]$ is total in $\C'$, since $\bar{(FM,\varphi_M)}=(FM,\eta_{FM}')$ and that local functors preserve total objects, $\eta_{FM}'=F\eta_M=F\id_M=\id_{FM}$.

\begin{lemma}
\label{lemma:lax-local-natural-transformations}
Given a local natural transformation $\varphi_M\colon FM\to GM$ between two local functors $F,G\colon\C\to\C'$, the transformation $\R[\varphi]_M\=(FM,\varphi_M)$ is a restriction natural transformation from $\R[F]$ to $\R[G]$.
\end{lemma}
\begin{proof}
Consider a morphism $(U,f)\colon M\nto N$ of $\R[\C]$. We want to show that the following diagram commutes:
\begin{equation*}
\begin{tikzcd}
FM & FM & GM \\
FN && GN
\arrow["{\R[F]\bar{(U,f)}}"{inner sep=.8ex}, "\shortmid"{marking}, from=1-1, to=1-2]
\arrow["{\R[F](U,f)}"'{inner sep=.8ex}, "\shortmid"{marking}, from=1-1, to=2-1]
\arrow["{\R[\varphi]_M}"{inner sep=.8ex}, "\shortmid"{marking}, from=1-2, to=1-3]
\arrow["{\R[G](U,f)}"{inner sep=.8ex}, "\shortmid"{marking}, from=1-3, to=2-3]
\arrow["{\R[\varphi]_N}"'{inner sep=.8ex}, "\shortmid"{marking}, from=2-1, to=2-3]
\end{tikzcd}
\end{equation*}
Let us start by composing $\R[F](U,f)$ with $\R[\varphi]_N$. To this end, we consider the following pullback diagram:
\begin{equation*}
\begin{tikzcd}
FU & FU \\
FN & FN
\arrow[equals, from=1-1, to=1-2]
\arrow["Ff"', from=1-1, to=2-1]
\arrow["Ff", from=1-2, to=2-2]
\arrow[equals, from=2-1, to=2-2]
\end{tikzcd}
\end{equation*}
Therefore, we get that: 
\begin{align*}
&\left(\R[F](U,f)\right)\R[\varphi]_N=(FU,Ff)(FN,\varphi_N)=(FU,(Ff)\varphi_N)
\end{align*}
Let us now compute the composition $(\R[F]\bar{(U,f)})\R[\varphi_M](\R[G](U,f))$:
\begin{align*}
(\R[F]\bar{(U,f)})\R[\varphi_M](\R[G](U,f))&=~(FU,F\eta_U)(FM,\varphi_M)(GU,Gf)\\
&=~(FU,(F\eta_U)\varphi_M)(GU,Gf)
\end{align*}
However, since $\varphi_M$ is natural in $M$, we can replace $(F\eta_U)\varphi_M$ with $\varphi_UG\eta_U=\varphi_U\eta'_{GU}$. Moreover, since $\eta'_{GU}$ is monic, we compute the following pullback diagram:
\begin{equation*}
\begin{tikzcd}
FU & GU & GU \\
FU & GU & GM
\arrow["{\varphi_U}", from=1-1, to=1-2]
\arrow[equals, from=1-1, to=2-1]
\arrow["\lrcorner"{anchor=center, pos=0.125}, draw=none, from=1-1, to=2-2]
\arrow[equals, from=1-2, to=1-3]
\arrow[equals, from=1-2, to=2-2]
\arrow["\lrcorner"{anchor=center, pos=0.125}, draw=none, from=1-2, to=2-3]
\arrow["{\eta'_{GU}}", from=1-3, to=2-3]
\arrow["{\varphi_U}"', from=2-1, to=2-2]
\arrow["{\eta'_{GU}}"', from=2-2, to=2-3]
\end{tikzcd}
\end{equation*}
Therefore:
\begin{align*}
&(\R[F]\bar{(U,f)})\R[\varphi_M](\R[G](U,f))=(FU,(F\eta_U)\varphi_M)(GU,Gf)=(FU,\varphi_UGf)
\end{align*}
Finally, by the naturality of $\varphi$, $\varphi_U(Gf)=(Ff)\varphi_N$. Therefore, we conclude that:
\begin{align*}
&(\R[F]\bar{(U,f)})\R[\varphi_M](\R[G](U,f))=(\R[F](U,f))\R[\varphi]_N
\end{align*}
Thus, $\R[\varphi]$ is a restriction natural transformation.
\end{proof}

\begin{lemma}
\label{lemma:lax-total-natural-transformations}
Given a total natural transformation $\varphi_M\colon FM\to GM$ between two local functors $F,G\colon\C\to\C'$, the transformation $\R[\varphi]_M\=(FM,\varphi_M)$ is a total natural transformation from $\R[F]$ to $\R[G]$.
\end{lemma}
\begin{proof}
We have already commented on the fact that $\R[\varphi]_M$ is a total morphism from $\R[F](M)$ to $\R[G](M)$. To prove that $\R[\varphi]$ is natural, we shall show that the following diagram commutes for every morphism $(U,f)\colon M\nto N$ of $\R[\C]$:
\begin{equation*}
\begin{tikzcd}
FM & GM \\
FN & GN
\arrow["{\R[\varphi]_M}"{inner sep=.8ex}, "\shortmid"{marking}, from=1-1, to=1-2]
\arrow["{\R[F](U,f)}"'{inner sep=.8ex}, "\shortmid"{marking}, from=1-1, to=2-1]
\arrow["{\R[G](U,f)}"{inner sep=.8ex}, "\shortmid"{marking}, from=1-2, to=2-2]
\arrow["{\R[\varphi]_N}"'{inner sep=.8ex}, "\shortmid"{marking}, from=2-1, to=2-2]
\end{tikzcd}
\end{equation*}
In the previous lemma, we have already computed the composition of $\R[F](U,f)$ with $\R[\varphi]_N$, which turns out to be the morphism $(FU,(Ff)\varphi_N)$. To compute the other composition, notice that, since $\varphi$ is total, the following diagram is a pullback:
\begin{equation*}
\begin{tikzcd}
FU & GU \\
FM & GM
\arrow["{\varphi_U}", from=1-1, to=1-2]
\arrow["{F\eta_U}"', from=1-1, to=2-1]
\arrow["\lrcorner"{anchor=center, pos=0.125}, draw=none, from=1-1, to=2-2]
\arrow["{\eta_U}", from=1-2, to=2-2]
\arrow["{\varphi_M}"', from=2-1, to=2-2]
\end{tikzcd}
\end{equation*}
Therefore:
\begin{align*}
&\R[\varphi]_M(\R[G](U,f))=(FM,\varphi_M)(GU,Gf)=(FU,\varphi_U(Gf))
\end{align*}
However, since $\varphi$ is natural, we can replace $\varphi_U(Gf)$ with $(Ff)\varphi_N$ and conclude that $\R[\varphi]$ is a total natural transformation.
\end{proof}

Bringing this together, we obtain a $2$-functor from local categories to restriction categories. 

\begin{proposition}
\label{proposition:R-functoriality}
$\R_\lax\colon\LCat_\lax\to\RCat_\lax$ and $\R\colon\LCat\to\RCat$ are $2$-functors. 
\end{proposition}

Now that we have made $\L$ and $\R$ into $2$-functors, we not only show that $\R_\lax\dashv\L_\lax$ and $\R\dashv\L$ form two $2$-adjunctions, but also that these adjunctions are $2$-equivalences. Let us start by defining the unit and the counit of the adjunction. Consider a restriction category $\X$ and let us unwrap the definition of $\R[\L[\X]]$:
\begin{description}
\item[Objects] An object of $\R[\L[\X]]$ consists of a pair $(A,a)$ formed by an object $A$ of $\X$ together with a restriction idempotent $a=\bar a$ of $\X$. However, since $(A,a)$ should be total in $\L[\X]$, $a=\id_A$. Thus, objects of $\R[\L[\X]]$, really are pairs $(A,\id_A)$, where $A\in\X$;

\item[Morphisms] A morphism $f\colon(A,\id_A)\to(B,\id_B)$ of $\R[\L[\X]]$ consists of a restriction idempotent $a$ of $\X$ together with a morphism $f\colon(A,a)\to(B,\id_B)$ of $\L[\X]$. In particular, $a=\bar f$;

\item[Restriction] The restriction idempotent of a morphism $f\colon(A,\id_A)\to(B,\id_B)$ is given by $\eta_{(A,\bar f)}\colon(A,\bar f)\to(A,\id_A)$, that is, by $\bar f\colon(A,\bar f)\to(A,\id_A)$.
\end{description}

\begin{proposition}
\label{proposition:R-L}
For every restriction category $\X$, the restriction category $\R[\L[\X]]$ is isomorphic, as a restriction category, to $\X$. Moreover, this isomorphism is natural in $\X$.
\end{proposition}
\begin{proof}
From the previous discussion, we can construct a functor $N\colon\X\to\R[\L[\X]]$ which sends every object $A$ to $(A,\id_A)$ and every morphism $f\colon A\to B$ to $f\colon(A,\bar f)\to(B,\id_B)$. Finally, it is easy to see that this functor is an isomorphism of categories, which is natural in $\X$ and preserves the restriction structure.
\end{proof}

Let us now unwrap the definition of $\L[\R[\C]]$ for a local category $\C$:
\begin{description}
\item[Objects] An object of $\L[\R[\C]]$ is a pair $(A,U)$ formed by a total object $A=\Las A$ of $\C$ together with an object $U$ of $\C$, such that $\Las U=A$;

\item[Morphisms] A morphism $(W,f)\colon(A,U)\nto(B,V)$ of $\L[\R[\C]]$ consists of an object $W$ of $\C$ subject to the condition $\Las W=A$, together with a morphism $f\colon W\to B$. The morphism $f$ needs also to meet two further conditions: (1) $f$ should be total in $\R[\C]$ and (2) $(W,f)(V,\eta_V)=(W,f)$. Condition (1) implies that $W=U$, while condition (2) implies the existence of a unique morphism $\tilde f\colon U\to V$ such that $\tilde f\eta_V=f$. Therefore, a morphism of $\L[\R[\C]]$, really is a morphism $\tilde f\colon U\to V$;

\item[Local] The $\Las$-assignment of $\L[\R[\C]]$ sends an object $(A,U)$ to $(A,A)$, and $\eta_{(A,U)}$ is\linebreak $\eta_U\colon U\to A$.
\end{description}

\begin{proposition}
\label{proposition:L-R}
For every local category $\C$, the local category $\L[\R[\C]]$ is isomorphic, as a local category, to $\C$. Moreover, this isomorphism is natural in $\C$.
\end{proposition}
\begin{proof}
From the previous discussion, we can construct a functor $E\colon\L[\R[\C]]\to\C$ which sends an object $(A,U)$ of $\L[\R[\C]]$ to $U$ and a morphism $f\colon(A,U)\to(B,V)$ to $\tilde f\colon U\to V$. It is not hard to see that such a functor is natural in $\C$, and that it preserves the local structure. Furthermore, we can also define an inverse of $E\colon\L[\R[\C]]\to\C$ which sends an object $M$ of $\C$ to $(\Las M,M)$ and a morphism $f\colon M\to N$ to the morphism $f\eta_N\colon(\Las M,M)\to(\Las N,N)$. However, since $\widetilde{f\eta_N}=f$, this functor, which is also a local functor, defines an inverse to $E$. Thus, $E$ is a natural isomorphism of local categories.
\end{proof}

Finally, we can prove the main result of this section.

\begin{theorem}
\label{theorem:2-equivalence-local-restriction}
The two $2$-functors $\L_\lax$ and $\R_\lax$ form a $2$-equivalence $\RCat_\lax\simeq\LCat_\lax$. Moreover, this $2$-equivalence restricts to a $2$-equivalence $\RCat\simeq\LCat$. 
\end{theorem}
\begin{proof}
To prove that $\L_\lax$ and $\R_\lax$ form a $2$-equivalence, one only needs to show that the two natural isomorphisms of Propositions~\ref{proposition:R-L} and~\ref{proposition:L-R} satisfy the triangle identities. Let us start with the first triangle equation:
\begin{align*}
&E_\L\o\L[N]=\id_\L
\end{align*}
First, notice that, $\L[N]$ sends an object $(A,a)$ of $\L[\X]$ to $(N(A),N(a))=((A,\id_A),a)$ and a morphism $f\colon(A,a)\to(B,b)$ to $f\colon((A,\id_A),a)\to((B,\id_B),b)$. Thus, $E_\L\o\L[N]$ sends each $(A,a)$ to $(A,a)$ and each $f\colon(A,a)\to(B,b)$ to itself. Now, let us prove the second triangle equation:
\begin{align*}
&\R[E]\o N_\R=\id_\R
\end{align*}
For starters, $N_\R$ sends an object $M=\Las M$ of $\R[\C]$ to $(M,\id_M)$ and a morphism $f\colon M\to N$ to $f\colon(M,\id_M)\to(N,\id_N)$. Finally, $\R[E]$ sends each $(M,\id_M)$ to $M$ and each $f\colon(M,\id_M)\to(N,\id_N)$ to $f\colon M\to N$.
\end{proof}


%% file: CODE/5-concepts.tex
\section{Concepts of local categories}
\label{section:concepts-local-categories}
Restriction categories offer a language for partiality on morphisms by introducing concepts like total morphisms or the compatibility between parallel morphisms. The goal of this section is to translate these concepts for local categories, giving us a language of partiality on objects.  For starters, let us recall some key concepts in restriction category theory.

\begin{definition}
\label{definition:total-compatible}
In a restriction category $\X$:
\begin{itemize}
\item~\cite[Proposition~6.3]{cockett:classical-distributive-restriction} Two parallel morphisms $f,g\colon A\to B$ are \textbf{compatible}, written as $f\smile g$, when $\bar fg=\bar gf$;

\item~\cite[Definition~6.7]{cockett:classical-distributive-restriction} A \textbf{family of compatible morphisms} consists of a family $\Famly$ of morphisms, which are pair-wise compatible, that is, $f\smile g$, for each $f,g\in\Famly$;

\item~\cite[Section~2.1.4]{cockett:restrictionI} A morphism $g\colon A\to B$ \textbf{extends} a morphism $f\colon A\to B$, written as $f\leq g$, when $\bar fg=f$.
\end{itemize}
\end{definition}

Two parallel morphisms $f$ and $g$ are compatible when, restricting $g$ to the domain of definition of $f$, that is, by considering the morphism $\bar fg$, it amounts to restricting $f$ to the domain of definition of $g$. Moreover, $f\leq g$ when, by restricting $g$ to the domain of definition of $f$, we get $f$.

\begin{example}
\label{example:concepts-trivial-restriction}
In any trivial restriction category as in Example~\ref{example:trivial-restriction}, two parallel morphisms are compatible if and only if they extend each other if and only if they concide.
\end{example}

\begin{example}
\label{example:concepts-partial-functions-restriction}
In the restriction categories of Examples~\ref{example:partial-set-restriction} and~\ref{example:open-restriction}, two parallel partial (continuous) functions $f$ and $g$ are compatible, precisely when they coincide in the intersection of their domains of definition. Moreover, $f\leq g$ if and only if the domain of definition of $f$ is included in the one of $g$ and $f$ is equal to the restriction of $g$ to the domain of $f$.
\end{example}

We now introduce the analogue concepts in the language of local categories.

\begin{definition}
\label{definition:partial-order-in-local-categories}
In a local category $\C$:
\begin{itemize}
\item Two objects $M,N$ are \textbf{compatible}, written as $M\smile N$, when $\Las M=\Las N$;

\item An object $N$ \textbf{extends} an object $M$ \textbf{from above}, written as $M\leq N$, when $M\smile N$ and there exists a morphism $m\colon M\to N$ such that $m\eta_N=\eta_M$.
\end{itemize}
\end{definition}

Two objects $M$ and $N$ of a local category are compatible when they are embedded in the same enlargement. Moreover, $M\leq N$ when $M$ is embedded into $N$.

\begin{example}
\label{example:concepts-trivial-local}
In any trivial local category, as in Example~\ref{example:trivial-local}, two objects are compatible if and only if they extend each other if and only if they coincide.
\end{example}

\begin{example}
\label{example:concepts-partial-sets-local}
In the local categories $\ParSet$ and $\ParTop$ of Examples~\ref{example:partial-set-local} and~\ref{example:open-local}, two objects $(A,U)$ and $(B,V)$ are compatible if and only if $A=B$ and $(A,U)\leq(B,V)$ if and only if $A=B$ and $U\subseteq V$.
\end{example}

Compatibility between morphisms in a restriction category is reflexive and symmetric; however, in general, it fails to be transitive. The next lemma shows that compatibility between objects of a local category defines instead an equivalence relation.

\begin{lemma}
\label{lemma:compatibility-equivalence-relation}
The binary relation $\smile$ of compatibility between objects of a local category is an equivalence relation.
\end{lemma}
\begin{proof}
This is immediate from the definition.
\end{proof}

We now show that the relation $\leq$ defined between objects of a local category is a partial order.

\begin{definition}
\label{definition:ordered-category}
A category $\C$ admits a \textbf{partial order on objects} when $\C$ comes equipped with a binary relation $\leq$ which satisfies the following conditions:
\begin{description}
\item[PO.1] For every object $A$, $A\leq A$;

\item[PO.2] For each $A$, $B$, and $C$, $A\leq B$ and $B\leq C$ implies $A\leq C$;

\item[PO.3] For each $A$ and $B$, $A\leq B$ and $B\leq A$ implies $A\cong B$.
\end{description}
\end{definition}

\begin{lemma}
\label{lemma:partial-order-on-local-categories}
In a local category $\C$, $A\leq B$ if and only if $\Las A=\Las B$ and there exists a unique morphism $m\colon A\to B$ such that $m\eta_B=\eta_A$. Moreover, such $m$ is monic.
\end{lemma}
\begin{proof}
Since $m\eta_B=\eta_A$ is monic, $m$ is monic as well. Moreover, if $m'\colon A\to B$ satisfy the equation $m'\eta_B=\eta_A=m\eta_B$, since $\eta_B$ is monic, $m'=m$.
\end{proof}

\begin{proposition}
\label{proposition:partial-order-on-local-categories}
Every local category $\C$ has a partial order on objects.
\end{proposition}
\begin{proof}
For every $A$ in $\C$, $A\leq A$, since $\id_A\eta_A=\eta_A$. Consider now $A\leq B$ and $B\leq C$ Thus, there exist $m\colon A\to B$ and $m'\colon B\to C$ such that $m\eta_B=\eta_A$ and $m'\eta_C=\eta_B$. Thus, $mm'\eta_C=m\eta_B=\eta_A$. Thus, $A\leq C$. Finally, suppose that $A\leq B$ and $B\leq A$. Thus, there exist $m\colon A\to B$ and $m'\colon B\to A$ such that $m\eta_B=\eta_A$ and $m'\eta_A=\eta_B$. However, using that $\eta_A$ and $\eta_B$ are monic, we also conclude that $mm'$ and $m'm$ must be the identities on $A$ and $B$, respectively. Thus, $A\cong B$.
\end{proof}

Next, we shall extend the partial order between objects of a local category to morphisms. First, let us introduce the following concepts.

\begin{definition}
\label{definition:partial-order-on-morphisms}
In a local category $\C$:
\begin{itemize}
\item Two morphisms $f\colon A\to C$ and $g\colon B\to C$ are \textbf{compatible}, written as $f\smile g$, when their domains are compatible $A\smile B$, and the following diagram commutes
\begin{equation*}
\begin{tikzcd}
{A\wedge B} & B \\
A & C
\arrow["{\iota_B}", from=1-1, to=1-2]
\arrow["{\iota_A}"', from=1-1, to=2-1]
\arrow["g", from=1-2, to=2-2]
\arrow["f"', from=2-1, to=2-2]
\end{tikzcd}
\end{equation*}
where $A \wedge B$ is given by the following pullback: 
\begin{equation*}
\begin{tikzcd}
{A\wedge B} & B \\
A & {\Las A=\Las B}
\arrow["{\iota_B}", from=1-1, to=1-2]
\arrow["{\iota_A}"', from=1-1, to=2-1]
\arrow["\lrcorner"{anchor=center, pos=0.125}, draw=none, from=1-1, to=2-2]
\arrow["{\eta_B}", from=1-2, to=2-2]
\arrow["{\eta_A}"', from=2-1, to=2-2]
\end{tikzcd}
\end{equation*}

\item A \textbf{family of compatible morphisms} consists of a family $\Famly$ of morphisms, which are pair-wise compatible, that is, $f\smile g$ for each $f,g\in\Famly$;

\item A morphism $g\colon B\to C$ \textbf{extends} a morphism $f\colon A\to C$, written $f\leq g$, when $A\leq B$ and the following diagram commutes
\begin{equation*}
\begin{tikzcd}
A \\
B & C
\arrow["m"', from=1-1, to=2-1]
\arrow["f", from=1-1, to=2-2]
\arrow["g"', from=2-1, to=2-2]
\end{tikzcd}
\end{equation*}
where $m\colon A\to B$ is the unique morphism such that $m\eta_B=\eta_A$.
\end{itemize}
\end{definition}

\begin{example}
\label{example:compatibility-morphisms-local-par-set}
In the local categories $\ParSet$ and $\ParTop$ of Examples~\ref{example:partial-set-local} and~\ref{example:open-local}, two morphisms $f\colon(A,U)\to(B,W)$ and $g\colon(A,V)\to(B,W)$ are compatible when $f\colon U\to W$ and $g\colon V\to W$ coincide on the intersection $U\cap V$. A family of compatible morphisms is then a family of morphisms $f_i\colon(A,U_i)\to(B,W)$ for which $f_i$ and $f_j$ coincide on $U_i\cap U_j$, for each $i,j$. Finally, a morphism $g\colon(A,U)\to(B,W)$ bounds a morphism $f\colon(A,V)\to(B,W)$ when $V\subseteq U$ and $g$ is equal to $f$ when restricted to $V$.
\end{example}

Compatibility between morphisms of a local category defines a coherence relation, that is, reflexive and symmetric, which, however, fails in general to be transitive; $\leq$ defines a partial order on morphisms.

\begin{lemma}
\label{lemma:compatibility-equivalence-relation-2}
The binary relation $\smile$ of compatibility between morphisms of a local category is a coherence relation.
\end{lemma}
\begin{proof}
This is a direct consequence of Lemma~\ref{lemma:compatibility-equivalence-relation}.
\end{proof}

\begin{proposition}
\label{proposition:partial-order-morphisms}
The relation $\leq$ defined on morphisms of a local category is a partial order.
\end{proposition}
\begin{proof}
This is a direct consequence of Proposition~\ref{proposition:partial-order-on-local-categories}. We leave the reader to spell out the details of the proof.
\end{proof}

We now prove that compatibility and partial order in the restriction sense correspond to compatibility and partial order in the local sense.

\begin{lemma}
\label{lemma:order-matters}
Consider two parallel morphisms $f,g\colon A\to B$ in a restriction category $\X$. Then $f$ and $g$ are compatible in $\X$ if and only if $f\colon(A,\bar f)\to(B,\id_B)$ and $g\colon(A,\bar g)\to(B,\id_B)$ are compatible in $\L[\X]$. Moreover, $f\leq g$ in $\X$ if and only if $f\leq g$ in $\L[\X]$.
\end{lemma}
\begin{proof}
For starters, let us assume that $f\smile g$ in $\X$. Thus, $\bar gf=\bar fg$. By Lemma~\ref{lemma:eta-display}, the pullback of $\eta_{(A,\bar f)}$ along $\eta_{(A,\bar g)}$ is given by the following diagram:
\begin{equation*}
\begin{tikzcd}
{(A,\bar f\bar g)} & {(A,\bar g)} \\
{(A,\bar f)} & {(A,\id_A)}
\arrow["{\bar f\bar g}", from=1-1, to=1-2]
\arrow["{\bar f\bar g}"', from=1-1, to=2-1]
\arrow["\lrcorner"{anchor=center, pos=0.125}, draw=none, from=1-1, to=2-2]
\arrow["{\bar g}", from=1-2, to=2-2]
\arrow["{\bar f}"', from=2-1, to=2-2]
\end{tikzcd}
\end{equation*}
Thus, since $f\smile g$, the following diagram commutes
\begin{equation*}
\begin{tikzcd}
{(A,\bar f\bar g)} & {(A,\bar g)} \\
{(A,\bar f)} & {(B,\id_B)}
\arrow["{\bar f\bar g}", from=1-1, to=1-2]
\arrow["{\bar f\bar g}"', from=1-1, to=2-1]
\arrow["g", from=1-2, to=2-2]
\arrow["f"', from=2-1, to=2-2]
\end{tikzcd}
\end{equation*}
since $\bar f\bar gf=\bar g\bar ff=\bar gf=\bar fg=\bar f\bar gg$. Thus, $f$ and $g$ are compatible in $\L[\X]$. 
\par Now, suppose that $f$ and $g$ are compatible in $\L[\X]$. Thus, the previous diagram must commute, explicitly, $\bar f\bar gf=\bar f\bar gg$; however, this means that $f\smile g$ in $\X$.
\par Now, assume that $f\leq g$ in $\X$. Thus, $\bar gf=f$. In particular, this implies that $\bar f\colon(A,\bar f)\to(A,\bar g)$ is a morphism of $\L[\X]$. Moreover, it is the unique morphism such that $\bar f\eta_{(A,\bar g)}=\bar f\bar g=\bar {\bar gf}=\bar f=\eta_{(A,\bar f)}$. Thus, $(A,\bar f)\leq(A,\bar g)$. Moreover, $\bar fg=f$. Thus, $f\leq g$ in $\L[\X]$.
\par Conversely, assume that $f\leq g$ in $\L[\X]$. Thus, there exists a morphism $m\colon A\to A$ such that the following equalities hold: $\bar m=\bar f$, $m\bar g=m$, and $m\bar g=\bar f$. Thus, $m=m\bar g=\bar f$. Moreover, $mg=f$. Thus, $\bar fg=f$, that is, $f\leq g$ in $\X$.
\end{proof}

Conversely, we can prove that compatibility and partial order in a local category correspond to compatibility and partial order in a restriction category. 

\begin{lemma}
\label{lemma:order-matters-2}
Consider two morphisms $f\colon M\to P$ and $g\colon N\to P$ in a local category $\C$. Thus, $f$ and $g$ are compatible in $\C$ if and only if $(M,f\eta_P)\colon\Las M\nto\Las P$ and $(N,g\eta_P)\colon\Las N\nto\Las P$ are compatible in $\R[\X]$. Moreover, $f\leq g$ in $\C$ if and only if $(M,f\eta_P)\leq(N,g\eta_P)$ in $\R[\X]$.
\end{lemma}
\begin{proof}
Consider $f$ and $g$ in $\C$ and assume that $f\smile g$; in particular, $\Las M=\Las N$. Consider now the composition $\bar{(M,f\eta_P)}(N,g\eta_P)$ and $\bar{(N,g\eta_P)}(M,f\eta_P)$ in $\R[\X]$:
\begin{equation*}
\adjustbox{width=\linewidth}{
\begin{tikzcd}
{\Las M} && {\Las N} && {\Las P} \\
& M && N \\
&& {M\wedge N}
\arrow["{\eta_M}", from=2-2, to=1-1]
\arrow["{\eta_M}"', from=2-2, to=1-3]
\arrow["{\eta_N}", from=2-4, to=1-3]
\arrow["{g\eta_P}"', from=2-4, to=1-5]
\arrow["\lrcorner"{anchor=center, pos=0.125, rotate=135}, draw=none, from=3-3, to=1-3]
\arrow[from=3-3, to=2-2]
\arrow[from=3-3, to=2-4]
\end{tikzcd}\qquad
\begin{tikzcd}
{\Las N} && {\Las M} && {\Las P} \\
& N && M \\
&& {M\wedge N}
\arrow["{\eta_N}", from=2-2, to=1-1]
\arrow["{\eta_N}"', from=2-2, to=1-3]
\arrow["{\eta_M}", from=2-4, to=1-3]
\arrow["{f\eta_P}"', from=2-4, to=1-5]
\arrow["\lrcorner"{anchor=center, pos=0.125, rotate=135}, draw=none, from=3-3, to=1-3]
\arrow[from=3-3, to=2-2]
\arrow[from=3-3, to=2-4]
\end{tikzcd}
}
\end{equation*}
However, since $f$ and $g$ are compatible, the right legs of the two spans coincide. Thus, $\bar{(M,f\eta_P)}(N,g\eta_P)=\bar{(N,g\eta_P)}(M,f\eta_P)$. Now, assume that $(M,f\eta_P)$ and $(N,g\eta_P)$ are compatible. However, this implies that the two right legs of the two above spans must coincide, which means that $f\smile g$ in $\C$.
\par Now, suppose that $f\leq g$ in $\C$. Thus, $M\leq N$ and $mg=f$, for $m$ the (necessarily unique) morphism of $\C$ such that $m\eta_N=\eta_M$. Now, consider the composition $\bar{(M,f\eta_P)}(N,g\eta_P)$ in $\R[\C]$:
\begin{equation*}
\begin{tikzcd}
{\Las M} && {\Las N} && {\Las P} \\
& M && N \\
&& {M\wedge N}
\arrow["{\eta_M}", from=2-2, to=1-1]
\arrow["{\eta_M}"', from=2-2, to=1-3]
\arrow["{\eta_N}", from=2-4, to=1-3]
\arrow["{g\eta_P}"', from=2-4, to=1-5]
\arrow["\lrcorner"{anchor=center, pos=0.125, rotate=135}, draw=none, from=3-3, to=1-3]
\arrow[from=3-3, to=2-2]
\arrow[from=3-3, to=2-4]
\end{tikzcd}
\end{equation*}
However, since $m\eta_N=\eta_M$ and that $\eta_N$ is monic, we can rewrite it as follows:
\begin{equation*}
\begin{tikzcd}
{\Las M} &&&& {\Las N} && {\Las P} \\
&&& N && N \\
&& M && N \\
&&& M
\arrow["{\eta_N}"', from=2-4, to=1-5]
\arrow["{\eta_N}", from=2-6, to=1-5]
\arrow["{g\eta_P}"', from=2-6, to=1-7]
\arrow["{\eta_M}", from=3-3, to=1-1]
\arrow["m"', from=3-3, to=2-4]
\arrow["\lrcorner"{anchor=center, pos=0.125, rotate=135}, draw=none, from=3-5, to=1-5]
\arrow[equals, from=3-5, to=2-4]
\arrow[equals, from=3-5, to=2-6]
\arrow["\lrcorner"{anchor=center, pos=0.125, rotate=135}, draw=none, from=4-4, to=2-4]
\arrow[equals, from=4-4, to=3-3]
\arrow["m"', from=4-4, to=3-5]
\end{tikzcd}
\end{equation*}
However, $mg=f$, thus, $\bar{(M,f\eta_P)}(N,g\eta_P)=(M,f\eta_P)$. Therefore, $(M,f\eta_P)\leq(N,g\eta_P)$ in $\R[\C]$.
\end{proof}

We conclude this section by showing that \emph{restricted monics} in restriction categories correspond to ordinary monics through the $2$-equivalence of Theorem~\ref{theorem:2-equivalence-local-restriction}. For starters, we recall the notion of restricted monics.

\begin{definition}[{\cite[Section~2.3.1]{cockett:restrictionI}}]
\label{definition:restricted-monic}
A \textbf{restricted monic} in a restriction category $\X$ consists of a morphism $f\colon A\to B$ such that, given two parallel morphisms $g,h\colon C\to A$, $gf=hf$ implies that $g\bar f=h\bar f$.
\end{definition}

\begin{lemma}
\label{lemma:partial-monic}
A morphism $f\colon(A,a)\to(B,b)$ of $\L[\X]$ is monic if and only if $f$ is a restricted monic of $\X$.
\end{lemma}
\begin{proof}
Suppose first that $f$ is monic in $\L[\X]$ and consider two parallel morphisms $g,h\colon C\to A$ of $\X$ such that $gf=hf$. Consider the two morphisms $g\bar f\colon(C,\bar{gf})\to(A,\bar f)$ and $h\bar f\colon(C,\bar{hf})\to(A,\bar f)$. First, we need to show that these are morphisms of $\L[\X]$. We start with $g\bar f$. Using Lemma~\ref{lemma:technical-restriction}(c), we have that $\bar{g\bar f}=\bar{gf}$, and by Lemma~\ref{lemma:technical-restriction}(a), $g\bar f\bar f=g\bar f$. Thus, $g\bar f$ becomes a morphism of $\L[\X]$. Same argument applies to $h\bar f$. However, since $gf=hf$, then $(C,\bar{gf})$ and $(C,\bar{hf})$ are the same object. Furthermore, $g\bar ff=gf=hf=h\bar ff$. Since $f$ is monic in $\L[\X]$, this implies that $g\bar f=h\bar f$.

Conversely, assume that $f$ is a restricted monic in $\X$ and consider two parallel morphisms $g,h\colon(C,\bar g=\bar h)\to(A,\bar f)$ such that $gf=hf$. Thus, $g\bar f=h\bar f$. However, since $g$ and $h$ are morphisms in $\L[\X]$, $g\bar f=g$ and $h\bar f=h$. Thus, $g=g\bar f=h\bar f=h$. Thus, $f$ is monic in $\L[\X]$.
\end{proof}

\begin{example}
\label{example:partial-monics-par-fnc}
A partial function $f\colon A\to B$ of $\ParFnc$ (Example~\ref{example:partial-set-restriction}) is a restricted monic if and only if $f\colon(U,A)\to(B,B)$ is a monic in $\ParSet$ of Example~\ref{example:partial-set-local}, where $U$ is the domain of definition of $f$. However, since $\ParSet$ is equivalent to the category $\Set$ of sets, monics in $\ParSet$ are precisely injective functions $f\colon U\to B$. Thus, restricted monics in $\ParFnc$ are precisely partial injective functions.
\end{example}

\begin{lemma}
\label{lemma:partial-monic-2}
A morphism $f\colon A\to B$ of a local category $\C$ is a monic in $\C$ if and only if $(A,f\eta_B)\colon\Las A\nto\Las B$ is a restricted monic in $\R[\C]$.
\end{lemma}
\begin{proof}
Assume that $f$ is monic in $\C$ and let us prove that $(A,f\eta_B)$ is restricted monic in $\R[\C]$. Consider two morphisms $(C,g),(C',h)\colon\Las C\nto\Las A$ of $\R[\C]$. In particular, they correspond to two morphisms $g\colon C\to\Las C$ and $h\colon C'\to\Las C'=\Las C$ of $\C$. Supppose also that $(C,g)(U,f\eta_B)=(C',h)(U,f\eta_B)$. Thus, $C=C'$ and there exist two morphisms $\tilde g\colon C\to A$ and $\tilde h\colon C\to A$ pullback of $g$ and $h$ along $\eta_A$, respectively, such that, $\tilde gf\eta_B=\tilde hf\eta_B$. Using that both $f$ and $\eta_B$ are monic in $\C$, this implies that $\tilde g=\tilde h$. However, $(C,g)\bar{(A,f\eta_B)}=(C,g)(A,\eta_A)=(C,\tilde g\eta_A)$ and similarly, $(C,h)\bar{(A,f\eta_B)}=(C,\tilde h\eta_A)$. Therefore, $(C,g)\bar{(A,f\eta_B)}=(C,h)\bar{(A,f\eta_B)}$.

Conversely, suppose that $(A,f\eta_B)$ is a restricted monic in $\R[\C]$. Consider two morphisms $g,f\colon C\to A$ of $\C$ such that $gf=hf$. Using that $\eta_A$ is monic, it is easy to see that $(C,g\eta_A)(A,f\eta_B)=(C,gf\eta_B)=(C,hf\eta_B)=(C,h\eta_A)(A,f\eta_B)$ in $\R[\C]$. However, since $(A,f\eta_B)$ is a restricted monic in $\R[\C]$, this implies that $(C,g\eta_A)=(C,g\eta_A)(A,\eta_A)=(C,h\eta_A)(A,\eta_A)$, which means that $g\eta_A=h\eta_A$. However, since $\eta_A$ is monic, this implies that $g=h$, proving that $f$ is monic in $\C$.
\end{proof}


%% file: CODE/6-cartesian.tex
\section{Cartesian local categories}
\label{section:cartesian}
In this section, we introduce Cartesian local categories, which, as the name suggests, are the local category analogue of Cartesian restriction categories under the $2$-equivalence of Theorem~\ref{theorem:2-equivalence-local-restriction}.

\par To begin, let us briefly recall the definition of a Cartesian restriction category, which was introduced and studied in~\cite{cockett:restrictionIII}. To this end, recall that a Cartesian object in a $2$-category $\CC$ with finite $2$-products (including $2$-terminal objects) consists of an object $\X$ of $\CC$ such that the terminal morphism $!\colon\X\to\1$ and the diagonal morphism $\Delta\colon\X\to\X\times\X$ admit right adjoints in $\CC$~\cite{carboni:cartesian-objects}. Therefore, a Cartesian restriction category is a Cartesian object in the $2$-category $\RCat_\lax$. This leads us to the notions of a restriction terminal object and restriction products, which have slightly weaker universal properties than the usual notions. 

\begin{definition}~\cite[Section~4.1]{cockett:restrictionIII}
\label{definition:restriction-terminal-object}
A \textbf{restriction terminal object} in a restriction category $\X$ consists of an object $\*$ of $\X$ such that for every object $A$ of $\X$ there exists a unique total morphism $!_A\colon A\to\*$ and for any morphism $f\colon A\to B$, the following diagram commutes:
\begin{equation*}
\begin{tikzcd}
A & B \\
A & {\*}
\arrow["f", from=1-1, to=1-2]
\arrow["{\bar f}"', from=1-1, to=2-1]
\arrow["{!_B}", from=1-2, to=2-2]
\arrow["{!_A}"', from=2-1, to=2-2]
\end{tikzcd}
\end{equation*}
\end{definition}

\begin{definition}~\cite[Section~4.1]{cockett:restrictionIII}
\label{definition:restriction-products}
The \textbf{restriction product} of two objects $A$ and $B$ in a restriction category $\X$ consists of an object $A\times B$ together with two total morphisms $\pi_A\colon A\times B\to A$ and $\pi_B\colon A\times B\to B$ such that for any pairs of morphisms $f\colon C\to A$ and $g\colon C\to B$, there exists a unique morphism $\<f,g\>\colon C\to A\times B$ rendering the following diagram
\begin{equation*}
\begin{tikzcd}
C & C & C \\
A & {A\times B} & B
\arrow["f"', from=1-1, to=2-1]
\arrow["{\bar g}"', from=1-2, to=1-1]
\arrow["{\bar f}", from=1-2, to=1-3]
\arrow["{\<f,g\>}", dashed, from=1-2, to=2-2]
\arrow["g", from=1-3, to=2-3]
\arrow["{\pi_A}", from=2-2, to=2-1]
\arrow["{\pi_B}"', from=2-2, to=2-3]
\end{tikzcd}
\end{equation*}
commutative. In other words, $\<f,g\>$ is the unique morphism such that $\bar{\<f,g\>} =\bar f\bar g$ and also that $\<f,g\>\pi_A \leq f$ and $\< f,g \> \pi_B \leq g$.
\end{definition}

Thus, a Cartesian restriction category consists of a restriction category which admits a restriction terminal object and restriction products.

\begin{example}
\label{example:cartesian-restriction-set}
The restriction categories of Examples~\ref{example:partial-set-restriction} and~\ref{example:open-restriction} are Cartesian. In both cases, the restriction terminal object coincide with the singleton set $\{\*\}$ and the restriction product of $A$ and $B$ is the Cartesian product $A\times B\=\{(a,b),a\in A,b\in B\}$. It is important to stress that, although $\{\*\}$ and $A\times B$ are a terminal object and the proper categorical product of $A$ and $B$ in the total subcategories, respectively, they are not in the full restriction category. 
\end{example}

To correctly define a Cartesian local category, we need to define the local analogue of restriction terminal objects and restriction products. To this end, we explain how restriction terminal objects and restriction products correspond to usual terminal objects and products in the local category $\L[\X]$.

\begin{lemma}
\label{lemma:L-terminal-object}
Let $\X$ be a restriction category. An object $\*$ of $\X$ is a restriction terminal object if and only if $(\*,\id_\*)$ is a terminal object in $\L[\X]$.
\end{lemma}
\begin{proof}
Suppose that $\*$ is a restriction terminal object of $\X$. Consider an object $(A,a)$ of $\L[\X]$ and define the morphism $!_{(A,a)} \colon(A,a) \to (\*,\id_\*)$ as the composite $!_{(A,a)} = a!_A$. First, let us show that $!_{(A,a)}$ is well-defined:
\begin{align*}
\bar{!_{(A,a)}}&=~\bar{a!_A}\\
&=~\bar{a\bar{!_A}}    \Tag{\text{Lemma}~\ref{lemma:technical-restriction}~(c)}\\
&=~\bar a            \Tag{!_A\text{ total}}\\
&=~a                 \Tag{\bar a=a}
\end{align*}
This proves that $!_{(A,a)}$ is a morphism of $\L[\X]$. Now, consider a morphism $f\colon(A,a)\to(\*,\id_\*)$. Thus, $\bar f=a$. Moreover, by the universal property of $\*$, the following diagram commutes:
\begin{equation*}
\begin{tikzcd}
A & {\*} \\
A & {\*}
\arrow["f", from=1-1, to=1-2]
\arrow["{\bar f}"', from=1-1, to=2-1]
\arrow["{!_\*}", from=1-2, to=2-2]
\arrow["{!_A}"', from=2-1, to=2-2]
\end{tikzcd}
\end{equation*}
However, since $!_\*=\id_\*$ (since the identity morphism is total and there is a unique total morphism of type $\ast \to \ast$), this implies that $f=\bar f!_A=a!_A=!_{(A,a)}$. Thus, $(\*,\id_\*)$ is terminal in $\L[\X]$, as desired. 

Conversely, assume that $(\*,\id_\*)$ is terminal in $\L[\X]$. Consider an object $A$ of $\X$ and let $!_A\colon(A,\id_A)\to(\*,\id_\*)$ be the terminal morphism $\L[\X]$. In particular, this means that $!_A\colon A\to\*$ is total in $\X$ and we also have that $!_\*=\id_\*$. Consider now a morphism $f\colon A\to B$. Thus, there exists a unique morphism $!_{(A,\bar f)}$ from $(A,\bar f)$ to $(\*,\id_\*)$. However, the morphism
\begin{align*}
&(A,\bar f)\xrightarrow{f}(B,\id_B)\xrightarrow{!_B}(\*,\id_\*)
\end{align*}
is also well-defined in $\L[\X]$. Thus, $f!_B$ must coincide with $!_{(A,\bar f)}$. However, for any object $(A,a)$, $!_{(A,a)}$ coincides with $a!_A$, since $a!_A=\eta_{(A,a)}!_{\Las(A,a)}\colon(A,a)\to(\*,\id_\*)$ is a morphism of $\L[\X]$. Thus, $f!_B=\bar f!_A$. So we conclude that $\ast$ is a restriction terminal object in $\mathbb{X}$ as desired. 
\end{proof}

\begin{lemma}
\label{lemma:L-products}
Let $\X$ be a restriction category and consider two objects $A$ and $B$ of $\X$. An object $A\times B$ is the restriction product of $A$ and $B$ in $\X$ if and only if $(A\times B,\id_{A\times B})$ is the product of $(A,\id_A)$ and $(B,\id_B)$ in $\L[\X]$.
\end{lemma}
\begin{proof}
Suppose that $A\times B$ is the restriction product of $A$ and $B$ in $\X$. Consider an object $(C,c)$ of $\L[\X]$ and two morphisms $f\colon(C,c)\to(A,\id_A)$ and $g\colon(C,c)\to(B,\id_B)$. In particular, $\bar f=c=\bar g$. Thus, there exists a unique morphism $\<f,g\>\colon C\to A\times B$ such that $\<f,g\>\pi_A=\bar gf$ and $\<f,g\>\pi_B=\bar fg$. However, since $\bar f=\bar g$, these two equations become $\<f,g\>\pi_A=f$ and $\<f,g\>\pi_B=g$. Moreover:
\begin{align*}
\bar{\<f,g\>}&=~\bar{\<f,g\>\bar\pi_A}              \Tag{\bar\pi_A=\id_{A\times B}}\\
&=~\bar{\<f,g\>\pi_A}        \Tag{\text{Lemma}~\ref{lemma:technical-restriction}~(c)}\\
&=~\bar f                    \Tag{\<f,g\>\pi_A=\bar gf=f}\\
&=~c                         \Tag{\bar f=c}
\end{align*}
Therefore, $\<f,g\>$ becomes the unique morphism $\<f,g\>\colon(C,c)\to(A\times B,\id_{A\times B})$ satisfying $\<f,g\>\pi_A=f$ and $\<f,g\>\pi_B=g$. Thus, $(A\times B,\id_{A\times B})$ is the product of $(A,\id_A)$ and $(B,\id_B)$ in $\L[\X]$ as desired. 

Conversely, suppose that $(A\times B,\id_{A\times B})$ is the product of $(A,\id_A)$ and $(B,\id_B)$ in $\L[\X]$. Thus, there are two morphisms
\begin{align*}
&\pi_A\colon(A\times B,\id_{A\times B})\to(A,\id_A)     &\pi_B&\colon(A\times B,\id_{A\times B})\to(B,\id_B)
\end{align*}
satisfying the universal property of a product. In particular, $\pi_A$ and $\pi_B$ are total in $\X$. Consider now an object $C$ of $\X$ and two morphisms $f\colon C\to A$ and $g\colon C\to B$. Consider the two following morphisms:
\begin{align*}
&(C,\bar f\bar g)\xrightarrow{\bar gf}(A,\id_A)         &&(C,\bar f\bar g)\xrightarrow{\bar fg}(B,\id_B)
\end{align*}
It is immediate to see that $\bar{\bar gf}=\bar f\bar g$ and $\bar{\bar gf}=\bar f\bar g$, thus, these are two morphisms of $\L[\X]$. Therefore, there exists a unique morphism $h\colon(C,\bar f\bar g)\to(A\times B,\id_{A\times B})$ of $\L[\X]$ such that $h\pi_A=\bar gf$ and $h\pi_B\bar fg$. Thus, $A\times B$ is the restriction product of $A$ and $B$ in $\X$.
\end{proof}

Inspired by the above lemmas, we define local terminal objects and local products as being usual terminal objects and products which are compatible with the local structure as follows.

\begin{definition}
\label{definition:local-terminal-object}
A \textbf{local terminal object} in a local category $\C$ consists of a total object $\*$ of $\C$ which is terminal in $\C$.
\end{definition}

\begin{definition}
\label{definition:cartesian-local-products}
A \textbf{local product} between two object $M$ and $N$ in a local category $\C$ consists of an object $M\times N$ together with two morphisms $\pi_M\colon M\times N\to M$ and $\pi_N\colon M\times N\to N$ such that $M\times N$ is the product of $M$ and $N$ in $\C$, and also that the product of $\Las M$ and $\Las N$ exists in $\C$ with $\Las M \times \Las N = \Las(M\times N)$ and $\eta_{M\times N} = \eta_M\times\eta_N$.
\end{definition}

\begin{definition}
\label{definition:cartesian-local-category}
A \textbf{Cartesian local category} consists of a local category $\C$ which admits a local terminal object and local products.
\end{definition}

\begin{example}
\label{example:cartesian-local-set}
The local categories $\ParSet$ and $\ParTop$ of Examples~\ref{example:partial-set-local} and~\ref{example:open-local} are Cartesian. Concretely, in both of the examples, the local terminal object is $(\{\*\},\{\*\})$ while the local product of $(A,U)$ and $(B,V)$ is given by $(A\times B,U\times V)$.
\end{example}

To further justify these definitions, we now show that a restriction category has restriction products if and only if its associated local category does. 

\begin{proposition}
\label{proposition:L-products}
A restriction category $\X$ admits a restriction terminal object $\*$ if and only if $\L[\X]$ admits a local terminal object. Moreover, $\X$ admits restriction products if and only if $\L[\X]$ admits local products.
\end{proposition}
\begin{proof}
The first statement is precisely the thesis of Lemma~\ref{lemma:L-terminal-object}. Let us focus on the second statement. First, assume that $\L[\X]$ admits local products. Consider two objects $A$ and $B$ of $\X$. Thus, the local product $(A,\id_A)\times(B,\id_B)$ of $(A,\id_A)$ and $(B,\id_B)$ exists in $\L[\X]$. Moreover, since $(A,\id_A)$ and $(B,\id_B)$ are total we have that: 
\begin{align*}
&\Las((A,\id_A)\times(B,\id_B))=\Las(A,\id_A)\times\Las(B,\id_B)=(A,\id_A)\times(B,\id_B)
\end{align*}
Therefore, $(A,\id_A)\times(B,\id_B)=(A\times B,\id_{A\times B})$. Thus, by Lemma~\ref{lemma:L-products}, $A\times B$ is the restriction product of $A$ and $B$ in $\X$. 

Conversely, assume that $\X$ admits all restriction products. Consider two objects $(A,a)$ and $(B,b)$ of $\L[\X]$. Define $(A,a)\times(B,b)\=(A\times B,a\times b)$, where $A\times B$ is the restriction product of $A$ and $B$ and $a\times b\colon A\times B\to A\times B$ is the unique morphism such that $(a\times b)\pi_A=\bar{\pi_Bb}\pi_Aa$ and $(a\times b)\pi_B=\bar{\pi_Aa}\pi_Bb$. Let us prove that $a\times b=\bar{\pi_Ab}\bar{\pi_Ba}$:
\begin{align*}
\bar{\pi_Aa}\bar{\pi_Bb}\pi_A&=~\bar{\pi_Bb}\bar{\pi_Aa}\pi_A              \Tag{\textbf{[R.2]}}\\
&=~\bar{\pi_Bb}\pi_A\bar a                                                 \Tag{\textbf{[R.4]}}\\
&=~\bar{\pi_Bb}\pi_Aa                                                      \Tag{\bar a=a}
\end{align*}
Similarly, one can show that $\bar{\pi_Aa}\bar{\pi_Bb}\pi_B=\bar{\pi_Aa}\pi_Bb$. Thus, by the universal property of restriction products, $a\times b=\bar{\pi_Ab}\bar{\pi_Ba}$. We use this to prove that $(A\times B,a\times b)$ is the local product of $(A,a)$ and $(B,b)$ in $\L[\X]$. First, define the two projections as follows:
\begin{align*}
\pi_{(A,a)}&\colon(A\times B,a\times b)\xrightarrow{\bar{\pi_Bb}\pi_Aa}(A,a)    &
\pi_{(B,b)}&\colon(A\times B,a\times b)\xrightarrow{\bar{\pi_Aa}\pi_Bb}(B,b)
\end{align*}
It is not hard to see that $\pi_{(A,a)}$ and $\pi_{(B,b)}$ are well-defined. Now consider two morphisms $f\colon(C,c)\to(A,a)$ and $g\colon(C,c)\to(B,b)$. In particular, $\bar f=c=\bar g$ and $fa=f$ and $gb=g$. Thus, there exists a unique morphism $\<f,g\>\colon C\to A\times B$ satisfying the following two equations:
\begin{align*}
&\<f,g\>\pi_A=\bar gf=cf=f      &&\<f,g\>\pi_B=\bar fg=cg=g
\end{align*}
Let us prove that $\<f,g\>\pi_{(A,a)}=f$ and $\<f,g\>\pi_{(B,b)}=g$:
\begin{align*}
\<f,g\>\pi_{(A,a)}&=~\<f,g\>\bar{\pi_Bb}\pi_Aa\\
&=~\bar{\<f,g\>\pi_Bb}\<f,g\>\pi_Aa              \Tag{\textbf{[R.4]}}\\
&=~\bar{gb}fa                                    \Tag{\<f,g\>\pi_B=g,\<f,g\>\pi_A=f}\\
&=~\bar gf                                       \Tag{gb=g,fa=f}\\
&=~f                                             \Tag{\bar g=c=\bar f,\textbf{[R.1]}}
\end{align*}
Similarly, one can show that $\<f,g\>\pi_{(B,b)}=g$. Suppose now that $h\colon(C,c)\to(A\times B,a\times b)$ satisfies the equations $h\pi_{(A,a)}=f$ and $h\pi_{(B,b)}=g$. However, since $h(a\times b)=h$, we can compute that: 
\begin{align*}
h\pi_{(A,a)}&=~h\bar{\pi_Bb}\pi_Aa\\
&=~h\bar{\pi_Bb}\pi_A\bar a                      \Tag{a=\bar a}\\
&=~h\bar{\pi_Bb}\bar{\pi_Aa}\pi_A                \Tag{\textbf{[R.4]}}\\
&=~h(a\times b)\pi_A                             \Tag{\bar{\pi_Bb}\bar{\pi_Aa}=a\times b}\\
&=~h\pi_A                                        \Tag{h(a\times b)=h}
\end{align*}
Similarly, one can show that $h\pi_{(B,b)}=h\pi_B$. Thus, by the universal property of the restriction product $A\times B$, $h=\<f,g\>$. This proves that $(A\times B,a\times b)$ is the product of $(A,a)$ and $(B,b)$ in $\L[\X]$. To prove that this is a local product, let us demonstrate its compatibility with the local structure. By Lemma~\ref{lemma:L-products}, we already know that the product $(A,\id_A)\times(B,\id_B)$ is equal to $(A\times B,\id_{A\times B})$. Thus:
\begin{align*}
&\Las(A,a)\times\Las(B,b)=(A,\id_A)\times(B,\id_B)=(A\times B,\id_{A\times B})\\
&\quad=\Las(A\times B,a\times b)=\Las((A,a)\times(B,b))
\end{align*}
Finally:
\begin{align*}
&\eta_{(A,a)\times(B,b)}=\eta_{(A\times B,a\times b)}=a\times b=\eta_{(A,a)}\times\eta_{(B,b)}
\end{align*}
Thus, $(A\times B,a\times b)$ is the local product of $(A,a)$ and $(B,b)$ in $\L[\X]$.
\end{proof}

We may now finally prove the main result of this section that a restriction (resp. local) category is Cartesian if and only if its associated local (resp. restriction) category is Cartesian. To this end, we denote by $\cRCat_\lax$ the $2$-subcategory of $\RCat_\lax$ spanned by Cartesian restriction categories and restriction functors which preserve restriction terminal objects and restriction products, and similarly we denote by $\cLCat_\lax$ the $2$-subcategory of $\LCat_\lax$ spanned by Cartesian local categories and local functors which preserve local terminal objects and local products. 

\begin{theorem}
\label{theorem:cartesian-local-categories}
Cartesian local categories are Cartesian objects in the $2$-category $\LCat_\lax$. Moreover, the $2$-equivalence of Theorem~\ref{theorem:2-equivalence-local-restriction} restricts to a $2$-equivalence $\cRCat_\lax\simeq\cLCat_\lax$. 
\end{theorem}
\begin{proof}
Let us first prove that Cartesian objects in $\LCat_\lax$ are precisely Cartesian local categories. Start by considering a Cartesian local category $\C$. By forgetting the local structure, we define the functors $\*\colon\1\to\C $ and $\times\colon\C\times\C\to\C$, which are right adjoints to the terminal functor $!\colon\C\to\1$ and the diagonal functor $\Delta\colon\C\to\C\times\C$ in the $2$-category $\Cat$ of categories, respectively. Since the $2$-morphisms of $\LCat_\lax$ are just natural transformations, we only need to check that $\*$ and $\times$ are local functors. First, notice that the terminal local category $\1$ is the terminal category $\{\1\}$ equipped with the trivial local structure (see Example~\ref{example:trivial-local}). Thus, by requiring $\*$ to be a local functor is to require the following equations to hold
\begin{align*}
&\Las(\*(\1))=\*(\Las(\1))=\*(\1)           &&\eta_{\*(\1)}=\*(\eta_\1)=\*(\id_\1)=\id_{\*(\1)}
\end{align*}
which is to say that the terminal object must be total in $\C$, that is, $\*=\*(\1)$ is a local terminal object. Note that the functors $\*$ and $\times$ necessarily preserve the pullbacks of the maximal inclusions since right adjoints preserve limits. Consider now the product functor $\times$. Requiring $\times$ to be local amounts to requiring the following equations to hold, for every $M$ and $N$ in $\C$
\begin{align*}
&\Las(M\times N)=(\Las M)\times(\Las N)     &&\eta_M\times\eta_N=\eta_{M\times N}
\end{align*}
which is precisely to say that $M\times N$ is a local product. Conversely, consider a Cartesian object $\C$ of $\LCat_\lax$. The forgetful $2$-functor $\LCat_\lax\to\Cat$ which forgets the local structure, preserves adjunctions, thus, the underlying category of $\C$ is Cartesian. Moreover, since the right adjoints $\*$ and $\times$ of the terminal local functor $!$ and of the diagonal local functor $\Delta$, respectively, are local functors, the terminal object of $\C$ must be a local terminal object and the products of $\C$ must be local products. This proves that Cartesian objects of $\LCat_\lax$ are precisely the Cartesian local categories. 

Finally, by~\cite[Proposition~6.1.7]{yau:2-categories}, every $2$-equivalence $\CC\simeq\DD$ between two categories preserves adjunctions, thus, in particular, it lifts to a $2$-equivalence between the $2$-categories of Cartesian objects and Cartesian preserving $1$-morphisms $\Cart[\CC]\simeq\Cart[\DD]$. Thus, the $2$-equivalence $\RCat_\lax\simeq\LCat_\lax$ of Theorem~\ref{theorem:2-equivalence-local-restriction} lifts to a $2$-equivalence $\Cart[\RCat_\lax]\simeq\Cart[\LCat_\lax]$. However, we have already shown that $\Cart[\LCat_\lax]$ is the $2$-category of Cartesian local categories. Finally, as discussed in~\cite[Section~4.1]{cockett:restrictionIII}, $\Cart[\RCat_\lax]$ is the $2$-category of Cartesian restriction categories.
\end{proof}

\begin{remark}
\label{remark:local-limits}
In~\cite{cockett:restrictionIII}, a full theory of restriction limits was explored. So far, we have not yet developed a full theory of local limits, that is, the correct notion of limits for local categories which would correspond to restriction limits under the $2$-equivalence $\RCat_\lax\simeq\LCat_\lax$. This turned out to be more subtle and complicated than one might expect, so we will leave this for future work. 
\end{remark}


%% file: CODE/7-split.tex
\section{Split local categories}
\label{section:split}
In this section, we introduce split local categories, which are those local categories that correspond to split restriction categories under the $2$-equivalence of Theorem~\ref{theorem:2-equivalence-local-restriction}. We begin by showing that split restriction idempotents in a restriction category are isomorphic to total objects in the associated local category. 

\begin{lemma}
\label{lemma:split-local-category}
Let $\X$ be a restriction category. A restriction idempotent $a\colon A\to A$ in $\X$ splits in $\X$ if and only if $(A,a)$ is isomorphic to a total object in $\L[\X]$.
\end{lemma}
\begin{proof}
Suppose that a restriction idempotent $a=\bar a$ splits as a section-retraction pair $(s,r)\colon E\leftrightarrows A$ in $\X$. We prove that $(A,a)$ is isomorphic to $(E,\id_E)$ in $\L[\X]$. To this end, let us first compute the restriction idempotent of $s$ and $r$. Since $s$ is monic, by Lemma~\ref{lemma:technical-restriction}(d), $\bar s=\id_E$. Moreover, we can compute:
\begin{align*}
\bar r&=~\bar{r\bar s}     \Tag{\bar s=\id_E}\\
&=~\bar{rs}                \Tag{\text{Lemma}~\ref{lemma:technical-restriction}~(c)}\\
&=~\bar a                  \Tag{rs=a}\\
&=~a                       \Tag{\bar a=a}
\end{align*}
Thus, since $sa=srs=s$, we get that $s\colon(E,\id_E)\to(A,a)$ and $r\colon(A,a)\to(E,\id_E)$ are morphisms of $\L[\X]$. Furthermore, since $sr=\id_E$ and $rs=a=\id_{(A,a)}$, $s$ inverts $r$ in $\L[\X]$, that is, $(A,a)$ is isomorphic to $(E,\id_E)$, which is a total object of $\L[\X]$. 

Conversely, assume that $(A,a)$ is isomorphic to a total object $(E,\id_E)$ of $\L[\X]$ via an isomorphism $s\colon(E,\id_E)\to(A,a)$ and its inverse $r\colon(A,a)\to(E,\id_E)$. In particular, this means that $sr=\id_E$ and that $rs=\id_{(A,a)}=a$. Thus, $(s,r)$ is a section-retraction pair and $a=rs$. Thus, $a$ splits in $\X$.
\end{proof}

\begin{remark}
\label{remark:iso-total}
It is important to notice that, despite the object $(A,a)$ of $\L[\X]$ being isomorphic to a total object when $a$ splits, the object $(A,a)$ itself is not total in $\L[\X]$. In particular, being total is not stable under isomorphisms.
\end{remark}

The previous lemma suggests the following definition in a local category. 

\begin{definition}
\label{definition:split-local-category}
An object $A$ in a local category $\C$ \textbf{splits} when $A$ is isomorphic to a total object, $A\cong\Las E$. A \textbf{split local category} is a local category for which all of its objects split. 
\end{definition}

\begin{example}
\label{example:split-local-set}
Unsurprisingly (see Example~\ref{example:split-restriction-set}), the local categories $\ParSet$ and $\ParTop$ of Examples~\ref{example:partial-set-local} and~\ref{example:open-local} split. Concretely, each object $(A,U)$ is isomorphic to $(U,U)$ by the identity on $U$.
\end{example}

Lemma~\ref{lemma:split-local-category} shows that if a restriction idempotent $a=\bar a\colon A\to A$ splits in a restriction category $\X$, then the object $(A,a)$ splits in $\L[\X]$, according to Definition~\ref{definition:split-local-category}. In fact, the main objective of this section is to show that the splitting condition in the local sense implies the splitting condition in the restriction sense and that this correspondence extends to a $2$-equivalence between split restriction categories and split local categories. To this end, we first prove a technical lemma.

\begin{lemma}
\label{lemma:monic-iso}
In any category, given two composable morphisms $f\colon A\to B$ and $g\colon B\to C$, if $fg$ is an isomorphism and $g$ is monic, then both $f$ and $g$ are isomorphisms.
\end{lemma}
\begin{proof} Denote $h\=fg$, which, by assumption, is an isomorphism. Let $f^{-1}\=gh^{-1}$. Thus $ff^{-1}=fgh^{-1}=hh^{-1}=\id_A$ and $f^{-1}fg=gh^{-1}h=g$. However, since $g$ is monic, this implies that $f^{-1}f=\id_B$, thus $f$ is an isomorphism. Finally, $g=f^{-1}h$ is a composition of isomorphisms, thus, $g$ is also an isomorphism.
\end{proof}

Thanks to the previous lemma, we can prove the following result.

\begin{lemma}
\label{lemma:split-local-category-2}
Let $\C$ be a local category. An object $U$ of $\C$ splits if and only if the restriction idempotent $(U,\eta_U)\colon A\nto A$ splits in $\R[\C]$, where $A=\Las U$. In particular, $\C$ is a split local category if and only if $\R[\C]$ is a split restriction category.
\end{lemma}
\begin{proof}
Assume that $U$ splits in $\C$. Thus, $U$ is isomorphic to a total object $E=\Las E$ via an isomorphism $s\colon E\to U$ and its inverse $r\colon U\to E$. Consider the morphisms $(U,r)\colon A\nto E$ and $(E,s\eta_U)\colon E\nto A$ of $\R[\C]$, where $A=\Las U$ and $\eta_U\colon U\to A$. Let us prove that $(U,r)$ and $(E,s\eta_U)$ define a splitting for $(U,\eta_U)\colon A\nto A$ in $\R[\C]$. For starters, let us compute the composition $(E,s\eta_U)(U,r)$:
\begin{equation*}
\begin{tikzcd}
E &&&& A && E \\
&&& U && U \\
&& E && U \\
&&& E
\arrow["{\eta_U}"', from=2-4, to=1-5]
\arrow["{\eta_U}", from=2-6, to=1-5]
\arrow["r"', from=2-6, to=1-7]
\arrow[equals, from=3-3, to=1-1]
\arrow["s"', from=3-3, to=2-4]
\arrow["\lrcorner"{anchor=center, pos=0.125, rotate=135}, draw=none, from=3-5, to=1-5]
\arrow[equals, from=3-5, to=2-4]
\arrow[equals, from=3-5, to=2-6]
\arrow["\lrcorner"{anchor=center, pos=0.125, rotate=135}, draw=none, from=4-4, to=2-4]
\arrow[equals, from=4-4, to=3-3]
\arrow["s"', from=4-4, to=3-5]
\end{tikzcd}
\end{equation*}
However, since $sr=\id_E$, we conclude that $(E,s\eta_U)(U,r)=\id_E$ in $\R[\C]$. We now compute the composition $(U,r)(E,s\eta_U)$ as follows:
\begin{equation*}
\begin{tikzcd}
A && E && A \\
& U && E \\
&& U
\arrow["{\eta_U}", from=2-2, to=1-1]
\arrow["r"', from=2-2, to=1-3]
\arrow[equals, from=2-4, to=1-3]
\arrow["{s\eta_U}"', from=2-4, to=1-5]
\arrow["\lrcorner"{anchor=center, pos=0.125, rotate=135}, draw=none, from=3-3, to=1-3]
\arrow[equals, from=3-3, to=2-2]
\arrow["r"', from=3-3, to=2-4]
\end{tikzcd}
\end{equation*}
However, since $rs=\id_U$, we conclude that $(U,r)(E,s\eta_U)=(U,\eta_U)$. Thus, $((E,s\eta_U),(U,r))$ forms a splitting for $(U,\eta_U)$ in $\R[\C]$.\\
Conversely, assume that $(U,\eta_U)$ splits in $\R[\C]$. Thus, there are two morphisms $(U_1,r')\colon A\nto E$ and $(U_2,s')\colon E\to A$ such that $(U_1,r')(U_2,s')=(U,\eta_U)$ and $(U_2,s')(U_1,r')=\id_E$. From these two equations, we deduce the existence of four morphisms of $\C$ as follows:
\begin{align*}
&p_1\colon U\to U_1             &&p_2\colon U\to U_2\\
&q_1\colon E\to U_1             &&q_2\colon E\to U_2
\end{align*}
satisfying the following conditions
\begin{equation*}
\begin{tikzcd}
U & {U_1} \\
& A
\arrow["{p_1}", from=1-1, to=1-2]
\arrow["{\eta_U}"', from=1-1, to=2-2]
\arrow["{\eta_1}", from=1-2, to=2-2]
\end{tikzcd}\qquad
\begin{tikzcd}
U & {U_2} \\
& A
\arrow["{p_2}", from=1-1, to=1-2]
\arrow["{\eta_U}"', from=1-1, to=2-2]
\arrow["{s'}", from=1-2, to=2-2]
\end{tikzcd}\qquad
\begin{tikzcd}
E & {U_1} \\
& E
\arrow["{q_1}", from=1-1, to=1-2]
\arrow[equals, from=1-1, to=2-2]
\arrow["{r'}", from=1-2, to=2-2]
\end{tikzcd}\qquad
\begin{tikzcd}
E & {U_2} \\
& E
\arrow["{q_2}", from=1-1, to=1-2]
\arrow[equals, from=1-1, to=2-2]
\arrow["{\eta_2}", from=1-2, to=2-2]
\end{tikzcd}
\end{equation*}
where, as a shorthand, we denoted by $\eta_1$ and $\eta_2$, $\eta_{U_1}$ and $\eta_{U_2}$, respectively. Furthermore, the following two diagrams are pullbacks in $\C$:
\begin{equation*}
\begin{tikzcd}
U & {U_2} \\
{U_1} & E
\arrow["{p_2}", from=1-1, to=1-2]
\arrow["{p_1}"', from=1-1, to=2-1]
\arrow["\lrcorner"{anchor=center, pos=0.125}, draw=none, from=1-1, to=2-2]
\arrow["{\eta_2}", from=1-2, to=2-2]
\arrow["{r'}"', from=2-1, to=2-2]
\end{tikzcd}\qquad
\begin{tikzcd}
E & {U_1} \\
{U_2} & A
\arrow["{q_1}", from=1-1, to=1-2]
\arrow["{q_2}"', from=1-1, to=2-1]
\arrow["\lrcorner"{anchor=center, pos=0.125}, draw=none, from=1-1, to=2-2]
\arrow["{\eta_1}", from=1-2, to=2-2]
\arrow["{s'}"', from=2-1, to=2-2]
\end{tikzcd}
\end{equation*}
From $q_2\eta_2=\id_E$, that $\eta_2$ is monic, and by Lemma~\ref{lemma:monic-iso}, we deduce that both $q_2$ and $\eta_2$ are isomorphisms in $\C$. However, since $p_1$ is a pullback of $\eta_2$ and $\eta_2$ is invertible, $p_1$ is also an isomorphism, since isomorphisms are stable under pullbacks. From the equation $p_2s'=\eta_U$, since $\eta_U$ is monic, so is $p_2$. However, since $p_1r'=p_2\eta_2$, that $p_1$ is invertible and that $\eta_2$ is monic, we conclude that $r'=p_1^{-1}p_2\eta_2$ is also monic. Finally, using that $q_1r'=\id_E$, that $r'$ is monic, and by Lemma~\ref{lemma:monic-iso}, $r'$ is an isomorphism in $\C$. Thus, since both $p_1\colon U\to U_1$ and $r'\colon U_1\to E$ are isomorphisms, $U$ is isomorphic to $E$ in $\C$. However, since $E$ is an object of $\R[\C]$, $E$ is total in $\C$.
\end{proof}

With this, we may now state the main result of this section. To this end, we denote by $\sLCat$ the full $2$-subcategory of $\LCat$ spanned by split local categories, and similarly, $\sRCat$ the full $2$-subcategory of $\RCat$ spanned by split restriction categories.

\begin{theorem}
\label{theorem:2-equivalence-local-restriction-split}
The $2$-equivalence $\RCat\simeq\LCat$ of Theorem~\ref{theorem:2-equivalence-local-restriction} restricts to a $2$-equivalence $\sRCat\simeq\sLCat$. 
\end{theorem}


%% file: CODE/8-inverse.tex
\section{Inverse local categories}
\label{section:inverse}
Inverse categories were introduced in~\cite{kastl:inverse-categories} as a multi-object version of inverse semigroups and inverse monoids. Concretely, an inverse category consists of a category for which every morphism $f\colon A\to B$ admits a morphism $f^\o\colon B\to A$ such that $ff^\o f=f$ and $f^\o ff^\o=f^\o$. As observed in~\cite[Section~2.3.2]{cockett:restrictionI}, inverse categories can be seen as a special class of restriction categories, where $ff^\o\colon A\to A$ defines the restriction idempotent of a morphism $f\colon A\to B$. Informally, inverse categories are the restriction version of a groupoid. Inverse categories find applications in computer science. In particular, they have been employed in~\cite{brett:inverse-computation} as a categorical language for inverse computation.
\par In this section, we introduce inverse local categories, which are those local categories that correspond to inverse categories under the $2$-equivalence of Theorem~\ref{theorem:2-equivalence-local-restriction}. For starters, recall the definition of an inverse category.

\begin{definition}~\cite[Section~2.3.2]{cockett:restrictionI}
\label{definition:inverse-restriction-category}
In a restriction category $\X$, a \textbf{restriction isomorphism} is a morphism $f\colon A\to B$ for which there exists a (necessarily unique) morphism $f^\o\colon B\to A$ such that $ff^\o=\bar f$ and $f^\o f=\bar{f^\o}$. An \textbf{inverse restriction category} is a restriction category $\X$ in which every morphism is a restriction isomorphism.
\end{definition}

\begin{remark}
\label{remark:partial-isomorphisms}
In the literature, restriction isomorphisms are also known as partial isomorphisms~\cite[Section~2.3.2]{cockett:restrictionI}. Moreover, often inverse restriction categories are simply known as inverse categories.
\end{remark}

\begin{example}
\label{example:inverse-restriction-set}
Consider the restriction category $\ParFnc$ of Example~\ref{example:partial-set-restriction}. A restriction isomorphism is a partial function which forms a bijection between its domain of definition and its image. Therefore, the subcategory of partial bijections forms an inverse restriction category.
\end{example}

We now introduce the analogue of inverse categories for local categories. 

\begin{definition}
\label{definition:inverse-local-category}
In a local category $\X$, a \textbf{local isomorphism} is a morphism $f\colon A\to B$ for which there exists an object $V\leq B$ and a morphism $f\up V\colon A\to V$ such that $(f\up V)m=f$ and $f\up V$ is an isomorphism, where $m\eta_B=\eta_V$. An \textbf{inverse local category} consists of a local category whose morphisms are all local isomorphisms.
\end{definition}

\begin{example}
\label{example:inverse-local-set}
In the local category $\ParSet$ of Example~\ref{example:partial-set-local}, a local isomorphism consists of a morphism $f\colon(A,U)\to(B,V)$ whose underlying function $f\colon U\to V$ is invertible. Therefore, the subcategory of those morphisms whose underlying functions are invertible, forms an inverse local category.
\end{example}

To justify this definition, let us show that restriction isomorphisms in a restriction category correspond to local isomorphisms in the associated local category.

\begin{lemma}
\label{lemma:inverse-local-categories}
A morphism $f\colon A\to B$ in a restriction category $\X$ is a restriction isomorphism if and only if $f\colon(A,\bar f)\to(B,\id_B)$ is a local isomorphism in $\L[\X]$. 
\end{lemma}
\begin{proof}
Assume that $f\colon A\to B$ is a restriction isomorphism. Thus, there is a morphism $f^\o\colon B\to A$ such that $ff^\o=\bar f$ and $f^\o f=\bar{f^\o}$. We compute the following
\begin{align*}
f\bar{f^\o}&=~\bar{ff^\o}f   \Tag{\textbf{[R.4]}}\\
&=~\bar ff                   \Tag{ff^\o=\bar f}\\
&=~f                         \Tag{\textbf{[R.1]}}
\end{align*}
and dually, we also have that $f^\o\bar f=f^\o$. Thus, we have two morphisms $f\colon(A,\bar f)\leftrightarrows(B,\bar{f^\o})\colon f^\o$ of $\L[\X]$. Moreover, $(B,\bar{f^\o})\leq(B,\id_B)$. Finally, from $ff^\o=\bar f$ and $f^\o f=\bar{f^\o}$, we deduce that $(A,\bar f)\cong(B,\bar{f^\o})$. Thus, $f$ becomes a local isomorphism in $\L[\X]$. 

Conversely, suppose that $f$ is a local isomorphism in $\L[\X]$. Thus, there exists an object $(B,b)\leq(B,\id_B)$ such that $(A,\bar f)\cong(B,b)$. Let $\tilde f\colon(A,\bar f)\to(B,b)$ be this isomorphism and $f^\o\colon(B,b)\to(A,\bar f)$ be its inverse. Now we have that $\tilde f\eta_{(B,b)}=f$, thus, $\tilde fb=f$. However, $\tilde fb=f$, thus, $\tilde f=f$ in $\X$. We also know that $b=\bar{f^\o}$ and that $f^\o\tilde f=\id_{(B,b)}$ and $\tilde ff^\o=\id_{(A,\bar f)}$, which precisely translates into the conditions $f^\o f=\bar{f^\o}$ and $ff^\o=\bar f$, respectively. Thus, $f$ is a restriction isomorphism in $\X$ as desired. 
\end{proof}

\begin{corollary}
\label{corollary:inverse-local-categories}
A restriction category $\X$ is an inverse category if and only if $\L[\X]$ is an inverse local category.
\end{corollary}

The main goal of this section is to show that the above corollary extends to a $2$-equivalence between inverse (restriction) categories and inverse local categories. For starters, let us prove that local isomorphisms in a local category $\C$ correspond to restriction isomorphisms in $\R[\C]$. To this end, let us prove a technical lemma.

\begin{lemma}
\label{lemma:pullback-with-monic-and-iso}
Consider a morphism $f\colon M\to N$ in a category such that there exists an object $P$ isomorphic to $M$ via an isomorphism $g\colon P\cong M$ and such that $f\colon M\to N$ factors through $g^{-1}$ as $f=g^{-1}m$ and $m\colon P\to N$ is monic. Then the following is a pullback diagram:
\begin{equation*}
\begin{tikzcd}
M & M \\
P & N
\arrow[equals, from=1-1, to=1-2]
\arrow["{g^{-1}}"', from=1-1, to=2-1]
\arrow["\lrcorner"{anchor=center, pos=0.125}, draw=none, from=1-1, to=2-2]
\arrow["f", from=1-2, to=2-2]
\arrow["m"', from=2-1, to=2-2]
\end{tikzcd}
\end{equation*}
\end{lemma}
\begin{proof}
Consider two morphisms $h\colon X\to M$ and $k\colon X\to P$ such that $hf=km$. Thus, the morphism $h\colon X\to M$ satisfies that $hg^{-1}m=hf=km$. However, since $m$ is monic, this implies that $hg^{-1}=k$. Finally, it is immediate to see that $h\colon X\to M$ is the unique morphism satisfying the universal property of the pullback diagram.
\end{proof}

\begin{lemma}
\label{lemma:inverse-local-categories-2}
A morphism $f\colon M\to N$ in a local category $\C$ is a local isomorphism if and only if $(M,f\eta_N)\colon\Las M\nto\Las N$ is a restriction isomorphism in $\R[\C]$. 
\end{lemma}
\begin{proof}
Suppose that $f\colon M\to N$ is a local isomorphism in $\C$. In particular, this means that there exists an object $P\leq N$ and an isomorphism $g\colon P\cong M$. The condition $P\leq N$ means that $\Las P=\Las N$ and that there exists a (necessarily unique) morphism $m\colon P\to N$ such that $m\eta_N=\eta_P$. Let us define $(M,f\eta_N)^\o$ as follows:
\begin{equation*}
(M,f\eta_N)^\o\colon
\begin{tikzcd}
{\Las N=\Las P} && {\Las M} \\
& P
\arrow["{\eta_P}", from=2-2, to=1-1]
\arrow["{g\eta_M}"', from=2-2, to=1-3]
\end{tikzcd}
\end{equation*}
We shall prove that $(M,f\eta_N)^\o$ is a partial inverse for $(M,f\eta_N)$ in $\R[\C]$. Let us start by working out the composition $(M,f\eta_N)(M,f\eta_N)^\o$:
\begin{equation*}
\begin{tikzcd}
{\Las M} &&& {\Las N} &&& {\Las M} \\
&& N && N \\
& M && N && P \\
&& M && P \\
&&& M
\arrow["{\eta_N}"', from=2-3, to=1-4]
\arrow["{\eta_N}", from=2-5, to=1-4]
\arrow["{\eta_M}", from=3-2, to=1-1]
\arrow["f"', from=3-2, to=2-3]
\arrow["\lrcorner"{anchor=center, pos=0.125, rotate=135}, draw=none, from=3-4, to=1-4]
\arrow[equals, from=3-4, to=2-3]
\arrow[equals, from=3-4, to=2-5]
\arrow["{g\eta_M}"', from=3-6, to=1-7]
\arrow["m", from=3-6, to=2-5]
\arrow[equals, from=4-3, to=3-2]
\arrow["f"', from=4-3, to=3-4]
\arrow["m", from=4-5, to=3-4]
\arrow[equals, from=4-5, to=3-6]
\arrow["\lrcorner"{anchor=center, pos=0.125, rotate=135}, draw=none, from=5-4, to=3-4]
\arrow[equals, from=5-4, to=4-3]
\arrow["{g^{-1}}"', from=5-4, to=4-5]
\end{tikzcd}
\end{equation*}
where we used Lemma~\ref{lemma:pullback-with-monic-and-iso} to compute the bottom pullback diagram. However, this span coincides with $\bar{(M,f\eta_N)}$. Now, consider the composition $(M,f\eta_N)^\o(M,f\eta_N)$:
\begin{equation*}
\begin{tikzcd}
{\Las N} &&& {\Las M} && {\Las N} \\
&& M \\
& P &&& M \\
&&& M \\
&& P
\arrow["{\eta_M}"', from=2-3, to=1-4]
\arrow["{\eta_P}", from=3-2, to=1-1]
\arrow["g"', from=3-2, to=2-3]
\arrow["{\eta_M}", from=3-5, to=1-4]
\arrow["{f\eta_N}"', from=3-5, to=1-6]
\arrow["\lrcorner"{anchor=center, pos=0.125, rotate=135}, draw=none, from=4-4, to=1-4]
\arrow[equals, from=4-4, to=2-3]
\arrow[equals, from=4-4, to=3-5]
\arrow["\lrcorner"{anchor=center, pos=0.125, rotate=135}, draw=none, from=5-3, to=2-3]
\arrow[equals, from=5-3, to=3-2]
\arrow["g"', from=5-3, to=4-4]
\end{tikzcd}
\end{equation*}
However, $gf\eta_N=gg^{-1}m\eta_N=m\eta_N=\eta_P$. Thus, $(M,f\eta_N)^\o(M,f\eta_N)=\bar{(M,f\eta_N)^\o}$. This proves that $(M,f\eta_N)$ is a restriction isomorphism in $\R[\C]$. 

Conversely, assume that $(M,f\eta_N)$ is a restriction isomorphism in $\R[\C]$. Thus, there exists a morphism $(P,f^\o)\colon\Las N\to\Las M$ such that $(M,f\eta_N)(P,f^\o)=\bar{(M,f\eta_N)}$ and $(P,f^\o)(M,f\eta_N)=\bar{(P,f^\o)}=(P,\eta_P)$. Recall that $P$ must satisfy $P\leq N$, thus, there exists an $m\colon P\to N$ such that $m\eta_N=\eta_P$. The first equation implies the existence of a morphism $h\colon M\to P$ such that $hf^\o=\eta_M$ and $h\eta_P=f\eta_N$. However, since $\eta_P=m\eta_N$ and that $\eta_N$ is monic, we obtain that $hm=f$. The second equation implies the existence of a second morphism $g\colon P\to M$ such that $gf\eta_N=\eta_P$ and $g\eta_M=f^\o$. Thus:
\begin{align*}
&gh\eta_M=hf^\o=\eta_M      &&hg\eta_P=gf\eta_N=\eta_N
\end{align*}
Using that $\eta_M$ and $\eta_N$ are monic, we conclude that $g$ and $h$ are inverse to each other. Thus, $f$ is a local isomorphism of $\C$.
\end{proof}

\begin{corollary}
\label{corollary:inverse-local-categories-2}
A local category $\C$ is an inverse local category if and only if $\R[\C]$ is an inverse category.
\end{corollary}

With this, we may finally state the main result of this section. We denote by $\iRCat_\lax$ the full $2$-subcategory of $\RCat_\lax$ spanned by inverse restriction categories, and similarly we denote by $\iLCat_\lax$ the full $2$-subcategory of $\LCat_\lax$ spanned by inverse local categories.

\begin{theorem}
\label{theorem:inverse-local-categories}
There is a $2$-equivalence $\iRCat_\lax\simeq\iLCat_\lax$. 
\end{theorem}


%% file: CODE/9-join.tex
\section{Join local categories}
\label{section:join}
Join restriction categories were introduced in~\cite{cockett:classical-restriction-categories}. Concretely, the join of two compatible morphisms $f,g\colon A\to B$ of a restriction category, that is, partially-defined morphisms that agree on the intersection of their domains, can be allegedly interpreted as the morphism obtained by gluing together $f$ and $g$ along the intersection of their domains. Join restriction categories were further studied in~\cite{guo:join-restriction}; they also found application in tangent category theory~\cite{cockett:tangent-cats,cockett:lie-groups-tangent-cats}.

\par In this section, we introduce join local categories, which are those local categories that correspond to join restriction categories under the $2$-equivalence of Theorem~\ref{theorem:2-equivalence-local-restriction}. We begin by recalling the notion of join restriction categories.

\begin{definition}~\cite[Definition~10.1]{cockett:classical-restriction-categories}
\label{definition:join-restriction-category}
In a restriction category $\X$ the \textbf{join} of a family $\Famly$ of compatible morphisms $f\colon A\to B$ of $\X$ is the (necessarily unique) morphism $\bigvee_{f\in\Famly}f\colon A\to B$ subject to the following conditions:
\begin{itemize}
\item For every $g\in\Famly$, $g\leq\bigvee_{f\in\Famly}f$;

\item If a morphism $h\colon A\to B$, satisfies $f\leq h$, for every $f\in\Famly$, then $\bigvee_{f\in\Famly}f\leq h$.
\end{itemize}
A \textbf{join restriction category} consists of a restriction category $\X$ subject to the following conditions:
\begin{description}
\item[JR.1] Every family $\Famly$ of compatible morphisms of $\X$ (including the empty family) admits a join;

\item[JR.2] The join of a family $\Famly$ of compatible morphisms is stable under pre-composition, explicitly:
\begin{align*}
&g\left(\bigvee_{f\in\Famly}f\right)=\bigvee_{f\in\Famly}(gf)
\end{align*}

\item[JR.3] The join of a family $\Famly$ of compatible morphisms is stable under post-composition, explicitly:
\begin{align*}
&\left(\bigvee_{f\in\Famly}f\right)h=\bigvee_{f\in\Famly}(fh)
\end{align*}
\end{description}
\end{definition}

\begin{example}
\label{example:join-restriction-set}
The restriction categories of Examples~\ref{example:partial-set-restriction} and~\ref{example:open-restriction} are join. In fact, a family $\{f_i\}$ of compatible morphisms in both the examples, is a family of partial (continuous) functions which satisfy the gluing condition, that is, $f_i$ and $f_j$ coincide on the intersection of their domains of definition. Therefore, the join of $\{f_i\}$ is the unique extension of each $f_i$, obtained by gluing those partial functions together.
\end{example}

As noticed by Cockett and Manes in~\cite[Lemma~6.10(2)]{cockett:classical-restriction-categories}, \textbf{[JR.3]} is redundant. We will prove a similar result for join local categories using the equivalence. First, let us introduce join local categories. To begin, let us adopt the following notation. Given a family $\Famly$ of compatible morphisms $f\colon M_f\to N$ of a local category $\C$, the diagram
\begin{equation*}
\begin{tikzcd}
{P\wedge M_f} & {M_f} \\
P & {\Las M_f}
\arrow["{g\up M_f}", from=1-1, to=1-2]
\arrow["{\pi_P^f}"', from=1-1, to=2-1]
\arrow["\lrcorner"{anchor=center, pos=0.125}, draw=none, from=1-1, to=2-2]
\arrow["{\eta_{M_f}}", from=1-2, to=2-2]
\arrow["g"', from=2-1, to=2-2]
\end{tikzcd}
\end{equation*}
denotes the pullback of $\eta_{M_f}$ along a morphism $g\colon P\to\Las M_f$.

\begin{definition}
\label{definition:join-local-category}
In a local category $\C$, the \textbf{join} of a family $\Famly$ of compatible morphisms is the (necessarily unique) morphism $\bigvee_{f\in\Famly}f\colon\bigvee_{f\in\Famly}M_f\to N$ subject to the following conditions:
\begin{itemize}
\item For every $f\in\Famly$, $f\leq\bigvee_{f\in\Famly}f$;

\item If $g\colon M\to N$ satisfies $f\leq g$, for every $f\in\Famly$, then $\bigvee_{f\in\Famly}f\leq g$.
\end{itemize}
A \textbf{join local category} consists of a local category $\C$ subject to the following conditions:
\begin{description}
\item[JL.1] Every family $\Famly$ of compatible morphisms of $\C$ (including the empty family) admits a join;

\item[JL.2] Given a family $\Famly$ of morphisms $f\colon M_f\to N$ of $\C$, the diagram
\begin{equation*}
\begin{tikzcd}
{\bigvee_{f\in\Famly}(P\wedge M_f)} && {\bigvee_{f\in\Famly}M_f} \\
P && {\Las\left(\bigvee_{f\in\Famly}M_f\right)}
\arrow["{\bigvee_{f\in\Famly}((g\up M_f)m_f)}", from=1-1, to=1-3]
\arrow["{\bigvee_{f\in\Famly}\pi_P^f}"', from=1-1, to=2-1]
\arrow["{\eta_{\bigvee_{f\in\Famly}M_f}}", from=1-3, to=2-3]
\arrow["g"', from=2-1, to=2-3]
\end{tikzcd}
\end{equation*}
is a pullback in $\C$, for every $g$, where $m_f\colon M_f\to\bigvee_{f\in\Famly}M_f$ denotes the unique monic that commutes with the maximal inclusions.
\end{description}
\end{definition}

\begin{example}
\label{example:join-local-set}
In the local categories $\ParSet$ and $\ParTop$ of Examples~\ref{example:partial-set-local} and~\ref{example:open-local}, the join of a family of compatible morphisms $f_i\colon(A,U_i)\to(B,V)$ consists of the morphism $\bigvee_if_i\colon(A,\bigvee_iU_i)\to(B,V)$ defined as follows. $\bigvee_iU_i$ is the union of all $U_i$ and the underlying function of $\bigvee_if_i$ is defined by gluing together the functions $f_i$, which, by assumption, satisfy the gluing condition, that is, for each $i,j$, $f_i$ and $f_j$ coincide on the intersection $U_i\cap U_j$ (see Example~\ref{example:compatibility-morphisms-local-par-set}).
\end{example}

In Section~\ref{section:concepts-local-categories}, we constructed a partial order on morphisms of a local category and showed a correspondence between compatibility and partial order in the local sense with compatibility and partial order in the restriction sense. In this section, we use this correspondence to prove that joins in the restriction sense correspond to joins in the local sense.

\begin{lemma}
\label{lemma:join-local-category}
Let $\X$ be a restriction category. A family $\Famly$ of compatible morphisms $f\colon A\to B$ of $\X$ admits a join if and only if the family of morphisms $f\colon(A,\bar f)\to(B,\id_B)$ admits a join in the local category $\L[\X]$. 
\end{lemma}
\begin{proof}
Suppose that $\Famly$ admits a join $\bigvee_ff\colon A\to B$ in $\X$. Consider the morphism
\begin{align*}
&\bigvee_ff\colon(A,\bigvee_f\bar f)\to(B,\id_B)
\end{align*}
$\L[\X]$, which is well-defined since the restriction operation commutes with joins, that is, $\overline{\bigvee_ff} = \bigvee_f\bar f$. Moreover, for every $g\in\Famly$, $(M,\bar g)\leq(M,\bigvee_ff)$ since $g\leq\bigvee_ff$ and thus, $\bar g\leq\bigvee_f\bar f$. Now, consider $h\colon(A,\bar h)\to(B,\id_B)$ such that, for every $f\in\Famly$, $f\leq h$ in $\L[\X]$. Thus, by Lemma~\ref{lemma:order-matters}, $f\leq h$ in $\X$. Thus, $\bigvee_ff\leq h$ in $\X$, which implies that $\bigvee_ff\leq h$ in $\L[\X]$.
\par Conversely, assume there exists a join $\bigvee_ff\colon(A,a)\to(B,\id_B)$ of the family of morphisms $f$ in $\L[\X]$. Thus $\bar{\bigvee_ff}=a$. Moreover, since each $g\in\Famly$ satisfies $g\leq\bigvee_ff$ in $\L[\X]$, by Lemma~\ref{lemma:order-matters}, $g\leq\bigvee_ff$ in $\X$. Now, consider $h\colon M\to N$ such that $f\leq h$ in $\X$ for each $f\in\Famly$. Thus, $f\leq h$ in $\L[\X]$, where $h\colon(A,\bar h)\to(B,\id_B)$. Thus, $\bigvee_ff\leq h$ in $\L[\X]$, which implies that $\bigvee_ff$ is a join of $\Famly$ in $\X$.
\end{proof}

Next, we prove a technical lemma that will come in handy when we will prove the equivalence between join restriction and join local categories.

\begin{lemma}
\label{lemma:join-local-categories-post-composition}
In a local category $\C$ satisfying {\normalfont\textbf{[JL.1]}}, joins are stable under post-composition, that is:
\begin{align*}
&\left(\bigvee_{f\in\Famly}f\right)h=\bigvee_{f\in\Famly}(fh)
\end{align*}
\end{lemma}
\begin{proof}
Consider a family $\Famly$ of compatible morphisms $f\colon M_f\to N$ and a morphism $h\colon N\to P$ of $\C$. For every $f\in\Famly$, $f\leq\bigvee_{f\in\Famly}f$, that is, there exists a monic $m_f\colon M_f\to\bigvee_{f\in\Famly}M_f$ that commutes with the maximal inclusions and such that $m_f\bigvee_{f\in\Famly}f=f$. However, by postcomposing by $h$, we also obtain $m_f(\bigvee_{f\in\Famly}f)h=fh$. Therefore, for every $f\in\Famly$, $fh\leq(\bigvee_{f\in\Famly}f)h$. By the properties of the join, this implies that $\bigvee_{f\in\Famly}(fh)\leq(\bigvee_{f\in\Famly}f)h$. Therefore, there exists a unique monic $m$ commuting with the maximal inclusions such that $m(\bigvee_{f\in\Famly}f)h=\bigvee_{f\in\Famly}(fh)$. However, the domains of $\bigvee_{f\in\Famly}(fh)$ and $(\bigvee_{f\in\Famly}f)h$ coincide with $\bigvee_{f\in\Famly}M_f$. Furthermore, for every object $Q$ of $\C$, there is exactly one monic $Q\to Q$ that commutes with the maximal inclusions: the identity. Therefore, $m$ is the identity, and thus, we obtain the desired identity $(\bigvee_{f\in\Famly}f)h=\bigvee_{f\in\Famly}(fh)$.
\end{proof}

\begin{corollary}
\label{corollary:join-local-category}
The local category $\L[\X]$ of a join restriction category $\X$ is a join local category.
\end{corollary}
\begin{proof}
Suppose that $\X$ is a join restriction category and consider a family $\Famly$ of compatible morphisms $f\colon(A,\bar f)\to(B,b)$ of $\L[\X]$. Thus, the family of morphisms $f\eta_{(B,b)}\colon(A,\bar f)\to(B,\id_B)$ admits a join in $\L[\X]$. Let $\bigvee_ff\eta_{(B,b)}$ denote the join. To construct the join of $\Famly$, consider the following pullback diagram
\begin{equation*}
\begin{tikzcd}
{(A,\bigvee_f\bar fb)} & {(B,b)} \\
{(A,\bigvee_f\bar fb)} & {(B,\id_B)}
\arrow["{\bigvee_ffb}", from=1-1, to=1-2]
\arrow["{\bigvee_f\bar fb}"', from=1-1, to=2-1]
\arrow["\lrcorner"{anchor=center, pos=0.125}, draw=none, from=1-1, to=2-2]
\arrow["{\eta_{(B,b)}}", from=1-2, to=2-2]
\arrow["{\bigvee_ffb}"', from=2-1, to=2-2]
\end{tikzcd}
\end{equation*}
where we used that $bb=b$ and that joins commute with composition. Consider a morphism $h\colon(A,\bigvee_f\bar fb)\to(B,b)$ of $\L[\X]$ such that $f\leq h$ in $\L[\X]$ for every $f\in\Famly$. It is not difficult to see that $f\eta_{(B,b)}\leq h\eta_{(B,b)}$, which implies that $\bigvee_f\bar fb\eta_{(B,b)}\leq h\eta_{(B,b)}$. Then using that $\eta$ is monic, we prove that $\bigvee_f\bar fb\leq h$, thus, $\L[\X]$ verifies \textbf{[JL.1]}. Now, let us prove \textbf{[JL.2]}. Thanks to Lemma~\ref{lemma:eta-display}, the diagram 
\begin{equation*}
\begin{tikzcd}
{(C,\bar{g\bigvee_f\bar f})} & {(A,\bigvee_f\bar f)} \\
{(C,c)} & {(A,\id_A)}
\arrow["{g\bigvee_f\bar f}", from=1-1, to=1-2]
\arrow["{\bar{g\bigvee_f\bar f}}"', from=1-1, to=2-1]
\arrow["\eta", from=1-2, to=2-2]
\arrow["g"', from=2-1, to=2-2]
\end{tikzcd}
\end{equation*}
is a pullback in $\L[\X]$. However, by \textbf{[JR.2]}, $g\bigvee_f\bar f=\bigvee_fg\bar f$. Moreover, the restriction operator commutes with joins. Thus, the diagram reduces to the following pullback
\begin{equation*}
\begin{tikzcd}
{(C,\bigvee_f\bar{g\bar f})} & {(A,\bigvee_f\bar f)} \\
{(C,c)} & {(A,\id_A)}
\arrow["{\bigvee_fg\bar f}", from=1-1, to=1-2]
\arrow["{\bigvee_f\bar{g\bar f}}"', from=1-1, to=2-1]
\arrow["\lrcorner"{anchor=center, pos=0.125}, draw=none, from=1-1, to=2-2]
\arrow["\eta", from=1-2, to=2-2]
\arrow["g"', from=2-1, to=2-2]
\end{tikzcd}
\end{equation*}
which is precisely the pullback of \textbf{[JL.2]}.
\end{proof}

Let us now prove the converse, in that local joins in a local category correspond to restriction joins in the associated restriction category. 

\begin{lemma}
\label{lemma:join-local-category-2}
A family $\Famly$ of compatible morphisms $f\colon M_f\to P$ of a local category $\C$ admits a join if and only if the family of morphisms $(M_f,f\eta_P)\colon\Las M_f\nto\Las P$ admits a join in $\R[\C]$. 
\end{lemma}
\begin{proof}
Suppose that a family $\Famly$ of compatible morphisms of $\C$ admits a join and let $\bigvee_ff\colon\bigvee_fM_f\to P$ be the join of $\Famly$. Consider the following morphism
\begin{align*}
&\left(\bigvee_fM_f,\bigvee_ff\eta_P\right)\colon\Las M=\Las M_f\nto\Las P
\end{align*}
in $\R[\C]$. Let us show that $\R[\bigvee_ff]$ is the join of $\Famly'\=\{(M_f,f\eta_P)\}_{f\in\Famly}$. For starters, since for each $g\in\Famly$, $g\leq\bigvee_ff$, by Lemma~\ref{lemma:order-matters-2}, it follows that $(M_g,g\eta_P)\leq(\bigvee_fM_f,\bigvee_ff\eta_P)$. Now, consider $(U,h)\colon\Las M\to\Las P$ such that, for every $f\in\Famly$, $(M_f,f\eta_P)\leq h$ in $\R[\C]$. Thus, $\bar{(M_f,f\eta_P)}h=(M_f,f\eta_P)$, that is, the span
\begin{equation*}
\begin{tikzcd}
{\Las M} && {\Las M} && {\Las P} \\
& {M_f} && U \\
&& {M_f\wedge U}
\arrow["{\eta_{M_f}}", from=2-2, to=1-1]
\arrow["{\eta_{M_f}}"', from=2-2, to=1-3]
\arrow["{\eta_U}", from=2-4, to=1-3]
\arrow["h"', from=2-4, to=1-5]
\arrow["\lrcorner"{anchor=center, pos=0.125, rotate=135}, draw=none, from=3-3, to=1-3]
\arrow["{\pi_{M_f}}", from=3-3, to=2-2]
\arrow["{m_f}"', from=3-3, to=2-4]
\end{tikzcd}
\end{equation*}
must coincide with $(M_f,f\eta_P)$. Thus, using that $\eta_{M_f}$ is monic, we conclude that $M_f\wedge U=M_f$, that $\pi_{M_f}=\id_{M_f}$, and that $m_f\colon M_f\to U$ satisfies $m_f\eta_U=\eta_{M_f}$ and $m_fh=f$, thus, $f\leq h$ in $\C$. Thus, $\bigvee_ff\leq h$ in $\C$ and, by Lemma~\ref{lemma:order-matters-2}, $(\bigvee_fM_f,\bigvee_ff)\leq(U,h)$ in $\R[\C]$. 

Conversely, assume that $\Famly'$ admits a join $\bigvee_f(M_f,f\eta_P)=(M,\tilde f)$ in $\R[\C]$. In particular, for every $f\in\Famly$, $(M_f,f\eta_P)\leq\bigvee_f(M_f,f\eta_P)$, which, by Lemma~\ref{lemma:order-matters-2}, implies that $f\leq\tilde f$ in $\C$, since $(M,\tilde f)=\bigvee_f(M_f,f\eta_P)$. Let $h$ in $\C$ such that $f\leq h$, for each $f\in\Famly$. Thus, $(M_f,f\eta_P)\leq(U,h)$ in $\R[\C]$. Thus, $\bigvee_f(M_f,f\eta_P)\leq(U,h)$, which implies that $\tilde f\leq h$ in $\C$. Therefore, $\tilde f$ is a join of $\Famly$ in $\C$.
\end{proof}

\begin{corollary}
\label{corollary:join-local-category-2}
The restriction category $\R[\C]$ of a join local category $\C$ is a join restriction category.
\end{corollary}
\begin{proof}
Let $\C$ be a join local category. We already shown that $\R[\C]$ admits joins, that is, $\R[\C]$ verifies \textbf{[JR.1]}. We shall prove that $\R[\C]$ verifies \textbf{[JR.2]}, as well. Consider a family of compatible morphisms $(M_f,f)\colon M\nto N$ of $\R[\C]$ and let $(V,g)\colon P\to M$ be a morphism. Thanks to \textbf{[JL.2]}, the composition of $(V,g)$ with $\bigvee_f(M_f,f)$ consists of the span:
\begin{equation*}
\begin{tikzcd}
P && M && N \\
& V && {\bigvee_fM_f} \\
&& {\bigvee_f(V\wedge M_f)}
\arrow["{\eta_V}", from=2-2, to=1-1]
\arrow["g"', from=2-2, to=1-3]
\arrow["\eta", from=2-4, to=1-3]
\arrow["{\bigvee_ff}"', from=2-4, to=1-5]
\arrow["\lrcorner"{anchor=center, pos=0.125, rotate=135}, draw=none, from=3-3, to=1-3]
\arrow[from=3-3, to=2-2]
\arrow["{\bigvee_f((g\up M_f)m_f)}"'{pos=1}, from=3-3, to=2-4]
\end{tikzcd}
\end{equation*}
However, by Lemma~\ref{lemma:join-local-categories-post-composition}, we can compute
\begin{align*}
&\left(\bigvee_f((g\up M_f)m_f)\right)\left(\bigvee_ff\right)=\bigvee_f\left((g\up M_f)m_f\bigvee_ff\right)=\bigvee_f((g\up M_f)f)
\end{align*}
where we used that $m_f\bigvee_ff=f$. However, the join of the family $\{(V,g)(M_f,f)\}_f$ is precisely the morphism $(V,\bigvee_f((g\up M_f)f))$. Thus, $\R[\C]$ verifies \textbf{[JR.2]}.
\end{proof}

Combining the above corollaries together, we may state the main result of this section. To this end, we denote by $\jRCat_\lax$ the $2$-subcategory of $\RCat_\lax$ of join restriction categories, restriction functors which preserve joins, and restriction natural transformations. Similarly, we denote by $\jLCat_\lax$ the $2$-subcategory of $\LCat_\lax$ of join local categories, join-preserving local functors, and local natural transformations.

\begin{theorem}
\label{theorem:join-local-categories}
There is a $2$-equivalence $\jRCat_\lax\simeq\jLCat_\lax$. 
\end{theorem}

We mentioned before that \textbf{[JR.3]} in Definition~\ref{definition:join-restriction-category} comes for free from the other axioms. For this reason, in our definition of a join local category, we omitted the corresponding axiom. In light of the $2$-equivalence of Theorem~\ref{theorem:join-local-categories}, we can now state the equivalent axiom for a join local category as a result.

\begin{corollary}
\label{corollary:join-local-categories-missing-axiom}
In a join local category $\C$, given a family $\Famly$ of morphisms $f\colon M_f\to\Las N$ of $\C$, the diagram
\begin{equation*}
\begin{tikzcd}
{\bigvee_{f\in\Famly}(M_f\wedge N)} && N \\
{\bigvee_{f\in\Famly}M_f} && {\Las N}
\arrow["{\bigvee_{f\in\Famly}\pi_N^f}", from=1-1, to=1-3]
\arrow["{\bigvee_{f\in\Famly}(\pi_{M_f}m_f)}"', from=1-1, to=2-1]
\arrow["{\eta_N}", from=1-3, to=2-3]
\arrow["{\bigvee_{f\in\Famly}f}"', from=2-1, to=2-3]
\end{tikzcd}
\end{equation*}
is a pullback in $\C$.
\end{corollary}


%% file: CODE/10-partial.tex
\section{Partial categories: an operational approach to partiality}
\label{section:partial-categories}
In Section~\ref{section:local-categories}, we introduced local categories as a new categorical framework to encode partiality. We made a case that local categories are suitable to describe \textit{partiality on objects}, in contrast with restriction categories, which are designed to encode \textit{partiality on morphisms}. In this section, we introduce a third approach to reasoning around partiality: partial categories. The objects and morphisms of a partial category shall not be interpreted as partial objects or partial morphisms, instead, the objects should be regarded as resources whose access is not constrained and morphisms as processes that consume their input resources in their entirety. Partiality in a partial category is instead captured by a partial order $\leq$ on the objects and by two operators, \textbf{restriction} and \textbf{contraction}, that control the access of the processes to the correct amount of resources.

\begin{definition}
\label{definition:partial-category}
A \textbf{partial structure} on a category $\C$ consists of an order $\leq$ on objects together with an operator, called \textbf{restriction}
\begin{prooftree}
\AxiomC{$U\leq A,\quad f\colon A\to B$}
\LeftLabel{[$\:\down\:$]}
\UnaryInfC{$f\down U\colon U\to B$}
\end{prooftree}
satisfying the following axioms:
\begin{description}
\item[P.1] For every $f\colon A\to B$:
\begin{align*}
&f\down A=f
\end{align*}

\item[P.2] For every $f\colon A\to B$ and $V\leq U\leq A$:
\begin{align*}
&(f\down U)\down V=f\down V
\end{align*}

\item[P.3] For every $f\colon A\to B$, $g\colon B\to C$, and $U\leq A$:
\begin{align*}
&(f\down U)g=(fg)\down U
\end{align*}
\end{description}
and with another operator, called \textbf{contraction}
\begin{prooftree}
\AxiomC{$V\leq B,\quad f\colon A\to B$}
\LeftLabel{[$\:\up\:$]}
\UnaryInfC{$A\up_fV\leq A,\quad f\up V\colon A\up_fV\to V$}
\end{prooftree}
satisfying the following axioms:
\begin{description}
\item[P.4] For every $f\colon A\to B$:
\begin{align*}
&A\up_fB=A                  &&f\up B=f
\end{align*}

\item[P.4'] For every $U\leq A$:
\begin{align*}
&A\up_{\id_A}U=U            &&\id_A\up U=\id_U
\end{align*}

\item[P.5] For every $f\colon A\to B$ and $W\leq V\leq B$:
\begin{align*}
&(A\up_fV)\up_{f\up V}W=A\up_fW             &&(f\up V)\up W=f\up W
\end{align*}

\item[P.6] For every $f\colon A\to B$, $g\colon B\to C$, and $W\leq C$:
\begin{align*}
&A\up_f(B\up_gW)=A\up_{fg}W         &&(f\up(B\up_gW))(g\up W)=(fg)\up W
\end{align*}
\end{description}
Furthermore, restriction and contraction must also satisfy the following coherence axioms:
\begin{description}
\item[P.7] For every $f\colon A\to B$ and $V\leq B$:
\begin{align*}
&(A\up_fV)\up_{f\down(A\up_fV)}V=A\up_fV        &&(f\down(A\up_fV))\up V=f\up V
\end{align*}

\item[P.8] For every $f\colon A\to B$, $g\colon B\to C$, and $V\leq B$:
\begin{align*}
&(fg)\down(A\up_fV)=(f\up V)(g\down V)
\end{align*}

\item[P.9] For every $U,V\leq A$ and $m\colon U\to V$, if $m(\id_A\down U)=\id_A\down V$ then $U\leq V$ and $m=\id_U\down V$\footnote{This axiom was added to the definition based on the suggestion from Tim Stokes about adding axiom \textbf{[I.6]} in Definition \ref{definition:inclusion-category}.}.
\end{description}
We refer to $f\down U$ as the \textbf{restriction of $f$ to $U$}, and to $A\up_fV$ and $f\up V$ as the \textbf{contraction of $A$ to $V$ along $f$} and the \textbf{contraction of $f$ to $V$}, respectively. A \textbf{partial category} is a category equipped with a partial structure.
\end{definition}

The restriction operator \textit{restricts} the access of a process $f\colon A\to B$ to a sub-resource $U\leq A$ of $A$, while the contraction operator \textit{contracts} a process $f\colon A\to B$ to a new process $f\up V$ whose output lands in a sub-resource $V\leq B$ of $B$. In particular, the contraction $A\up_fV$ of $A$ to $V$ along $f$ can be allegedly interpreted as the pre-image of $V$ along $f$. By \textbf{[P.1]}, restricting a morphism to its domain gives back the original morphism. \textbf{[P.2]} establishes that restriction is an idempotent operation: by restricting a morphism $f\colon A\to B$ first to an object $U\leq A$ and then restricting the restriction to a $V\leq U$ is equivalent to restricting $f$ directly to $V$. \textbf{[P.3]} establishes that the restriction of the composition of two morphisms is equal to the composition of the restriction of the first one by the second one. By \textbf{[P.4]}, contracting a morphism to its codomain gives back the original domain and the original morphism, while by \textbf{[P.4']}, the contraction of the identity of $A$ to $U\leq A$ gives back $U$ and the identity on $U$. \textbf{[P.5]} says that the contraction is an idempotent operation, that is, that by contracting first a morphism $f\colon A\to B$ to $V\leq B$ and then contracting the contraction to $W\leq V$, it is equivalent to contracting $f$ directly to $W$. \textbf{[P.6]} establishes that contracting the composition of two morphisms to $V\leq C$ is equivalent to the composition of the contraction of the two morphisms, the second one to $V$ and the first one to the contraction $B\up_gV$ of $B$ onto $V$ along $g$. Moreover, by \textbf{[P.7]}, contracting $f$ to any $V\leq B$ coincides with contracting to $V$ the restriction $f\down(A\up_fV)$ of $f$ along the contraction $A\up_fV$ of $A$. \textbf{[P.8]} establishes that the restriction of the composition of two morphisms $f\colon A\to B$ and $g\colon B\to C$ to the contraction of $A$ onto $V\leq B$ along $f$ coincides with the composition of the contraction of $f$ onto $V$ with the restriction of $g$ to $V$. Finally, \textbf{[P.9]} establishes that, for $U,V\leq A$, if the composition of a morphism $m\colon U\to V$ with the restriction of the identity of $B$ onto to $V$ coincides with the restriction of $\id_B$ onto $U$, then $U\leq V$ and $m$ is necessarily the restriction of $\id_V$ onto $U$. Here are now some examples of partial categories. 

\begin{example}
\label{example:trivial-partial}
Every category $\C$ is trivially a partial category in which $A\leq B$ if and only if $A=B$ and in which restriction and contraction are just identities.
\end{example}

\begin{example}
\label{example:partial-set-partial}
$\ParSet$ (Example~\ref{example:partial-set-local}) is a partial category in which $(A,W)\leq(A',U)$ if and only if $A=A'$ and $W$ is a subset of $U$, that is, $W\subseteq U$. The restriction operator sends a morphism $f\colon(A,U)\to(B,V)$, explicitly, a function $f\colon U\to V$, and an object $(A,W)\leq(A,U)$ to the restriction $f\down W\=f|_W\colon W\to V$ of $f$ to $W$. The contraction operator sends a morphism $f\colon(A,U)\to(B,V)$ and an object $(B,W)\leq(B,V)$ to $f\up(B,W)\colon(A,U)\up_f(B,W)\to(B,W)$, where:
\begin{align*}
&(A,U)\up_f(B,W)\=(A,\{u\in U,f(u)\in W\})&&f\up(B,W)\colon(A,U)\up_f(B,W)\ni u\mapsto f(u)\in W
\end{align*}
\end{example}

\begin{example}
\label{example:set-partial}
$\Set$ is a partial category in which $U\leq A$ if and only if $U$ is a subset of $A$, that is, $U\subseteq A$, and in which the restriction operator sends a function $f\colon A\to B$ and a subset $U$ of $A$ to the restriction $f\down U\=f|_U\colon U\to B$ of $f$ to $U$, while the contraction operator sends a function $f\colon A\to B$ and a subset $V$ of $B$ to $f\up V\colon A\up_fV\to V$, where:
\begin{align*}
&A\up_fV\=\{a\in A,f(a)\in V\}&&f\up V\colon A\up_fV\ni a\mapsto f(a)\in V
\end{align*}
This example shows that not every partial category is a local category (in a non-trivial way).
\end{example}

\begin{example}
\label{example:partial-topological-partial}
$\ParTop$ (Example~\ref{example:open-local}) is a partial category whose partial structure is defined as in Example~\ref{example:partial-set-partial}.
\end{example}

\begin{example}
\label{example:topological-partial}
$\Top$ is a partial category whose partial structure is defined as in Example~\ref{example:set-partial}.
\end{example}

\begin{example}
\label{example:semigroup-partial}
Consider a restriction semigroup $A$~\cite{jones:restriction-semigroups}, that is, a semigroup $A$ equipped with an operator that sends every element $a$ to another element $\bar a$ of $A$ satisfying the same axioms \textbf{[R.1]-[R.4]} of Definition~\ref{definition:restriction-category}. Construct the following category $\P[A]$ as in Example~\ref{example:local-monoid} (notice that this time there is no unit). The objects are the restriction idempotents $a=\bar a$ of $A$, and a morphism $x\colon a\to b$ is an element $x$ of $A$ such that $\bar x=a$ and $xb=x$. In particular, the identities are the morphisms $a\colon a\to a$. $\P[A]$ comes with a partial structure defined as follows. For starters, $a\leq b$ when $ab=a$. Given a morphism $x\colon a\to b$ and $a'\leq a$, the restriction $x\down a'$ of $x$ to $a'$ is the morphism $a'x\colon a'\to b$; the contraction of $x\colon a\to b$ to $b'\leq b$ is the morphism $xb'\colon ab'\to b'$.
\end{example}

\begin{example}
\label{example:ring-partial}
Let $\IdemRing^\bullet$ be the category of pairs $(R,e_R)$ formed by a non-unital commutative ring $R$ and an idempotent element $e_R\in R$ of $R$, explicitly, $e_R^2=e_R$. A morphism $f\colon(S,e_S)\to(R,e_R)$ of $\IdemRing^\bullet$ is a non-unital ring morphism $f\colon S\to R$ which satisfies the following equations
\begin{align*}
&f(se_S)=f(s)=e_Rf(s)
\end{align*}
for every $s\in S$. The identity morphism of an object $(R,e_R)$ of $\IdemRing^\bullet$ is the morphism which sends each $r\in R$ to $e_Rr$. Thus, ${\IdemRing^\bullet}^\op$ is a partial category, that is $\IdemRing^\bullet$ is a copartial category where $(R,e_R)\leq(S,e_S)$ if and only if $R=S$ and $e_Se_R=e_S$. The corestriction operator sends a morphism $f\colon(S,e_S)\to(R,e_R)$ and $(R',e_R')\leq(R,e_R)$ to $f\down(R',e_R')\colon(S,e_S)\to(R',e_R')$ defined as follows:
\begin{align*}
&(f\down(R,e_R'))(s)\=e_R'f(s)
\end{align*}
The cocontraction operator sends a morphism $f\colon(S,e_S)\to(R,e_R)$ and $(S,e_S')\leq(S,e_S)$ to $(R,e_R)\up_f(S,e_S')\=(R,e_Rf(e_S'))$ and $f\up(R,e_R)\colon(R,e_Rf(e_S'))\to(S,e_S')$ defined as follows:
\begin{align*}
&(f\up(R,e_R))(s)\=f(e_S's)
\end{align*}
\end{example}

 We now prove two important lemmas about partial categories. The first lemma establishes that the restriction operator is equivalent to precomposing morphisms by a special family of monics. The second lemma establishes that the contraction operator in a partial category is equivalent to taking the pullback of morphisms along those monics.

\begin{lemma}
\label{lemma:monic-partial-categories}
In a partial category $\C$, $U\leq A$ if and only if there exists a unique monic $m\colon U\to A$ such that, for any $f\colon A\to B$, $f\down U=mf$. In particular, $U\up_mU=U$ and $m\up U=\id_U$.
\end{lemma}
\begin{proof}
Let $m\=\id_A\down U\colon U\to A$. Notice that for every $f\colon A\to B$:
\begin{align*}
mf&=~(\id_A\down U)f\\
&=~(\id_Af)\down U               \Tag{\textbf{[P.3]}}\\
&=~f\down U
\end{align*}
Moreover, if $m'\colon U\to A$ satisfies the equation $m'f=f\down U$ for every $f$, then, by taking $f$ to be $\id_A$, we obtain:
\begin{align*}
&m'=m'\id_A=\id_A\down U=m
\end{align*}
Let us prove that $m$ is monic. To this end, we first prove that $U=U\up_mU$. By \textbf{[P.4']}, $U=A\up_{\id_A}U$. Moreover, we can compute
\begin{align*}
&U=A\up_{\id_A}U=(A\up_{\id_A}U)\up_{\id_A\down(A\up_{\id_A}U)}U=U\up_mU
\end{align*}
where we used \textbf{[P.7]} for $(A\up_{\id_A}U)\up_{\id_A\down(A\up_{\id_A}U)}U=A\up_{\id_A}U$ and \textbf{[P.8]} and \textbf{[P.4']} to compute that:
\begin{align*}
&\id_A\down(A\up_{\id_A}U)=(\id_A\id_A)\down(A\up_{\id_A}U)=(\id_A\up U)(\id_A\down U)=\id_U(\id_A\down U)=\id_A\down U=m
\end{align*}
Now, consider $f,g\colon C\to A$ such that $fm=gm$. Let us compute the following:
\begin{align*}
(fm)\up U&=~f\up(U\up_mU)(m\up_mU)          \Tag{\textbf{[P.6]}}\\
&=~f\up(U\up_mU)                            \Tag{\textbf{[P.1]}}\\
&=~f\up U                                   \Tag{U\up_mU=U}\\
&=~f                                        \Tag{\textbf{[P.4]}}
\end{align*}
Similarly, $(gm)\up U=g$. Thus, $f=(fm)\up U=(gm)\up U=g$. Therefore, $m$ is monic. Furthermore, by taking $f=\id_U$, we also obtain $m\up U=\id_U$.
\end{proof}

\begin{lemma}
\label{lemma:pullback-partial-categories}
In a partial category, for every morphism $f\colon A\to B$ and $V\leq B$, the diagram
\begin{equation*}
\begin{tikzcd}
{A\up_fV} & V \\
A & B
\arrow["{f\up V}", from=1-1, to=1-2]
\arrow["{m'}"', from=1-1, to=2-1]
\arrow["m", from=1-2, to=2-2]
\arrow["f"', from=2-1, to=2-2]
\end{tikzcd}
\end{equation*}
is a pullback diagram, where $m=\id_B\down V$ and $m'=\id_A\down(A\up_fV)$.
\end{lemma}
\begin{proof}
We first prove that the diagram commutes:
\begin{align*}
(f\up V)m&=~(f\up V)(\id_B\down V)\\
&=~(f\id_B)\down(A\up_fV)                  \Tag{\textbf{[P.8]}}\\
&=~(\id_Af)\down(A\up_fV)\\
&=~(\id_A\down(A\up_fV))f                    \Tag{\textbf{[P.2]}}\\
&=~m'f
\end{align*}
Consider now an object $X$ and two morphisms $\alpha\colon X\to A$ and $\beta\colon X\to V$ such that $\alpha f=\beta m$. Consider also the following diagram
\begin{equation*}
\begin{tikzcd}
{X\up_\alpha(A\up_fV)} && {A\up_fV} \\
X && A
\arrow["{\alpha\up(A\up_fV)}", from=1-1, to=1-3]
\arrow["{m''}"', from=1-1, to=2-1]
\arrow["{m'}", from=1-3, to=2-3]
\arrow["\alpha"', from=2-1, to=2-3]
\end{tikzcd}
\end{equation*}
where $m''=\id_X\down(X\up_\alpha(A\up_fV))$. By using the same calculations as above, this diagram commutes. Moreover, we can compute that:
\begin{align*}
X\up_\alpha(A\up_fV)&=~X\up_{\alpha f}V    \Tag{\textbf{[P.6]}}\\
&=~X\up_{\beta m}V                         \Tag{\alpha f=\beta m}\\
&=~X\up_\beta(V\up_mV)                     \Tag{\textbf{[P.6]}}\\
&=~X\up_\beta V                            \Tag{\text{Lemma}~\ref{lemma:monic-partial-categories}, V\up_mV=V}\\
&=~X                                       \Tag{\textbf{[P.4]}}
\end{align*}
Moreover, we also obtain:
\begin{align*}
(\alpha\up(A\up_fV))(f\up V)&=~(\alpha f)\up V         \Tag{\textbf{[P.6]}}\\
&=~(\beta m)\up V                                      \Tag{\alpha f=\beta m}\\
&=~(\beta\up(V\up_mV))(m\up V)                         \Tag{\textbf{[P.6]}}\\
&=~(\beta\up V)(m\up V)                                \Tag{\text{Lemma}~\ref{lemma:monic-partial-categories}, V\up_mV=V}\\
&=~\beta(m\up V)                                       \Tag{\textbf{[P.4]}}\\
&=~\beta                                               \Tag{\text{Lemma}~\ref{lemma:monic-partial-categories}, m\up V=\id_V}
\end{align*}
Thus, define $\gamma\=\alpha\up(A\up_fV)\colon X\to A\up_fV$. Then we obtain that:
\begin{align*}
&\gamma(f\up V)=(\alpha\up(A\up_fV))(f\up V)=\beta
\end{align*}
Moreover:
\begin{align*}
\gamma m'&=~(\alpha\up(A\up_fV))(\id_A\down(A\up_fV))\\
&=~(\alpha\id_A)\down(X\up_\alpha(A\up_fV))            \Tag{\textbf{[P.8]}}\\
&=~\alpha\down X                                       \Tag{X\up_\alpha(A\up_fV)=X}\\
&=~\alpha                                              \Tag{\textbf{[P.1]}}
\end{align*}
Finally, if $\gamma'\colon X\to A\up_fV$ satisfies the same equations as $\gamma$, thus, $\gamma'm'=\gamma m'$, however, since $m'$ is monic by Lemma~\ref{lemma:monic-partial-categories}, $\gamma'=\gamma$.
\end{proof}

We end this section by introducing the $2$-categorical structure for partial categories, starting with partial functors.

\begin{definition}
\label{definition:partial-functors}
A \textbf{partial functor} from a partial category $\C$ to a partial category $\C'$ consists of a functor $F\colon\C\to\C'$ which preserves the partial order, that is, such that if $A\leq B$ in $\C$, then $FA\leq FB$ in $\C'$, and preserves restriction and contraction, explicitly, for every morphism $f\colon A\to B$ and $U\leq A$ and $V\leq B$:
\begin{align*}
&F(f\down U)=(Ff)\down(FU)           &&F(A\up_fV)=(FA)\up_{Ff}(FV)        &&F(f\up V)=(Ff)\up(FV)
\end{align*}
\end{definition}

Next, we define natural transformations between partial functors. As per restriction categories and local categories, we introduce two distinct flavours of $2$-morphisms.

\begin{definition}
\label{definition:partial-natural-transformation}
A \textbf{partial natural transformation} from a partial functor $F\colon\C\to\C'$ to a partial functor $G\colon\C\to\C'$ consists of a natural transformation $\varphi_A\colon FA\to GA$, natural in $A$. A \textbf{total natural transformation} from $F$ to $G$ consists of a partial natural transformation $\varphi_A\colon FA\to GA$ such that for every pair of objects $U,A$ of $\C$ such that $U\leq A$, the following equations hold:
\begin{align*}
&(FA)\up_{\varphi_A}(GU)=FU     &&\varphi_A\up(GU)=\varphi_U
\end{align*}
\end{definition}

Partial categories, partial functors, and partial natural transformations form a $2$-category denoted by $\PCat_\lax$. The full $2$-subcategory of $\PCat_\lax$ of partial categories, partial functors, and total natural transformations is denoted by $\PCat$.


%% file: CODE/11-inclusion.tex
\section{Inclusion categories: partiality via inclusions}
\label{section:inclusion-systems}
In the previous section, we introduced partial categories as an operational context for controlling partiality. With Lemmas~\ref{lemma:monic-partial-categories} and~\ref{lemma:pullback-partial-categories}, we proved that in every partial category there is a special family of monics, obtained by restricting identities $\id_A\colon A\to A$ to each $U\leq A$, and that the restriction and the contraction operators are entirely determined by those monics. In particular, the restriction of a morphism $f\colon A\to B$ to $U\leq A$ coincides with the composition $mf\colon U\to B$ where $m\=\id_A\down U\colon U\to A$ and the contraction of a morphism $f\colon A\to B$ onto $V\leq B$ coincides with the pullback of $f$ along $m\=\id_B\down V\colon V\to B$. In this section, we isolate the properties of this family of monics by introducing inclusion categories. We then show that inclusion categories are $2$-equivalent to partial categories.

\begin{definition}
\label{definition:inclusion-category}
An \textbf{inclusion system} on a category $\C$ consists of a family $\Incl$ of morphisms of $\C$ satisfying the following conditions:
\begin{description}
\item[I.1] Identities belong to $\Incl$;

\item[I.2] Every morphism of $\Incl$ is monic in $\C$;

\item[I.3] For each pair of parallel monics $m,m'\colon A\to B$, $m\in\Incl$ and $m'\in\Incl$ implies $m=m'$;

\item[I.4] $\Incl$ is stable under composition, that is, for each $m\colon A\to B$ and $m'\colon B\to C$ where $m,m'\in\Incl$, thus, $mm'\in\Incl$;

\item[I.5] $\Incl$ is a display system, that is, for every $m\colon A\to B$ of $\Incl$ and any morphism $f\colon C\to B$, there exists an object $D$ together with two morphisms $g\colon D\to A$ and $m'\colon D\to C$ such that the following diagram
\begin{equation*}
\begin{tikzcd}
D & A \\
C & B
\arrow["g", from=1-1, to=1-2]
\arrow["{m'}"', from=1-1, to=2-1]
\arrow["\lrcorner"{anchor=center, pos=0.125}, draw=none, from=1-1, to=2-2]
\arrow["m", from=1-2, to=2-2]
\arrow["f"', from=2-1, to=2-2]
\end{tikzcd}
\end{equation*}
is a pullback diagram and $m'\colon D\to C$ belongs to $\Incl$;

\item[I.6] For every pair of composable morphisms $m\colon A\to B$ and $i\colon B\to c$, if $mi$ and $i$ belong to the family $\Incl$, so does $m$.\footnote{We thank Tim Stokes for suggesting to include this axiom as part of the definition.}
\end{description}
The monics of an inclusion system are called \textbf{inclusions}. An \textbf{inclusion category} consists of a category equipped with an inclusion system.
\end{definition}

\begin{remark}
\label{remark:comparison-with-M-categories}
\cite[Theorem~3.4]{cockett:restrictionI} establishes a $2$-equivalence between split restriction categories and $\M$-categories. An $\M$-category is a category equipped with a family $\M$ of monics which is stable under composition and pullbacks, and that contains all isomorphisms. The axioms of an inclusion category resemble the ones of an $\M$-category, however, they are not the same concepts. Crucially, \textbf{[I.3]} establishes that two objects $A$ and $B$ might have at most one inclusion $m\colon A\to B$, which in particular, means that the composition of an inclusion and an isomorphism might fail to be again an inclusion.
\end{remark}

\begin{remark}
\label{remark:canonical-pullback-inclusion-categories}
\textbf{[I.5]} establishes that (1) the pullback of a morphism $f\colon A\to B$ along an inclusion $m\colon C\to B$ exists and (2) among these isomorphic pullbacks (remember that pullbacks are only unique up to a unique isomorphism) there is at least one in which the morphism $m'$ is again an inclusion. This does not imply that such a pullback is unique nor that for each of these pullbacks the corresponding map $m'$ is an inclusion. Crucially, if $m$ is an inclusion and $f$ is an isomorphism, the composition $gm$ might fail to be an inclusion (see also Remark~\ref{remark:canonical-pullback-local-cats}).
\end{remark}

Here are some examples of inclusion categories. 

\begin{example}
\label{example:trivial-inclusion}
Every category admits a trivial inclusion system, whose only inclusions are identities.
\end{example}

\begin{example}
\label{example:partial-set-inclusion}
$\ParSet$ (Example~\ref{example:partial-set-local}) is an inclusion category whose inclusions are given by inclusions of sets, that is, $m\colon(A,U)\to(B,V)$ is an inclusion if and only if $A=B$ and $m\colon U\to V$ is the inclusion morphism of $U\subseteq V$.
\end{example}

The next example shows that not every inclusion category is a local category (in a non-trivial way).

\begin{example}
\label{example:set-inclusion}
$\Set$ is an inclusion category whose inclusions are inclusions of sets, that is, $m\colon A\to B$ is an inclusion if and only if $A\subseteq B$ and $m$ is the inclusion morphism.

\end{example}

\begin{example}
\label{example:partial-topological-inclusion}
$\ParTop$ (Example~\ref{example:partial-set-local}) is an inclusion category whose inclusions are defined as in Example~\ref{example:partial-set-inclusion}.
\end{example}

\begin{example}
\label{example:topological-inclusion}
$\Top$ is an inclusion category whose inclusions are defined as in Example~\ref{example:set-inclusion}.
\end{example}

\begin{example}
\label{example:semigroup-inclusion}
$\P[A]$ (Example~\ref{example:semigroup-partial}) is an inclusion category whose inclusions $m\colon a\to b$ are those morphisms $m=a\colon a\to b$, where $ab=a$.
\end{example}

\begin{example}
\label{example:ring-inclusion}
${\IdemRing^\bullet}^\op$ (Example~\ref{example:ring-partial}) is an inclusion category, that is $\IdemRing^\bullet$ is a coinclusion category whose coinclusions $m\colon(S,e_S)\to(R,e_R)$ are those morphisms $m\colon S\to R$ where $R=S$ and $m$ sends each $s\in S=R$ to $se_R$.
\end{example}

Next, we introduce the appropriate $2$-categorical structure for inclusion categories, starting with inclusion functors.

\begin{definition}
\label{definition:inclusion-functor}
An \textbf{inclusion functor} from an inclusion category $\C$ to an inclusion category $\C'$ consists of a functor $F\colon\C\to\C'$ which sends every inclusion $m\in\Incl$ of $\C$ to an inclusion $Fm\in\Incl'$ of $\C'$ and such that, for any morphism $f\colon A\to B$ and inclusion $m\colon V\to B$, it preserves the pullback of $f$ along $m$.
\end{definition}

Next, we define natural transformations between inclusion functors. As per restriction categories, local categories, and partial categories, we introduce two distinct flavours of $2$-morphisms.

\begin{definition}
\label{definition:inclusion-natural-transformation}
An \textbf{inclusion natural transformation} from an inclusion functor $F\colon\C\to\C'$ to an inclusion functor $G\colon\C\to\C'$ consists of a natural transformation $\varphi_A\colon FA\to GA$, natural in $A$. A \textbf{total natural transformation} from $F$ to $G$ consists of an inclusion natural transformation $\varphi_A\colon FA\to GA$ such that for each inclusion $m\colon U\to A$ of $\C$, the naturality square of $m$ is a pullback diagram:
\begin{equation*}
\begin{tikzcd}
FU & GU \\
FA & GA
\arrow["{\varphi_U}", from=1-1, to=1-2]
\arrow["Fm"', from=1-1, to=2-1]
\arrow["\lrcorner"{anchor=center, pos=0.125}, draw=none, from=1-1, to=2-2]
\arrow["Gm", from=1-2, to=2-2]
\arrow["{\varphi_A}"', from=2-1, to=2-2]
\end{tikzcd}
\end{equation*}
\end{definition}

We may now prove that inclusion categories are $2$-equivalent to inclusion categories. To this end, we denote by $\ICat_\lax$ the $2$-category of inclusion categories, inclusion functors, and inclusion natural transformations. We also denote by $\ICat$ the full $2$-subcategory of $\ICat_\lax$ of inclusion categories, inclusion functors, and total natural transformations. For starters, let us show that every partial category carries an inclusion system whose inclusions are the restriction of identities $\id_A\colon A\to A$ to each $U\leq A$.

\begin{proposition}
\label{proposition:partial-categories-have-inclusion-system}
Given a partial category $\C$, the monics $m\colon A\to B$ of Lemma~\ref{lemma:monic-partial-categories} form an inclusion system for $\C$.
\end{proposition}
\begin{proof}
\textbf{[I.1]} and \textbf{[I.2]} are established by Lemma~\ref{lemma:monic-partial-categories}. Given $m\colon A\to B$ and $m'\colon B\to C$, since, $A\leq B$ and $B\leq C$, by transitivity, $A\leq C$, thus, by Lemma~\ref{lemma:monic-partial-categories}, there exists a monic $m''\colon A\to C$, which, by uniqueness, must coincide with $mm'$. \textbf{[I.6]} correponds precisely to \textbf{[P.9]}. Finally, stability under pullbacks is established by Lemma~\ref{lemma:pullback-partial-categories}.
\end{proof}

Next, we show that every inclusion system defines a partial structure.

\begin{proposition}
\label{proposition:inclusion-system-generate-partial-categories}
If a category $\C$ comes equipped with an inclusion system $\Incl$, then $\C$ is a partial category with the following partial structure:
\begin{itemize}
\item $A\leq B$ if there exists an inclusion $m\colon A\to B$;

\item For every $U\leq A$ and $f\colon A\to B$, the restriction of $f$ to $A$ is the morphism $mf\colon U\to B$, where $m\colon U\to A$ is the (necessarily unique) inclusion from $U$ to $A$;

\item For every $V\leq B$ and $f\colon A\to B$, the contraction of $f$ to $V$ is the morphism $f\up V\colon A\up_fV\to V$ defined by the pullback diagram
\begin{equation*}
\begin{tikzcd}
{A\up_fV} & A \\
A & B
\arrow["{f\up V}", from=1-1, to=1-2]
\arrow["{m'}"', from=1-1, to=2-1]
\arrow["\lrcorner"{anchor=center, pos=0.125}, draw=none, from=1-1, to=2-2]
\arrow["m", from=1-2, to=2-2]
\arrow["f"', from=2-1, to=2-2]
\end{tikzcd}
\end{equation*}
of $f$ along $m\colon V\to B$, where $m$ and $m'$ are inclusions.
\end{itemize}
\end{proposition}
\begin{proof}
For starters, since $\Incl$ contains identities, is stable under composition, and by the uniqueness of the monics of $\Incl$, it is not hard to prove that $\leq$ defines a partial order on the objects of $\C$. To prove that $mf$ defines a restriction for $f$, we use that $f\down A=\id_Af=f$, $(fg)\down U=mfg=(f\down U)g$ and finally, that $(f\down U)\down V=nmf=f\down V$, where $m\colon U\to A$ and $n\colon V\to U$. Let us prove that $f\up V$ defines a contraction for $f$. \textbf{[P.4]} and \textbf{[P.4']} follow from the fact that the pullback of a morphism $f$ along the identity is still $f$. Consider the following diagram:
\begin{equation*}
\begin{tikzcd}[column sep=huge]
{(A\up_fV)\up_{f\up V}W} & W \\
{A\up_fV} & V \\
A & B
\arrow["{(f\up V)\up W}", from=1-1, to=1-2]
\arrow[from=1-1, to=2-1]
\arrow["n", from=1-2, to=2-2]
\arrow[""{name=0, anchor=center, inner sep=0}, "{f\up V}", from=2-1, to=2-2]
\arrow[from=2-1, to=3-1]
\arrow["m", from=2-2, to=3-2]
\arrow[""{name=1, anchor=center, inner sep=0}, "f"', from=3-1, to=3-2]
\arrow["\lrcorner"{anchor=center, pos=0.125}, draw=none, from=1-1, to=0]
\arrow["\lrcorner"{anchor=center, pos=0.125}, draw=none, from=2-1, to=1]
\end{tikzcd}
\end{equation*}
Thus, \textbf{[P.5]} follows from the fact that vertical composition of pullback diagrams is still a pullback diagram. Next, consider the following diagram:
\begin{equation*}
\begin{tikzcd}[column sep=huge]
{A\up_f(B\up_gW)} & {B\up_gW} & W \\
A & B & C
\arrow[from=1-1, to=1-2]
\arrow[from=1-1, to=2-1]
\arrow["{g\up W}", from=1-2, to=1-3]
\arrow[from=1-2, to=2-2]
\arrow["m", from=1-3, to=2-3]
\arrow[""{name=0, anchor=center, inner sep=0}, "f"', from=2-1, to=2-2]
\arrow[""{name=1, anchor=center, inner sep=0}, "g"', from=2-2, to=2-3]
\arrow["\lrcorner"{anchor=center, pos=0.125}, draw=none, from=1-1, to=0]
\arrow["\lrcorner"{anchor=center, pos=0.125}, draw=none, from=1-2, to=1]
\end{tikzcd}
\end{equation*}
Thus, \textbf{[P.6]} follows from the fact that horizontal composition of pullback diagrams is still a pullback diagram. To prove \textbf{[P.7]}, consider the following diagram:
\begin{equation*}
\begin{tikzcd}
{A\up_fV} & V \\
A & B
\arrow["{f\up V}", from=1-1, to=1-2]
\arrow["{m'}"', from=1-1, to=2-1]
\arrow["\lrcorner"{anchor=center, pos=0.125}, draw=none, from=1-1, to=2-2]
\arrow["m", from=1-2, to=2-2]
\arrow["f"', from=2-1, to=2-2]
\end{tikzcd}
\end{equation*}
Thus, $f\down(A\up_fV)=m'f=(f\up V)m$. Let us take the pullback diagram of $f\down(A\up_fV)$ along $m$ and use that $m$ is monic:
\begin{equation*}
\begin{tikzcd}
{A\up_fV} & V & V \\
{A\up_fV} & V & B
\arrow["{f\up V}", from=1-1, to=1-2]
\arrow[equals, from=1-1, to=2-1]
\arrow["\lrcorner"{anchor=center, pos=0.125}, draw=none, from=1-1, to=2-2]
\arrow[equals, from=1-2, to=1-3]
\arrow[equals, from=1-2, to=2-2]
\arrow["\lrcorner"{anchor=center, pos=0.125}, draw=none, from=1-2, to=2-3]
\arrow["m", from=1-3, to=2-3]
\arrow["{f\up V}"', from=2-1, to=2-2]
\arrow["{f\down(A\up_fV)}"', curve={height=18pt}, from=2-1, to=2-3]
\arrow["m"', from=2-2, to=2-3]
\end{tikzcd}
\end{equation*}
However, by definition, the pullback of $f\down(A\up_fV)$ along $m$ is $(f\down(A\up_fV))\up V$. This proves \textbf{[P.7]}. Let us prove \textbf{[P.8]}. Consider the following diagram:
\begin{equation*}
\begin{tikzcd}
{A\up_fV} & V \\
A & B & C
\arrow["{f\up V}", from=1-1, to=1-2]
\arrow["{m''}"', from=1-1, to=2-1]
\arrow["\lrcorner"{anchor=center, pos=0.125}, draw=none, from=1-1, to=2-2]
\arrow["{m'}"', from=1-2, to=2-2]
\arrow["{g\down V}", from=1-2, to=2-3]
\arrow["f"', from=2-1, to=2-2]
\arrow["g"', from=2-2, to=2-3]
\end{tikzcd}
\end{equation*}
However, by definition, $(fg)\down(A\up_fV)=m''fg=(f\up V)(g\down V)$. Finally, \textbf{[P.9]} corresponds precisely to \textbf{[I.6]}.
\end{proof}

Consider now a partial functor $F\colon\C\to\C'$. Since $F$ preserves the restriction operator, it sends each $\id_A\down U\colon U\to A$ to $F(\id_A\down U)=(F\id_A)\down(FU)=\id_{FA}\down(FU)\colon FU\to FA$. Furthermore, since $F$ preserves the contraction operator, it also preserves the pullbacks of inclusions along every morphism, since the following is a pullback by Lemma~\ref{lemma:pullback-partial-categories}:
\begin{equation*}
\begin{tikzcd}
{(FA)\up_{Ff}(FV)} & FV \\
FA & FB
\arrow["{(Ff)\up(FV)}", from=1-1, to=1-2]
\arrow["{\id_{FA}\down((FA)\up_{Ff}(FV))}"', from=1-1, to=2-1]
\arrow["\lrcorner"{anchor=center, pos=0.125}, draw=none, from=1-1, to=2-2]
\arrow["{\id_{FB}\down(FV)}", from=1-2, to=2-2]
\arrow["Ff"', from=2-1, to=2-2]
\end{tikzcd}
\end{equation*}
Conversely, consider an inclusion functor $F\colon\C\to\C'$. Since $F$ preserves inclusions, it also preserves the contraction operator defined by pre-composing with inclusions. Furthermore, since $F$ preserves pullbacks along inclusions, it also preserves the associated contraction operator. Thus, $F$ becomes an inclusion functor.
\par Both partial natural transformations and inclusion natural transformations are simply natural transformations, thus, the correspondence between partial categories and inclusion categories becomes $2$-functorial. In particular, there are two $2$-functors:
\begin{align*}
&\P_\lax\colon\ICat_\lax\to\PCat_\lax            &&\I_\lax\colon\PCat_\lax\to\ICat_\lax
\end{align*}
Furthermore, a total natural transformation $\varphi_A\colon FA\to GA$ of partial functors satisfy the following property. Since $FA\up_{\varphi_A}GU=FU$ and $\varphi_A\up GU=\varphi_U$, by Lemma~\ref{lemma:pullback-partial-categories}, the following diagram is a pullback diagram:
\begin{equation*}
\begin{tikzcd}
FU & GU \\
FA & GA
\arrow["{\varphi_U}", from=1-1, to=1-2]
\arrow["{F(\id_A\down U)}"', from=1-1, to=2-1]
\arrow["{G(\id_A\down U)}", from=1-2, to=2-2]
\arrow["{\varphi_A}"', from=2-1, to=2-2]
\end{tikzcd}=
\begin{tikzcd}
{FA\up_{\varphi_A}GU} & GU \\
FA & GA
\arrow["{\varphi_A\up GU}", from=1-1, to=1-2]
\arrow["{\id_{FA}\down FU}"', from=1-1, to=2-1]
\arrow["\lrcorner"{anchor=center, pos=0.125}, draw=none, from=1-1, to=2-2]
\arrow["{\id_{GA}\down GU}", from=1-2, to=2-2]
\arrow["{\varphi_A}"', from=2-1, to=2-2]
\end{tikzcd}
\end{equation*}
Thus, $\varphi_A\colon FA\to GA$ becomes a total natural transformation of inclusion functors. Conversely, if $\varphi_A\colon FA\to GA$ is a total natural transformation of inclusion functors, since the naturality square of each inclusion $m\colon U\to A$ along $\varphi_A$ is a pullback, we conclude that $FA\up_{\varphi_A}GU=FU$ and that $\varphi_A\up GU=\varphi_U$. Thus, $\varphi_A$ becomes a total natural transformation of partial functors. Thus, the two $2$-functors $\P_\lax$ and $\I_\lax$ restrict to two $2$-functors:
\begin{align*}
&\P\colon\ICat\to\PCat            &&\I\colon\PCat\to\ICat
\end{align*}

\begin{theorem}
\label{theorem:equivalence-partial-categories-inclusion-systems}
The $2$-functors $\P_\lax$ and $\I_\lax$ form a $2$-equivalence $\PCat_\lax\simeq\ICat_\lax$. Moreover, this $2$-equivalence restricts to another $2$-equivalence $\PCat\simeq\ICat$. 
\end{theorem}
\begin{proof}
Consider a partial category $\C$. Thus, $\P[\I[\C]]$ is the same category as $\C$ equipped with the following partial structure. $U\leq A$ in $\P[\I[\C]]$ if and only if there exists an inclusion $m\colon U\to A$ in $\I[\C]$, if and only if $U\leq A$ in $\C$. The restriction of $f\colon A\to B$ is the morphism $mf\colon U\to B$ where $m\colon U\to A$ is the inclusion of $U$ into $A$ in $\I[\C]$, that is, the restriction $\id_A\down U$ of $\id_A$ to $U$ in $\C$, thus:
\begin{align*}
mf&=~(\id_A\down U)f\\
&=~(\id_Af)\down U           \Tag{\textbf{[P.2]}}\\
&=~f\down U
\end{align*}
The contraction of $f\colon A\to B$ onto $V\leq B$ in $\P[\I[\C]]$ is the pullback of $f$ along the inclusion $m=\id_B\down V\colon V\to B$. Thus, by Lemma~\ref{lemma:pullback-partial-categories}, it coincides with the contraction of $f$ onto $V$ in $\C$ (up to a unique isomorphism). 

Conversely, consider an inclusion category $\C$. Thus, $\I[\P[\C]]$ is the same category as $\C$. Moreover, $m\colon U\to A$ is an inclusion in $\C$ if and only if $m=\id_A\down U$ in $\P[\C]$, however, the restriction of a morphism $f\colon A\to B$ in $\C$ is given by pre-composing $f$ by the inclusion $m\colon U\to A$, which is unique by \textbf{[I.3]}. Thus, $m$ is an inclusion in $\I[\P[\C]]$ if and only if it is an inclusion in $\C$. We leave it to the reader to prove that this correspondence defines a full $2$-equivalence.
\end{proof}

\begin{remark}
\label{remark:examples-partial-to-inclusion}
Under the $2$-equivalence of Theorem~\ref{theorem:equivalence-partial-categories-inclusion-systems}, Examples~\ref{example:trivial-partial}~-~\ref{example:topological-partial}, and~\ref{example:ring-partial}, of partial categories correspond to the Examples~\ref{example:trivial-inclusion}~-~\ref{example:topological-inclusion}, and~\ref{example:ring-inclusion}, respectively.
\end{remark}


%% file: CODE/12-fundamental-theorem.tex
\section{The fundamental theorem of partiality}
\label{section:fundamental-theorem}
Every local category $\C$ carries a canonical inclusion system, whose inclusions are those monics $m\colon U\to A$ between two objects $U\leq A$ of $\C$ such that $m\eta_A=\eta_U$. In general, an inclusion category does not carry a local structure (see Example~\ref{example:set-inclusion}). In this section, we isolate the exact condition which makes inclusion categories and partial categories $2$-equivalent to restriction and local categories: boundedness. We then use this characterization to prove the main result of this paper: restriction categories, local categories, bounded partial categories, and bounded inclusion categories are all $2$-equivalent to each other. As an application, we revisit the definition of an inverse local category and connect this to the Ehresmann-Schein-Nambooripad theorem for inverse categories. We begin by defining bounded partial categories and bounded inclusion categories.

\begin{definition}
\label{definition:bounded-partial-category}
A partial category $\C$ is \textbf{bounded} if for every object $A$ of $\C$ there exists a unique object $\Las A$ such that $A\leq\Las A$ and $\Las A$ is maximal in the sense that $\Las A\leq B$ implies $B=\Las A$.
\end{definition}

\begin{definition}
\label{definition:bounded-inclusion-system}
An inclusion category $\C$ is \textbf{bounded} when for every object $A$ of $\C$, there exists a unique inclusion $\eta_A\colon A\to\Las A$, called the \textbf{maximal inclusion} of $A$, such that if there exists an inclusion $x\colon \Las A\to B$, then $\Las A=B$.
\end{definition}

We now show that the $2$-equivalence of Theorem~\ref{theorem:equivalence-partial-categories-inclusion-systems} restricts to a $2$-equivalence between bounded partial categories and bounded inclusion categories. To this end, we denote by $\BPCat_\lax$ ($\BPCat$) the $2$-subcategory of $\PCat_\lax$ ($\PCat$) of bounded partial categories, partial functors which preserve the boundedness, that is, $F\Las A=\Las FA$, and partial (total) natural transformations. Similarly, we denote by $\BICat_\lax$ ($\BICat$) the $2$-subcategory of $\ICat_\lax$ ($\ICat$) of bounded inclusion categories, inclusion functors which preserve the boundedness, that is, $F\eta_A=\eta_{FA}$, and inclusion (total) natural transformations.

\begin{theorem}
\label{theorem:equivalence-bounded}
The $2$-equivalences of Theorem~\ref{theorem:equivalence-partial-categories-inclusion-systems} restrict to a $2$-equivalence $\BPCat_\lax\simeq\BICat_\lax$ and a $2$-equivalence $\BPCat\simeq\BICat$. 
\end{theorem}
\begin{proof}
Consider a bounded partial category $\C$. Thus, for every object there exists a unique maximal object $\Las A$ which bounds $A$ from above, that is, $A\leq\Las A$. Let $\eta_A$ be the contraction $\id_{\Las A}\down A\colon A\to\Las A$ of the identity on $\Las A$ to $A$. Suppose that there exists an inclusion $\Las A\to B$ in $\I[\C]$. In particular, $\Las A\leq B$. Thus, $B=\Las A$. Now, suppose that $m'\colon A\to\Las'A$ is another maximal inclusion for $A$ in $\I[\C]$. Thus, $A\leq\Las'A$. Consider $B$ such that $\Las'A\leq B$, thus, there exists an inclusion $\Las'A\to B$, thus, since $m'$ is also maximal, $B=\Las'A$. Therefore, $\Las'A$ is another maximal object which bounds $A$ from above. Thus, by the uniqueness of $\Las A$, $\Las'A=\Las A$. This, together with \textbf{[I.3]}, proves that $\eta_A$ is the unique maximal inclusion of $\I[\C]$ and therefore that $\I[\C]$ is bounded.
\par Conversely, consider a bounded inclusion category $\C$. Thus, every object $A$ of $\C$ admits an object $\Las A$ and a unique maximal inclusion $\eta_A\colon A\to\Las A$. Thus, $A\leq\Las A$ in $\P[\C]$. Suppose that $\Las A\leq B$. Thus, there exists an inclusion $x\colon\Las A\to B$. Thus, by maximality of $\eta_A$, $B=\Las A$. Suppose now that $A\leq\Las'A$ and for every $B$, if $\Las'A\leq B$, thus $B=\Las'A$. Thus, there is an inclusion $m'\colon A\to\Las'A$, which is maximal. By uniqueness of $\eta_A$, $m'=\eta_A$, thus, $\Las'A=\Las A$. Therefore, $\P[\C]$ is bounded. 

We leave it to the reader to show that boundeness-preserving partial functors correspond to boundedness-preserving inclusion functors.
\end{proof}

The next step is to compare bounded inclusion categories with local categories. We first show that every local category carries a canonical inclusion system.

\begin{proposition}
\label{proposition:local-categories-inclusion-categories}
Let $\C$ be a local category. Define $\Incl$ as the family of morphisms $m\colon A\to B$ of $\C$ where $A\smile B$ and subject to the following equation:
\begin{align*}
&m\eta_B=\eta_A
\end{align*}
Then $\Incl$ is a bounded inclusion system for $\C$.
\end{proposition}
\begin{proof}
First of all, identities are part of $\Incl$, since trivially, $\id_A\eta_A=\eta_A$. Secondly, since $m\eta_B=\eta_A$ is monic, $m$ is also monic. Moreover, if $m'\colon A\to B$ satisfies the equation $m'\eta_B=\eta_A=m\eta_B$, since $\eta_B$ is monic, $m=m'$. Consider now, $m\colon A\to B$ and $m'\colon B\to C$ such that $m\eta_B=\eta_A$ and $m'\eta_C=\eta_B$. Thus, $mm'\eta_C=m\eta_B=\eta_A$. Thus, $mm'\in\Incl$. Furthermore, consider a morphism $f\colon C\to B$. Consider the following diagram:
\begin{equation*}
\begin{tikzcd}
D && A \\
C & B & {\Las B=\Las A}
\arrow["g", from=1-1, to=1-3]
\arrow["{m'}"', from=1-1, to=2-1]
\arrow["\lrcorner"{anchor=center, pos=0.125}, draw=none, from=1-1, to=2-2]
\arrow["{\eta_A}", from=1-3, to=2-3]
\arrow["f"', from=2-1, to=2-2]
\arrow["{\eta_B}"', from=2-2, to=2-3]
\end{tikzcd}
\end{equation*}
Using that $m\eta_B=\eta_A$, we can rewrite this diagram as follows:
\begin{equation*}
\begin{tikzcd}
D & A \\
C & B \\
&& {\Las B}
\arrow["g", from=1-1, to=1-2]
\arrow["{m'}"', from=1-1, to=2-1]
\arrow["m", from=1-2, to=2-2]
\arrow["{\eta_A}", curve={height=-12pt}, from=1-2, to=3-3]
\arrow["f"', from=2-1, to=2-2]
\arrow["{f\eta_B}"', curve={height=12pt}, from=2-1, to=3-3]
\arrow["{\eta_B}", from=2-2, to=3-3]
\end{tikzcd}
\end{equation*}
Since the outer diagram is a pullback, it is not hard to see that the internal square diagram is a pullback as well. To prove that $\Incl$ is bounded, notice that for each $A\in\C$, $A\leq\Las A$, since $\eta_A\colon A\to\Las A$. Moreover, if $\Las A\leq B$, then $\Las B=\Las\Las A=\Las A$. Now, suppose that $\Las' A$ is another object such that $A\leq\Las'A$. Thus, $\Las A=\Las\Las'A$. Thus, $\Las'A\leq\Las\Las'A=\Las A$. Thus, $\Las'A=\Las A$. Finally, if $mm'$ and $m'$ belong to $\Incl$, then $m\eta=mm'\eta=\eta$. Therefore, $m$ also belongs to $\Incl$.
\end{proof}

Conversely, every bounded inclusion category comes with a local structure.

\begin{proposition}
\label{proposition:inclusion-categories-local-categories}
Let $\C$ be a bounded inclusion category. Then, $\C$ is a local category where for every object $A$, $\eta_A\colon A\to\Las A$ is the maximal inclusion of $A$.
\end{proposition}
\begin{proof}
Let us start by proving \textbf{[L.1]}. By definition, $\eta_A\eta_{\Las A}\colon A\to\Las\Las A$ factors through $\eta_A$, thus, since $\eta_A$ is maximal, $\Las\Las A=\Las A$. However, since identities are inclusions, by \textbf{[I.3]}, $\eta_{\Las A}=\id_{\Las A}$. Since $\eta_A$ is an inclusion, $\eta_A$ is monic, thus, also \textbf{[L.2]} holds. To prove \textbf{[L.3]}, consider a morphism $f\colon A\to\Las B$. Since $\eta_B$ is an inclusion and since inclusions are stable under pullbacks, the pullback of $\eta_B$ along $f$ exists. 
\begin{equation*}
\begin{tikzcd}
C & B \\
A & {\Las B}
\arrow[from=1-1, to=1-2]
\arrow["m"', from=1-1, to=2-1]
\arrow["\lrcorner"{anchor=center, pos=0.125}, draw=none, from=1-1, to=2-2]
\arrow["{\eta_B}", from=1-2, to=2-2]
\arrow["f"', from=2-1, to=2-2]
\end{tikzcd}
\end{equation*}
 Moreover, $m\colon C\to A$ is an inclusion. Consider $m\eta_A\colon C\to\Las A$. For any $B$ for which there is an inclusion $\Las A\to B$, $B=\Las A$, thus, $m\eta_A$ is maximal. However, by the uniqueness of maximal inclusions, we conclude that $\eta_C=m\eta_A$ and that $\Las C=\Las A$, proving \textbf{[L.3]}.
\end{proof}

Thanks to Propositions~\ref{proposition:local-categories-inclusion-categories} and~\ref{proposition:inclusion-categories-local-categories}, we can now prove a $2$-equivalence between local categories and bounded inclusion systems.

\begin{theorem}
\label{theorem:equivalence-local-bounded-inclusion}
There is a $2$-equivalence $\LCat_\lax\simeq\BICat_\lax$. Furthermore, this $2$-equivalence restricts to a $2$-equivalence $\LCat\simeq\BICat$. 
\end{theorem}
\begin{proof}
Propositions~\ref{proposition:local-categories-inclusion-categories} and~\ref{proposition:inclusion-categories-local-categories} already proved a correspondence between local categories and bounded inclusion categories. It is immediate to see that these two constructions invert each other, since the $\eta_A$ of a local category are maximal inclusions in the associated inclusion system. We leave it to the reader to prove that this correspondence lifts to a $2$-equivalence.
\end{proof}

\begin{remark}
\label{remark:boundedness-vs-equivalences-of-categories}
$\Set$ and $\ParSet$ (Example~\ref{example:partial-set-local}) are equivalent as categories. However, while the latter carries a local structure, the former does not\footnote{We thank Steve Lack for pointing this out to us.}. As pointed out in Examples~\ref{example:set-partial},~\ref{example:partial-set-partial},~\ref{example:set-inclusion}, and~\ref{example:partial-set-inclusion}, both $\Set$ and $\ParSet$ carry a partial structure and an inclusion system and $\Set$ and $\ParSet$ are equivalent both as partial categories and as inclusion categories. In light of Theorem~\ref{theorem:equivalence-local-bounded-inclusion}, since the local structure of $\ParSet$ is not carried over by the equivalence to $\Set$, we conclude that boundedness is not stable under equivalence of categories.
\end{remark}

\begin{remark}
\label{remark:comparison-with-M-categories-2}
As pointed out in Remark~\ref{remark:comparison-with-M-categories}, inclusion categories resemble $\M$-categories. Theorems~\ref{theorem:fundamental} and~\ref{theorem:2-equivalence-local-restriction-split} clarify exactly the relationship between these two concepts. By~\cite[Theorem~3.4]{cockett:restrictionI}, $\M$-categories are $2$-equivalent to split restriction categories. Thus, there is a $2$-equivalence between $\M$-categories and split bounded inclusion categories, that is, bounded inclusion categories whose associated local structure splits. However, the family of inclusions $\Incl$ of the inclusion category associated to an $\M$-category $\X$ does not lie in the same category of $\X$. For instance, consider the restriction category $\ParFnc$ of Example~\ref{example:partial-set-restriction}. $\ParFnc$ splits, thus it corresponds to the $\M$-category whose underlying category is still $\ParFnc$ and whose system of monics is given by the inclusion morphisms $U\hookrightarrow A$. The inclusion category that corresponds to the restriction category $\ParFnc$ is the category $\ParSet$ of Example~\ref{example:partial-set-inclusion}. Thus, the system of monics $\M$ and inclusion system $\Incl$ corresponding to a split restriction category do not live in the same category\footnote{We thank Steve Lack and Richard Garner for helping clarify this point.}. 
\end{remark}

\begin{corollary}
\label{corollary:axiom-I-6-is-automatic-when-bounded}
For a bounded inclusion category, assuming the other axioms, axiom {\normalfont\textbf{[I.6]}} holds for free. Similarly, for a bounded partial category, assuming the other axioms, {\normalfont\textbf{[P.9]}} holds for free.
\end{corollary}
\begin{proof}
Let us suppose that $\C$ is a bounded inclusion category. Via the $2$-equivalence of Theorem~\ref{theorem:equivalence-local-bounded-inclusion}, $\C$ corresponds to a local category. In particular, this means that a morphism $m\colon A\to B$ of $\C$ is an inclusion if and only if $\Las A=\Las B$ and $m$ satisfies the following equation. $m\eta_B=\eta_A$. Now, consider and two morphisms $m\colon A\to B$ and $i\colon B\to C$ such that, $mi$ and $i$ are inclusions. This means that $\Las A=\Las C=\Las B$ and that $i\eta_C=\eta_B$ and $mi\eta_C=\eta_A$. This implies that $m\eta_B=mi\eta_C=\eta_A$, that is, $m$ is an inclusion. This is precisely \textbf{[I.6]}. Finally, via the equivalence of Theorem~\ref{theorem:equivalence-bounded} between bounded inclusion and bounded partial categories, axiom \textbf{[I.6]} corresponds precisely to axiom \textbf{[P.9]}.
\end{proof}

In this paper, we have introduced three new categorical approaches to reason around partiality: local categories, in which partiality is on objects instead of being on morphisms as per restriction categories; (bounded) partial categories, in which partiality is encoded operationally, via the restriction and the contraction operators; (bounded) inclusion categories, in which partiality is described by a class of monics, which generalize inclusions of sets. We may now state the main result of this paper: that all these four approaches for describing partiality, restriction categories, local categories, bounded partial categories, and bounded inclusion categories are equivalent, allowing one to move from one approach to another one, as desired.

\begin{theorem}
\label{theorem:fundamental}
There is a chain of $2$-equivalences
\begin{align*}
&\RCat_\lax\simeq\LCat_\lax\simeq\BICat_\lax\simeq\BPCat_\lax
\end{align*}
Moreover, these $2$-equivalences restrict to total natural transformations, providing another chain of $2$-equivalences:
\begin{align*}
&\RCat\simeq\LCat\simeq\BICat\simeq\BPCat
\end{align*}
\end{theorem}
\begin{proof}
Theorem~\ref{theorem:2-equivalence-local-restriction} already proves that $\RCat_\lax$ ($\RCat$) is $2$-equivalent to $\LCat_\lax$ ($\LCat$). Moreover, Theorem~\ref{theorem:equivalence-bounded} proves that $\BPCat_\lax$ ($\PCat_\lax$) is $2$-equivalent to $\BICat_\lax$ ($\BICat$). Finally, Theorem~\ref{theorem:equivalence-local-bounded-inclusion} proves that local categories are $2$-equivalent to bounded inclusion categories.
\end{proof}

We conclude this paper with a discussion on how our main result connects to the celebrated Ehresmann-Schein-Nambooripad (ESN) theorem~\cite[Theorem~4.1.8]{lawson:inverse-semigroups}\footnote{We thank Richard Garner for suggesting this direction to us.}. Briefly recall that the ESN theorem establishes an equivalence between the categories of inverse semigroups and inductive groupoids, which are groupoids equipped with a partial order and two operations, a restriction and a corestriction, satisfying some properties. We suggest the reader consult~\cite{hollings:Ehresmann-Schein-Nambooripad-theorem} for details on this result. An inverse category can be seen as a multi-object categorical version of an inverse semigroup. In~\cite{dewolf:ESN-theorem-inverse-cats}, the ESN theorem was generalized to inverse categories by proving that inverse categories are equivalent to \textit{top-heavy locally inductive groupoids}. We show how our construction of the bounded inclusion category associated to an inverse category recaptures the construction of DeWolf and Pronk. 

Thanks to the equivalence provided by Theorem~\ref{theorem:equivalence-local-bounded-inclusion}, we can reformulate the definition of an inverse local category in terms of inclusion categories as follows.

\begin{definition}
\label{definition:inverse-inclusion-category}
An \textbf{inverse inclusion category} consists of an inclusion category $\C$ in which every morphism factors as an isomorphism followed by an inclusion. An inverse inclusion category is \textbf{bounded} when it is bounded as an inclusion category.
\end{definition}

Let $\iBICat_\lax$ and $\iBICat$ be the two full $2$-subcategories of $\BICat_\lax$ and $\BICat$, respectively, spanned by inverse inclusion categories.

\begin{theorem}
\label{theorem:inverse-inclusion-categories}
The $2$-equivalences of Theorem~\ref{theorem:equivalence-local-bounded-inclusion} restrict to two $2$-equivalences\linebreak $\iLCat_\lax\simeq\iBICat_\lax$ and $\iLCat\simeq\iBICat$.  
\end{theorem}

\begin{remark}
\label{remark:inverse-inclusion-categories}
Notice that inverse inclusion categories are more general than inverse local categories, and consequently, to inverse (restriction) categories, since they are not, in general, bounded.
\end{remark}

By putting together the equivalence between inverse (restriction) categories and inverse local categories of Theorem~\ref{theorem:inverse-local-categories} and the equivalence of Theorem~\ref{theorem:inverse-inclusion-categories}, we obtain the following. 

\begin{theorem}
\label{theorem:ESN-revisited}
There are two $2$-equivalences $\iRCat_\lax\simeq\iBICat_\lax$ and $\iRCat\simeq\iBICat$. 
\end{theorem}

In particular, given an inverse (restriction) category $\X$, the morphisms of the associated bounded inverse inclusion category $\I[\X]$ factor as isomorphisms followed by inclusions. The maximal subgroupoid of $\I[\X]$, that is, the groupoid of all isomorphisms of $\I[\X]$, coincides with the top-heavy locally inductive groupoid of $\X$ constructed by DeWolf and Pronk. In particular, the order is induced by the partial order defined by inclusions in $\I[\X]$.

\begin{theorem}
\label{theorem:ESN-revisited-2}
The maximal subgroupoid of the bounded inverse inclusion category $\I[\X]$ associated to an inverse (restriction) category $\X$ equipped with the partial order induced by $\I[\X]$ is precisely the top-heavy locally inductive groupoid of $\X$ of~\cite[Theorem~2.17]{dewolf:ESN-theorem-inverse-cats}.
\end{theorem}